\documentclass[12pt]{amsart}
\usepackage{amssymb}
\usepackage{amsmath, graphicx, rotating}
\usepackage{color}
\usepackage{xcolor}
\usepackage{soul}
\usepackage{todonotes}

\colorlet{corrcol}{red}   

\newcommand{\HL}[1]{#1}   

\definecolor{rouge}{rgb}{0.7,0.00,0.00}
\definecolor{vert}{rgb}{0.00,0.5,0.00}
\definecolor{bleu}{rgb}{0.00,0.00,0.8}

\newtheorem{theorem}{Theorem}[section]
\newtheorem{lemma}[theorem]{Lemma}

\newtheorem{corollary}[theorem]{Corollary}
\newtheorem{remark}[theorem]{Remark}
\newtheorem{proposition}[theorem]{Proposition}

\numberwithin{equation}{section}

\begin{document}
\title[Products of random matrices]{Conditioned limit theorems for products \\of random matrices}
\author{ Ion~Grama}
\curraddr[Grama, I.]{ Universit\'{e} de Bretagne-Sud, LMBA UMR CNRS 6205, Vannes, France.}
\email{ion.grama@univ-ubs.fr}
\author{\'Emile Le Page}
\curraddr[Le Page, E.]{Universit\'{e} de Bretagne-Sud, LMBA UMR CNRS 6205, Vannes, France.}
\email{emile.lepage@univ-ubs.fr}
\author{Marc~Peign\'{e}}
\curraddr[Peign\'{e}, M.]{Universit\'{e} F. Rabelais, LMPT UMR CNRS 7350, Tours, France.}
\email{peigne@lmpt.univ-tours.fr}
\date{\today}
\footnotetext[1]{ This is the full version of the paper }

\begin{abstract}
\HL{Let $g_{1},g_{2},\dots$ be i.i.d.~random matrices in $GL\left( d,\mathbb{R}\right).$ 
For any $n\geq 1$ consider the product $G_{n}=g_{n} \dots g_{1}$} and
the random process $G_{n}v=g_{n}\dots g_{1}v$ 
in $\mathbb{R}^{d}$ starting at point $v\in \mathbb{R}^{d}\smallsetminus \left\{ 0\right\} .$ 
It is well known that under appropriate assumptions, the sequence $\left( \log \left\Vert G_{n}v\right\Vert \right) _{n\geq 1}$ 
behaves like a sum of i.i.d.~r.v.'s
and satisfies standard classical properties such as the law of large
numbers, the law of iterated logarithm and the central limit theorem. 
\HL{For any vector $v$ 
with $\left\Vert v \right\Vert >1$ 
denote by $\tau_v$ the first time when the
random process $G_{n}v$ enters 
the closed unit ball in $\mathbb{R}^{d}.$}
We establish the asymptotic as $n\rightarrow +\infty $ of the probability of the event $\left\{ \tau
_{v}>n\right\} $ and find the limit law for the quantity $\frac{1}{\sqrt{n}} \log \left\Vert G_{n}v\right\Vert $ 
conditioned that $\tau _{v}>n.$ 
\end{abstract}

\maketitle

\section{Introduction  \label{Intro} }

Let $\mathbb{G}=GL\left( d,\mathbb{R}\right) $ be the general linear group
of $d\times d$ invertible matrices with respect to ordinary matrix multiplication.
The Euclidean norm in $\mathbb{V}=\mathbb{R}^{d}$ is denoted by $\left\Vert
v\right\Vert =\sqrt{\sum_{i=1}^{d}v_{i}^{2}},$ for $v\in \mathbb{V}.$ Denote
by $\left\Vert g\right\Vert =\sup_{v\in \mathbb{V}\smallsetminus \left\{
0\right\} }\frac{\left\Vert gv\right\Vert }{\left\Vert v\right\Vert }$ the
operator norm of an element $g$ of $\mathbb{G}$ and endow the group $\mathbb{%
G}$ with the usual Borel $\sigma $-algebra with respect to $\left\Vert \cdot
\right\Vert .$ Suppose that on the probability space $\left( \Omega ,%
\mathcal{F},\mathbf{Pr}\right) $ we are given an i.i.d.~sequence $\left(
g_{n}\right) _{n\geq 1}$ of $\mathbb{G}$-valued random elements of the same
law $\mathbf{Pr}\left( g_{1}\in dg\right) =\boldsymbol{\mu }\left( dg\right)
,$ where $\boldsymbol{\mu }$ is a probability measure on $\mathbb{G}.$
Consider the product $G_{n}=g_{n}\dots g_{1}$ of the random matrices $%
g_{1},\dots,g_{n}$ and the random process $G_{n}v=g_{n}\dots g_{1}v$ in $\mathbb{V%
}$ starting at point $v\in \mathbb{V}\smallsetminus \left\{ 0\right\} .$ The
object of interest is the size of the vector $G_{n}v$ which is controlled by
the quantity $\log \left\Vert G_{n}v\right\Vert .$ It follows from the
results of Le Page \cite{LePage82} that, under appropriate assumptions, the
sequence $\left( \log \left\Vert G_{n}v\right\Vert \right) _{n\geq 1}$
behaves like a sum of i.i.d.~r.v.'s and satisfies standard classical
properties such as the law of large numbers, law of iterated logarithm and
the central limit theorem. Further results and a discussion of the
assumptions under which these results hold true can be found in Furstenberg
and Kesten \cite{FurstKest}, Bougerol and Lacroix \cite{Boug-Lacr85},
Guivarc'h and Raugi \cite{GuivRag}, Benoist and Quint \cite{BenoisQuint2016}, 
Hennion \cite{Hennion}, Jan \cite{Jan}, Goldsheid and Margulis \cite{Gold}.

Denote by $\mathbb{B}$ the closed unit ball in $\mathbb{V}$ and by $\mathbb{B%
}^{c}$ its complement. For any $v\in \mathbb{B}^{c}$ define the exit time of
the random process $G_{n}v$ from $\mathbb{B}^{c}$ by%
\begin{equation*}
\tau _{v}=\min \left\{ n\geq 1:G_{n}v\in \mathbb{B}\right\} .
\end{equation*}%
The goal of this paper is to establish the asymptotic as $n\rightarrow +\infty $ 
of the probability of the event $\left\{ \tau _{v}>n\right\}
=\left\{ G_{1}v\in \mathbb{B}^{c},\dots,G_{n}v\in \mathbb{B}^{c}\right\} $ and
find the limit law for the quantity $\frac{1}{\sqrt{n}}\log \left\Vert
G_{n}v\right\Vert $ conditioned that $\tau _{v}>n.$

The study of related problems for random walks in $\mathbb{R}$ has attracted
much attention. We refer the reader to Spitzer \cite{Spitzer}, Iglehart \cite%
{Igle74}, Bolthausen \cite{Bolth}, Bertoin and Doney \cite{BertDoney94},
Doney \cite{Doney85}, Borovkov \cite{Borovkov04a}, \cite{Borovkov04b},
Vatutin and Wachtel \cite{VatWacht09}, Caravenna \cite{Carav05} and to
references therein. Random walks in $\mathbb{R}^{d}$ conditioned to stay in
a cone have been considered in 
Shimura \cite{Shim91}, 
Garbit 
\cite{Garb10}, 
Echelsbacher and K\"{o}nig \cite{EichKonig} and
Denisov and Wachtel \cite{Den Wacht 2008}, \cite{Den Wachtel 2011}. 
The case of Markov chains with bounded jumps were considered by Varapoulos \cite{Var1999} who obtained
upper and lower bounds for the exit probability. 
However, to the best of our knowledge, the exact asymptotic of these probabilities 
for the products of random matrices 
has not yet been studied in the literature.

\HL{To state our results we need more notations}.  Let 
$N\left( g\right) =\max \left\{ \left\Vert g\right\Vert, \left\Vert g^{-1}\right\Vert \right\}.$ 
Denote by $\mathbb{P}\left(\mathbb{V}\right) $ 
the projective space of $\mathbb{V}$ endowed with the angular distance. 
In the sequel for any $v\in \mathbb{V}\smallsetminus \left\{ 0\right\} $ we
denote by $\overline{v}=\mathbb{R}v\in \mathbb{P}\left( \mathbb{V}\right) $
its direction and for any direction $\overline{v}\in \mathbb{P}\left( 
\mathbb{V}\right) $ we denote by $v$ a vector in $\mathbb{V}\smallsetminus
\left\{ 0\right\} $ of direction $\overline{v}.$ For any $g\in \mathbb{G}$
and $\overline{v}\in \mathbb{P}\left( \mathbb{V}\right) $ denote by $g\cdot 
\overline{v}$ the element of the projective space 
$\mathbb{P}\left( \mathbb{V}\right) $ 
associated to the product $gv,$ i.e.\ $g\cdot \overline{v}=\overline{gv}.$ 

Now we formulate our conditions.
The first condition requires exponential moments of the quantity $\log N\left( g\right) \geq 0.$

\vskip0.2cm \noindent \textbf{P1.} \textit{There exists }$\delta _{0}>0$%
\textit{\ such that}%
$\displaystyle  \quad  
\int_{\mathbb{G}} \exp \left( \delta _{0} \log N\left( g\right)    \right)  \boldsymbol{\mu } \left(dg\right) = 
\int_{\mathbb{G}}                 N\left( g\right) ^{\delta _{0}}              \boldsymbol{\mu }\left(dg\right) < +\infty.$
\vskip0.2cm

Denote by $\Gamma_{\boldsymbol \mu }$ the smallest closed semigroup which contains 
the support of ${\boldsymbol \mu }.$
The second condition requires, roughly speaking, that the dimension of
$\Gamma_{\boldsymbol \mu }$ cannot be reduced.

\vskip0.2cm \noindent \textbf{P2 (Strong irreducibility).}\textit{\ The
support  of $\boldsymbol{\mu }$
acts strongly irreducibly on $\mathbb{V}$,  
i.e.\ no proper \HL{finite union of} subspaces of  $\mathbb{V}$  is
invariant with respect to all elements of $\Gamma_{\boldsymbol \mu }.$ }
\vskip0.2cm

We say that the sequence $\left( h_{n}\right) _{n\geq 1}$ in $\mathbb G$
is contracting \HL{iff
$\left\Vert h_n\right\Vert ^{-1} h_n$ converges to a rank $1$ matrix $M$. 
The word "contracting" refers to the action of 
$h_n$ for the projective space $\mathbb{P}\left(\mathbb{V}\right) :$
for any $v \in \mathbb V \smallsetminus \textrm{Ker} (M) $ it holds 
$h_n \cdot \overline v \to \overline v_0 
$
 as $n \to +\infty,$ where the limit $\overline v_0$ 
 is the unique direction (not depending on $v$) corresponding to any vector $v_0\not = 0,$ $v_0 \in \textrm{Im}(M)$.} 

\vskip0.2cm \noindent \textbf{P3 (Contracting property). }\textit{The semigroup} 
$\Gamma_{\boldsymbol \mu }$ 
\textit{contains a contracting sequence.}

\vskip0.2cm
Some comments on the conditions \textbf{P2} and \textbf{P3} are appropriate. Note that any $h_n \in \mathbb G$ admits a polar decomposition: $h_n=w_n^1 a_n w_n^2,$ where $w_n^1, w_n^2$ are orthogonal matrices and $a_n$ is a diagonal matrix with entries $a_n(1)\geq \dots \geq a_n(d) > 0 $  on the diagonal.
A necessary and sufficient condition for the sequence $\left( h_{n}\right) _{n\geq 1}$
to be contracting is that  $\lim_{n\rightarrow +\infty }a_n(1) a_n(2)^{-1}= +\infty $ (see \cite{Boug-Lacr85}).
Then condition \textbf{P3} is verified if $\Gamma_{\boldsymbol \mu }$
contains a matrix $b$ with a unique eigenvalue (counting multiplicities) of maximal modulus (take $h_n=b^n$).  
For more details on the conditions \textbf{P2} and \textbf{P3} we refer the reader to \cite{Boug-Lacr85}.
It follows from \cite{Gold} that \textbf{P2} and \textbf{P3} are satisfied if and only if
the action of
the Zariski closure of $\Gamma_{\boldsymbol \mu }$ in $\mathbb G$
is strongly irreducible on $\mathbb V$ and contracting on $\mathbb{P(V)}.$ 

On the product space $\mathbb{G}\times \mathbb{P}\left( \mathbb{V}\right)$
define the function $\rho$ called norm cocycle by setting%
\begin{equation}
\rho \left( g,\overline{v}\right) :=\log \frac{\left\Vert gv\right\Vert }{%
\left\Vert v\right\Vert },\ \text{for}\ \left( g,\overline{v}\right) \in \mathbb{G}\times \mathbb{P}\left( \mathbb{V}\right).  \label{cocycle001}
\end{equation}%
It is well known (\cite{LePage82}, \cite{Boug-Lacr85}) 
that under conditions \textbf{P1-P3} 
there exists an unique $\boldsymbol{\mu }$-invariant measure $\boldsymbol{\nu }$ on $%
\mathbb{P}\left( \mathbb{V}\right) $ such that, for any continuous function $%
\varphi $ on $\mathbb{P}\left( \mathbb{V}\right) ,$%
\begin{equation}
\left( \boldsymbol{\mu }\ast \boldsymbol{\nu }\right) \left( \varphi \right)
=\int_{\mathbb{G}}\int_{\mathbb{P}\left( \mathbb{V}\right) }\varphi \left(
g\cdot \overline{v}\right) \boldsymbol{\nu }\left( d\overline{v}\right) 
\boldsymbol{\mu }\left( dg\right) =\int_{\mathbb{P}\left( \mathbb{V}\right)
}\varphi \left( \overline{v}\right) \boldsymbol{\nu }\left( d\overline{v}%
\right) =\boldsymbol{\nu }\left( \varphi \right) .  \label{st mes proj sp}
\end{equation}%
Moreover the upper Lyapunov exponent 
\begin{equation}
\gamma_{\boldsymbol{ \mu}} =\int_{\mathbb{G\times P}\left( \mathbb{V}%
\right) }\rho \left( g,\overline{v}\right) \boldsymbol{\mu }\left( dg\right) 
\boldsymbol{\nu }\left( d\overline{v}\right)  \label{Lyap expon}
\end{equation}%
is finite and there exists a constant $\sigma >0$ such that for any $v\in 
\mathbb{V\smallsetminus }\left\{ 0\right\} $ and any $t\in \mathbb{R},$%
\begin{equation*}
\lim_{n\rightarrow +\infty }\mathbf{Pr}\left( \frac{\log \left\Vert
G_{n}v\right\Vert -n\gamma_{\boldsymbol{ \mu}} }{\sigma \sqrt{n}}\leq t\right) =\Phi \left(
t\right) ,
\end{equation*}%
where $\Phi \left( \cdot \right) $ is the standard normal distribution.

\vskip0.2cm \noindent \textbf{P4. }\textit{The upper Lyapunov exponent $\gamma_{\boldsymbol{ \mu}}$ is equal to $0$}.

Hypothesis $\textbf{P4}$ does not imply that the event 
\HL{$\left\{ \tau _{v}>n\right\} 
$
occurs} with positive probability for any $v\in \mathbb{B}^{c}$ and $n$ large enough;  to ensure
this we need the following additional condition:

\vskip0.2cm \noindent \textbf{P5. }\textit{There exists }$\delta >0$\textit{%
\ such that}%
$\displaystyle \quad 
\inf_{s\in \mathbb{S}^{d-1}}\boldsymbol{\mu }\left( g:\log \left\Vert
gs\right\Vert >\delta \right) >0.
$

 It is shown in the Appendix  that a measure $\boldsymbol{\mu }$
on $\mathbb{G}$ satisfying \textbf{P1-P5} and such that $\gamma_{\boldsymbol{ \mu}} =0$ exists.

\HL{From the main results of the paper we deduce the following statements.} 
Assume \textbf{P1-P5}. \HL{Then, see Theorems \ref{Theorem harmonic func} and \ref{Th asympt tau},} 
for any $v\in \mathbb{B}^{c},$
\begin{equation}
\mathbf{Pr}\left( \tau _{v}>n\right) =\frac{2V\left( v\right) }{\sigma \sqrt{%
2\pi n}}\left( 1+o\left( 1\right) \right) \;\text{as\ }n\rightarrow +\infty ,
\label{intro001}
\end{equation}%
where $V$ is a positive function on $\mathbb{B}^{c}$ with following
properties: for any $s\in \mathbb{S}^{d-1}$ the function $t\rightarrow
V\left( ts\right) $ is increasing on $\left( 1,\infty \right) ,$ $0\vee
\left( \log t-a\right) \leq V\left( ts\right) \leq c\left( 1+\log t\right) $
for $t>1$ and some constant $a>0,$ and $\lim_{t\rightarrow +\infty }\frac{%
V\left( ts\right) }{\log t}=1.$
Moreover, in Theorem \ref{Th weak conv cond positive}, 
we prove that the \HL{conditional law of  
$\frac{1}{\sigma \sqrt{n}}\log \left\Vert G_{n}v\right\Vert $ 
given the event $\left\{ \tau _{v}>n\right\} $
converges} to the Rayleigh distribution $\Phi ^{+}\left( t\right) =1-\exp
\left( -\frac{t^{2}}{2}\right) :$ for any $v\in \mathbb{B}^{c}$ 
and for any $t\geq 0,$
\begin{equation}
\lim_{n\rightarrow +\infty }\mathbf{Pr}\left( \left. \frac{\log \left\Vert
G_{n}v\right\Vert }{\sigma \sqrt{n}}\leq t\right\vert \tau _{v}>n\right)
=\Phi ^{+}\left( t\right) .
\label{intro002}
\end{equation}

The usual way for obtaining such type of results for a classical random walk on the real line is the
Wiener-Hopf factorization (see Feller \cite{Feller}). 
Unfortunately, the Wiener-Hopf factorization is not suited for studying the exit time probabilities for random
walks in $\mathbb R^d$ or walks based on dependent variables. 
Alternative approaches have been developed recently: we refer to
\cite{Var1999}, \cite{EichKonig}, \cite{Den Wacht 2008} and \cite{Den Wachtel 2011}. 
For the results (\ref{intro001}) and (\ref{intro002}) we rely, on the one hand, upon these developments, and, 
on the other hand,  upon a key strong approximation result for dependent random variables
established separately in \cite{GLP2014}.
As it is shown in Le Page \cite{LePage82}, the study of a product of random matrices is reduced to
the study of a specially designed Markov chain.
One of the difficulties is the construction of the harmonic function for the obtained chain. 
To crop with this, we show that, under assumptions \textbf{P1-P4}, its perturbed transition operator 
satisfies a spectral gap property on an appropriately chosen Banach space. 
This property allows to built a martingale approximation, from which we derive the existence of the harmonic function. 
In the second part of the paper we transfer the properties of the exit time 
from the Gaussian sequence to the associated Markov walk
based on the strong approximation \cite{GLP2014}. 

We end this section by recalling some standard notations. Throughout the
paper $c,c^{\prime },c^{\prime \prime },...$ with or without indices denote
absolute constants. By $c_{\varepsilon },c_{\varepsilon ,\delta }^{\prime
},...$ we denote constants depending only on theirs indices. All these
constants are not always the same when used in different formulas. Other
constants will be specifically indicated.
\HL{The integer part of a real number $a$ is denoted by $[a].$} 
By $\phi _{\sigma }\left( t\right) =\frac{1}{\sqrt{2\pi }}\exp \left( -\frac{%
t^{2}}{2\sigma ^{2}}\right) $ and $\Phi _{\sigma }\left( t\right)
=\int_{-\infty }^{t}\phi _{\sigma }\left( t\right) du$ we denote
respectively the normal density and distribution functions of mean $0$ and
variance $\sigma ^{2}$ on the real line $\mathbb{R}.$ The identity matrix in 
$\mathbb{G}$ is denoted by $I$ and $g^{\prime }$ is the transpose of $g\in 
\mathbb{G}.$ 
For two
sequences $\left( a_{n}\right) _{n\geq 1}$ and $\left( b_{n}\right) _{n\geq
1}$ in $\mathbb{R}^{\ast }_{+}$ the equivalence $a_{n}\sim b_{n}$ as $%
n\rightarrow +\infty $ means $\lim_{n\rightarrow +\infty } \frac{a_{n}}{b_{n}}=1.$

\section{Main results}

Consider the homogenous Markov chain $\left( X_{n}\right) _{n\geq 0}$ with
values in the product space $\mathbb{X}=\mathbb{G}\times \mathbb{P}\left( 
\mathbb{V}\right) $ and initial value $X_{0}=\left( g,\overline{v}\right)
\in \mathbb{X}$ by setting $X_{1}=\left( g_{1},g\cdot \overline{v}\right) $
and%
\begin{equation}
X_{n+1}=\left( g_{n+1},g_{n}...g_{1}g\cdot \overline{v}\right) ,\;n\geq 1.
\label{Markov chain001}
\end{equation}%
For any  $x=\left( g,\overline{v}\right) \in \mathbb X$ 
denote by 
$\mathbf P (x,  dg' ,  d\overline{v}') = \boldsymbol{\mu }\left( dg'\right)\delta_{ g\cdot \overline{v}} (d \overline{v}')$
the transition probability of $\left( X_{n}\right) _{n\geq 0}$
and by
\begin{equation}
\mathbf{P}f\left( x \right) 
= \int_{\mathbb{X}}f\left( x'  \right)   \mathbf P (x, d x') 
= \int_{\mathbb{G}}f\left(g',g\cdot \overline{v}\right) \boldsymbol{\mu }\left( dg'\right), 
\label{trans prob X}
\end{equation}%
the corresponding transition operator, where $f$ is any bounded
measurable function $f$ on $\mathbb{X}.$ On the space $\mathbb{X}$ define
the probability measure 
\begin{equation}
\boldsymbol{\lambda }\left( dg,d\overline{v}\right) =\boldsymbol{\mu }\left(
dg\right) \times \boldsymbol{\nu }\left( d\overline{v}\right) ,
\label{invar mes X}
\end{equation}%
where $\boldsymbol{\nu }$ is the $\boldsymbol{\mu }$-invariant measure
defined by (\ref{st mes proj sp}). 
 It is shown in Section \ref{Sec Exist Inv Prob} that under conditions \textbf{P1-P3} the measure $\boldsymbol{\lambda }
$ is stationary for the Markov chain $\left( X_{k}\right) _{k\geq 0},$ i.e.
that $\boldsymbol{\lambda }\left( \mathbf{P}f\right) =\boldsymbol{\lambda }%
\left( f\right) $ for any bounded measurable function $f$ on $\mathbb{X}.$
Denote by $\mathbb{P}_{x}$ the probability measure generated by the finite
dimensional distributions of $\left( X_{k}\right) _{k\geq 0}$ starting at $%
X_{0}=x\in \mathbb{X}$ and by $\mathbb{E}_{x}$ 
the corresponding expectation; 
for any probability measure ${\bf \nu}$ on $\mathbb X$, we set $\mathbb P_{\bf \nu}=\int_{\mathbb X}\mathbb P_x {\bf \nu}(dx)$. Then for any bounded measurable function $f$ on $\mathbb{X},$%
\begin{equation*}
\mathbb{E}_{x}f\left( X_{1}\right) =\mathbf{P}f\left( x\right) \ \ \text{and
\ }\mathbb{E}_{x}f\left( X_{n}\right) =\mathbf{P}^{n}f\left( x\right)
,\;n\geq 2.
\end{equation*}

Let $v\in \mathbb{V\smallsetminus }\left\{ 0\right\} $ be a starting vector
and $\overline{v}$ be its direction. From (\ref{cocycle001}) and (\ref%
{Markov chain001}), iterating the cocycle property
$\rho \left( g_2 g_1,\overline{v}\right) =\rho \left( g_2,g_1\cdot \overline{v}\right) +\rho \left( g_1,\overline{v}%
\right) $ one gets the basic representation%
\begin{equation}
\log \left\Vert G_{n}gv\right\Vert =y+\sum_{k=1}^{n}\rho \left( X_{k}\right)
,\ n\geq 1,  \label{relation001}
\end{equation}%
where $y=\log \left\Vert g v\right\Vert $ determines the ``size'' of the vector $%
gv. $ In the sequel we will deal with the   random walk $\left( y+S_{n}\right)
_{n\geq 0}$ associated to the Markov chain $\left( X_{n}\right) _{n\geq 0},$
where $X_{0}=x=\left( g,\overline{v}\right) $ is an arbitrary element of $%
\mathbb{X},$ $y$ is any real number and%
\begin{equation}
S_{0}=0,\;S_{n}=\sum_{k=1}^{n}\rho \left( X_{k}\right) ,\;n\geq 1.
\label{Markov walk001}
\end{equation}

In the proof of our main results we will make use of the following CLT which
can be deduced from Theorem 2 (p. 273) in Le Page \cite{LePage82} and where
we assume that $\gamma_{\boldsymbol{ \mu}} =0.$

\begin{theorem}
\label{CLT GLDR}Assume \textbf{P1-P4}. Then there exists a constant $\sigma
\in \left( 0,\infty \right) $ such that uniformly in $x\in \mathbb{X}$ and $%
t>0,$ 
\begin{equation*}
\lim_{n\rightarrow +\infty }\mathbb{P}_{x}\left( \frac{S_{n}}{\sigma \sqrt{n}}%
\leq t\right) =\Phi \left( t\right) .
\end{equation*}
\end{theorem}

Using (\ref{relation001}) and the results from \cite{LePage82} it is easy to
obtain the well known expression%
\begin{equation}
\sigma ^{2}=\text{Var}_{\mathbb{P}_{\boldsymbol{\lambda }}}\left( \rho
\left( X_{1}\right) \right) +2\sum_{k=2}^{+\infty }\text{Cov}_{\mathbb{P}_{%
\boldsymbol{\lambda }}}\left( \rho \left( X_{1}\right) ,\rho \left(
X_{k}\right) \right) < +\infty .  \label{sigma001}
\end{equation}

For any $y>0$ denote by $\tau _{y}$ the first time when the Markov walk $%
\left( y+S_{n}\right) _{n\geq 0}$ becomes negative:%
\begin{equation*}
\tau _{y}=\min \left\{ n\geq 1:y+S_{n}\leq 0\right\} .
\end{equation*}%
From Lemma \ref{Lemma finite tau} of Section \ref{sec invar func} it follows
that, for any $y>0$ and $x\in \mathbb{X},$ the stopping time $\tau _{y}$ is $\mathbb{P}_{x}$-a.s.~finite.

To state our first order asymptotic of the probability $\mathbb{P}_{x}\left(
\tau _{y}>n\right) $ we need an harmonic function which we proceed to
introduce. For any $\left( x,y\right) \in \mathbb{X\times R}$ denote by 
$\mathbf{Q}\left( x,y,dx^{\prime }\times dy^{\prime }\right) =\mathbb{P}_{x}\left( X_{1}\in dx^{\prime },y+S_{1}\in dy^{\prime }\right) $ the
transition probability of the two dimensional Markov chain $\left(
X_{n},y+S_{n}\right) _{n\geq 0}$ under the measure $\mathbb{P}_{x}.$
Consider the transition kernel 
$\left( x,y\right) \in \mathbb{X}\times \mathbb{R}^{\ast }_{+}\rightarrow \mathbf{Q}_{+}\left( x,y,\cdot \right) $
on $\mathbb{X}\times \mathbb{R}^{\ast }_{+}$ defined by%
\begin{equation*}
\mathbf{Q}_{+}\left( x,y,\cdot \right) =1_{\mathbb{X}\times \mathbb{R}^{\ast}_{+}}\left( \cdot \right) \mathbf{Q}\left( x,y,\cdot \right) .
\end{equation*}%
Note that $\mathbf{Q}_{+}$ is not a Markov kernel. 
A $\mathbf{Q}_{+}$-positive harmonic function $V$ is any 
function $V:\mathbb{X}\times \mathbb{R}^{\ast }_{+}\rightarrow \mathbb{R}^{\ast }_{+}$ satisfying
\begin{equation}
\mathbf{Q}_{+}V=V.  \label{harm}
\end{equation}%
The function $V$ can be extended by setting $V\left( x,y\right)= 0$ for $\left( x,y\right) \in \mathbb{X}\times \mathbb{R}_{-}.$
The kernel $\mathbf{Q}_{+}$ 
and the function $V$ are related to the exit time 
$\tau _{y}$ by the following identity: for any $x\in \mathbb{X},$ $y>0,$ 
$n\geq 1$ and bounded measurable function $\varphi $ on $\mathbb{X}\times \mathbb{R}^{\ast }_{+},$%
\begin{equation*}
\mathbf{Q}_{+}^{n}\left( V\varphi \right) \left( x,y\right) =\mathbb{E}%
_{x}\left( V\varphi \left( X_{n},y+S_{n}\right) ;\tau _{y}>n\right).
\end{equation*}%
With $\varphi =1$ we see that $V$ is $\mathbf{Q}_{+}$-harmonic iff, for any $x\in \mathbb{X},$ $y>0,$
\begin{equation*}
V\left( x,y\right) =\mathbb{E}_{x}\left( V\left( X_{1},y+S_{1}\right) ;\tau
_{y}>1\right) .
\end{equation*}

The following theorem proves the existence of a $\mathbf{Q}_{+}$-harmonic
function. We also establish some of its important properties such as linear
behavior as $y\rightarrow +\infty .$

\begin{theorem}
\label{Theorem harmonic func}Assume hypotheses \textbf{P1-P5}.

\noindent 1. For any $x\in \mathbb{X}$ and $y>0$ the limit 
\begin{equation*}
V\left( x,y\right) =\lim_{n\rightarrow +\infty }\mathbb{E}_{x}\left(
y+S_{n};\tau _{y}>n\right)
\end{equation*}%
exists and satisfies $V\left( x,y\right) >0.$ Moreover, for any $x\in 
\mathbb{X}$ the function $V\left( x,\cdot \right) $ is increasing on $%
\mathbb{R}^{\ast }_{+},$ satisfies $0\vee \left( y-a\right) \leq V\left(
x,y\right) \leq c\left( 1+y\right) $ for any $y>0$ and some $a>0,$ and $%
\lim_{y\rightarrow +\infty }\frac{V\left( x,y\right) }{y}=1.$

\noindent 2. The function $V$ is $\mathbf{Q}_{+}$-harmonic, i.e. for any $%
x\in \mathbb{X}$ and $y>0,$ 
\begin{equation*}
\mathbb{E}_{x}\left( V\left( X_{1},y+S_{1}\right) ;\tau _{y}>1\right)
=V\left( x,y\right) .
\end{equation*}
\end{theorem}

The proof of this theorem is given in Section \ref{sec invar func} (see
Propositions \ref{PROP func V for S} and \ref{PROP harmonic func}) with $%
a=2\left\Vert \mathbf{P}\theta \right\Vert _{\infty },$ where the function $%
\theta $ is the solution of the Poisson equation $\theta -\mathbf{P}\theta
=\rho $ (see Section \ref{secMartDec} for the existence of the function $%
\theta $ and its properties).

Now we state our main result concerning the limit behavior of the exit time $\tau _{y}.$

\begin{theorem}
\label{Th asympt tau}Assume hypotheses \textbf{P1-P5}. Then, for any $x\in 
\mathbb{X}$ and $y>0,$%
\begin{equation*}
\mathbb{P}_{x}\left( \tau _{y}>n\right) \sim \frac{2V\left( x,y\right) }{%
\sigma \sqrt{2\pi n}}\text{\ as\ }n\rightarrow +\infty .
\end{equation*}
Moreover, there exists a constant $c$ such that for any $y>0$ and $x \in \mathbb X,$
\begin{equation*}
\sup_{n \geq 1}\sqrt{n}\mathbb{P}_{x}\left( \tau _{y}>n\right) \leq c\frac{1+y}{\sigma}.
\end{equation*}
\end{theorem}

The following theorem establishes a central limit theorem for the sum $%
y+S_{n}$ conditioned to stay positive.

\begin{theorem}
\label{Th weak conv cond positive}Assume hypotheses \textbf{P1-P5}. For any $%
x\in \mathbb{X},\ y>0$ and $t>0,$%
\begin{equation*}
\lim_{n\rightarrow +\infty } \mathbb{P}_{x}\left( \left. \frac{%
y+S_{n}}{\sigma \sqrt{n}}\leq t\right\vert \tau _{y}>n\right) = \Phi^{+}\left( t\right),
\end{equation*}%
where $\Phi ^{+}\left( t\right) =1-\exp \left( -\frac{t^{2}}{2}\right) .$
\end{theorem}

The results for $\log \left\Vert G_{n}v\right\Vert $ stated in the previous
section are obtained by taking $X_{0}=x=\left( I,\overline{v}\right) $ as
the initial state of the Markov chain $\left( X_{n}\right) _{n\geq 0}$ and
setting \HL{$y=\ln \Vert v\Vert$ and} $V\left( v\right) =V\left( \HL{\left( I,\overline{v}\right)},\ln \Vert v\Vert\right).$
\section{Banach space and spectral gap conditions\label{Banach space}}

In this section we verify the spectral gap properties \textbf{M1-M3} of the
perturbed transition operator of the Markov chain $\left( X_{n}\right)_{n\geq 0}$ 
acting on a Banach space to be introduced below; for more details we refer to \cite{LePage82}.  
Under these properties and some additional moment conditions \textbf{M4-M5} stated
below, in the paper \cite{GLP2014} we have established a
Komlos-Major-Tusnady type strong approximation result for Markov chains (see
Proposition \ref{Proposition KMT}) which is one of the crucial points in the
proof of the main results of the paper. 
\HL{The conditions \textbf{M1-M5} also imply
the existence of the solution $\theta $ of the Poisson
equation $\rho =\theta -\mathbf{P}\theta $ which is used in the next section
to construct a martingale approximation of the Markov walk $\left(
S_{n}\right) _{n\geq 0}.$}

On the projective space $\mathbb{P}\left( 
\mathbb{V}\right) $ define the \HL{angular} distance 
$\displaystyle 
d\left( \overline{u},\overline{v}\right) =\frac{\left\Vert u\wedge
v\right\Vert }{\left\Vert u\right\Vert \left\Vert v\right\Vert },$ 
where $u\wedge v$ is the vector product of $u$ and $v.$
Let $\mathcal{C}_{b}$ be the vector space of complex valued continuous bounded functions 
\HL{$f:\mathbb{X}\rightarrow \mathbb{C}$} 
endowed with the supremum norm 
$\left\Vert f\right\Vert _{\infty}=\sup_{\left( g,\overline{v}\right) \in \mathbb{X}}\left\vert f\left( g,\overline{v}\right) \right\vert.$ 
Let $\delta _{0}>0$ and $0<\varepsilon <\delta _{0}.$ For any $f\in \mathcal{%
C}_{b}$ set%
\begin{eqnarray}\qquad \qquad 
k_{\varepsilon }\left( f\right) =\sup_{\overline{u}\neq \overline{v}%
,\;g\in \mathbb{G}}\frac{\left\vert f\left( g,\overline{u}\right) -f\left( g,%
\overline{v}\right) \right\vert }{d\left( \overline{u},\overline{v}\right)
^{\varepsilon }N\left( g\right) ^{4\varepsilon }} 
+\sup_{g\neq h,\;\overline{u}\in \mathbb{P}\left( \mathbb{V}\right) }\frac{%
\left\vert f\left( g,\overline{u}\right) -f\left( h,\overline{u}\right)
\right\vert }{\left\Vert g-h\right\Vert ^{\varepsilon }\left( N\left(
g\right) N\left( h\right) \right) ^{3\varepsilon }}.  \label{module cont}
\end{eqnarray}%
Define the vector space 
$
\mathcal{B}=\mathcal{B}_{\varepsilon }:=\left\{ f\in \mathcal{C}%
_{b}:k_{\varepsilon }\left( f\right) < +\infty \right\} .
$
Endowed with the norm%
\begin{equation}
\left\Vert f\right\Vert _{\mathcal{B}}=\left\Vert f\right\Vert _{\infty
}+k_{\varepsilon }\left( f\right)  \label{norm}
\end{equation}%
the space $\mathcal{B}$ becomes a Banach space. Note also that $f_1,f_2\in \mathcal{B}$ implies $f_1f_2\in \mathcal{B}$ with $\left\Vert f_1f_2\right\Vert
_{\mathcal{B}}\leq \left\Vert f_1\right\Vert _{\mathcal{B}%
}\left\Vert f_2\right\Vert _{\mathcal{B}},$ so that $\mathcal{%
B}$ is also a Banach algebra.

Denote by $\mathcal{B}^{\prime }=\mathcal{L}\left( \mathcal{B},\mathbb{C}%
\right) $ the topological dual of $\mathcal{B}$ equipped with the norm $%
\left\Vert \cdot \right\Vert _{\mathcal{B}^{\prime }}:$ $\left\Vert \psi
\right\Vert _{\mathcal{B}^{\prime }}=\sup_{\left\Vert f\right\Vert _{%
\mathcal{B}}\leq 1}\frac{\left\vert \psi \left( f\right) \right\vert }{%
\left\Vert f\right\Vert _{\mathcal{B}}},$ for any linear functional $\psi
\in \mathcal{B}^{\prime }.$  For any linear operator $A$ from $\mathcal{B}$
to $\mathcal{B}$, its operator norm on $\mathcal{B}$ is 
$\left\Vert A\right\Vert _{\mathcal{B\rightarrow B}}:= \sup_{\left\Vert f\right\Vert _{%
\mathcal{B}}\leq 1}\frac{\left\Vert Af\right\Vert _{\mathcal{B}}}{\left\Vert
f\right\Vert _{\mathcal{B}}}.$  The Dirac measure $\delta_x$ at $x\in \mathbb{X}$ is defined by $%
\delta _{x}\left( f\right) =f\left( x\right) $ for any $f\in 
\mathcal{B}.$ The unit function $e$ on $\mathbb{X}$ is defined by $e\left(
x\right) =1$ for $x\in \mathbb{X}.$

Using the techniques of the paper \cite{LePage82}, it can be checked that under \textbf{P1-P4} the condition \textbf{M1-M3} below are satisfied:
\vskip0.2cm \noindent \textbf{M1 (Banach space):}

\nobreak\noindent a)\textit{\ The unit function }$e$\textit{\ belongs to }$%
\mathcal{B}.$

\noindent b)\textit{\ For every }$x\in \mathbb{X}$\textit{\ the Dirac
measure }$\boldsymbol{\delta }_{x}$\textit{\ belongs to }$\mathcal{B}%
^{\prime },$ with\textit{\ }$\sup_{x\in \mathbb{X}}\left\Vert \boldsymbol{%
\delta }_{x}\right\Vert _{\mathcal{B}^{\prime }}\leq 1.$\textit{\ }

\noindent c)\textit{\ }$\mathcal{B}\subseteq L^{1}\left( \mathbf{P}\left(
x,\cdot \right) \right) $\textit{\ for every }$x\in \mathbb{X}.$

\noindent d)\textit{\ There exists a constant }$\eta _{0}\in \left(
0,1\right) $\textit{\ such that for any }$t\in \left[ -\eta _{0},\eta _{0}%
\right] $ and\textit{\ }$f\in \mathcal{B}$\textit{\ the function }$e^{it\rho
}f$ \textit{belongs to }$\mathcal{B}.$
\begin{proof}
Proof of assertion a). Obvious.

Proof of assertion b). Since $\boldsymbol{\delta }_{x}\left(
f\right) =f\left( x\right) $ for any $x\in \mathbb{X},$ using (\ref{norm})
we have 
\begin{equation*}
\left\Vert \boldsymbol{\delta }_{x}\left( f\right) \right\Vert _{\mathcal{B}%
^{\prime }}=\sup_{f}\frac{\left\vert \boldsymbol{\delta }_{x}\left( f\right)
\right\vert }{\left\Vert f\right\Vert _{\mathcal{B}}}=\sup_{f}\frac{%
\left\vert f\left( x\right) \right\vert }{\left\Vert f\right\Vert _{\mathcal{%
B}}}\leq \sup_{f}\frac{\left\Vert f\right\Vert _{\infty }}{\left\Vert
f\right\Vert _{\mathcal{B}}}\leq 1.
\end{equation*}%
Assertion c) follows from the fact that the functions in $\mathcal{B}$ are
bounded.

Proof of assertion d). From Corollaries \ref{CorollaryBBB2} and \ref%
{CorollaryBBB4} if follows that $e^{it\rho }$ belongs to $\mathcal{B}$ for
any $\left\vert t\right\vert \leq \eta _{0}.$ Since $\mathcal{B}$ is an
algebra the function $e^{it\rho }g$\ belongs to $\mathcal{B}$ for any for
any $t$\ satisfying $\left\vert t\right\vert \leq \eta _{0}.$ This finishes
the proof of d).
\end{proof}

Condition \textbf{M1} c) implies that  the operator $\mathbf{P}$ defined by 
\begin{equation*}
f\in \mathcal{B}\rightarrow \mathbf{P}f(\cdot )=\int_{\mathbb{X}}f(g,%
\overline{v})\mathbf{P}(\cdot ,dg,d\overline{v})
\end{equation*}%
is well defined. 
Moreover,  it follows from \textbf{M1} d) that the perturbed
operator $\mathbf{P}_{t}f=\mathbf{P}(e^{it\rho }f)$ is well defined for any $%
t\in \lbrack -\eta _{0},\eta _{0}]$ and $f\in \mathcal{B}$ (notice that $%
\mathbf{P}=\mathbf{P}_{0}).$
  
 
\noindent \textbf{M2 (Spectral gap): }

\nobreak\noindent a)\textit{\ The map }$f\mapsto \mathbf{P}f$\textit{\ is a
bounded operator on }$\mathcal{B}.$

\noindent b)\textit{\ There exist constants }$\HL{C}>0$\textit{\ and }$r
\in (0,1)$\textit{\ such that}%
\begin{equation*}
\mathbf{P}=\Pi +R,   \end{equation*}%
\textit{where }$\Pi $\textit{\ is a one dimensional projector, }$\Pi 
\mathcal{B}=\left\{ f\in \mathcal{B}:\mathbf{P}f=f\right\} $\textit{\ and }$%
R $\textit{\ is an operator on }$\mathcal{B}$\textit{\ satisfying }$\Pi
R=R\Pi =0$\textit{\ and }$\left\Vert R^{n}\right\Vert _{\mathcal{B}%
\rightarrow \mathcal{B}}\leq \HL{C} r ^{n},$ $n\geq 1.$

\begin{proof}
Assertion a) can be easily checked. It also follows from condition \textbf{M3} below.

We now prove assertion b). Let $\alpha $ be any complex eigenvalue of
modulus $1$ and $f_{\alpha }\not=0$ be a corresponding eigenfunction. Then
there exists $x_{0}\in \mathbb{X}$ such that $f\left( x_{0}\right) \not=0$
and $\mathbf{P}^{n}f_{\alpha }\left( x_{0}\right) =\alpha ^{n}f_{\alpha
}\left( x_{0}\right) .$ Since $\lim_{n\rightarrow +\infty }\mathbf{P}%
^{n}f_{\alpha }\left( x_{0}\right) $ exists (see Section \ref{Sec Exist Inv
Prob}) it follows that $\alpha =1$ is the unique eigenvalue of modulus $1.$

It is shown in Section \ref{Sec I T M} that under \textbf{P1-P4} 
the conditions of the theorem of Ionescu-Tulcea and Marinescu \cite{ITM50} are satisfied.
Taking into account that $\alpha =1$ is the unique eigenvalue of
modulus $1$ from this theorem we conclude that%
\begin{equation*}
\mathbf{P}=\Pi +R,
\end{equation*}%
where $\Pi $ is an operator on $\mathcal{B}$ satisfying $\Pi ^{2}=\Pi $
(i.e. $\Pi $ is a projector), $\Pi \mathcal{B}=\left\{ f\in \mathcal{B}:%
\mathbf{P}f=f\right\} $ and $R$ is an operator on $\mathcal{B}$ satisfying $%
R\Pi =\Pi R=0$ with spectral radius $r\left( R\right) <1.$ The required
assertion follows.
\end{proof}

Since $e$ is the eigenfunction corresponding to eigenvalue $1$ of $\mathbf{P}%
,$ the space $\Pi \mathcal{B}$ is generated by $e.$ Moreover for any $f\in 
\mathcal{B},$%
\begin{equation}
\Pi f=\boldsymbol{\lambda }\left( f\right) e,  \label{linear form}
\end{equation}%
where $\boldsymbol{\lambda }$ is the stationary measure defined by (\ref%
{invar mes X}). Indeed, for any $f\in \mathcal{B},$ there exists $c_{f}\in 
\mathbb{R}$ such that $\Pi f=c_{f}e.$ By condition \textbf{M2} it follows
that $\boldsymbol{\lambda }\left( f\right) =\boldsymbol{\lambda }\left( 
\mathbf{P}^{n}f\right) =\boldsymbol{\lambda }\left( \Pi f\right) +%
\boldsymbol{\lambda }\left( R^{n}f\right) ,$ for any $n\geq 1.$ Taking the
limit as $n\rightarrow +\infty $ and using again \textbf{M2} we obtain $%
\boldsymbol{\lambda }\left( f\right) =\boldsymbol{\lambda }\left( \Pi
f\right) =c_{f}.$

\vskip0.2cm \noindent \textbf{M3 (Perturbed transition operator): }\textit{%
There exists a constant }$C_{\mathbf{P}}>0$ \textit{such that, for all } $%
n\geq 1,$ 
\begin{equation}
\sup_{\left\vert t\right\vert \leq \eta _{0}}\left\Vert \mathbf{P}%
_{t}^{n}\right\Vert _{\mathcal{B}\rightarrow \mathcal{B}}\leq C_{\mathbf{P}}.
\label{boundness of Ptn}
\end{equation}

\begin{proof}
From Lemma \ref{PropDoeblFortet} it follows that there exist constants $\eta
_{0},c_{\varepsilon }>0$ and $\rho _{\varepsilon }\in (0,1)$ such that for
any $n\geq 1$ 
\begin{equation*}
\sup_{\left\vert t\right\vert <\eta _{0}}\left\Vert \mathbf{P}%
_{t}^{n}f\right\Vert _{\mathcal{B}}\leq \left( 1+c_{\varepsilon }\rho
_{\varepsilon }^{n}\right) \left\Vert f\right\Vert _{\infty }\leq \left(
1+c_{\varepsilon }\rho _{\varepsilon }^{n}\right) \left\Vert f\right\Vert _{%
\mathcal{B}}.
\end{equation*}%
This implies that 
\begin{equation*}
\sup_{\left\vert t\right\vert \leq \eta _{0}}\left\Vert \mathbf{P}%
_{t}^{n}\right\Vert _{\mathcal{B}\rightarrow \mathcal{B}}\leq \left(
1+c_{\varepsilon }\rho _{\varepsilon }^{n}\right) \leq 1+c_{\varepsilon }.
\end{equation*}%
This proves (\ref{boundness of Ptn}) with $C_{\mathbf{P}}=1+c_{\varepsilon
}. $
\end{proof}

 Using condition \textbf{P1}, we readily deduce the conditions \textbf{M4}-\textbf{M5} below:
\vskip0.2cm \noindent \textbf{M4 (Moment condition): }\textit{For any }$p>2,$
\textit{\ it holds} 
$$ 
\sup_{x\in \mathbb{X}}\sup_{n\geq 1}\mathbb{E}_{x}^{1/p}\left\vert
\rho \left( X_{n}\right) \right\vert ^{p}  < \infty .
$$
\HL{\begin{proof}
Using condition \textbf{P1}, for any $p>2$ we have%
\begin{eqnarray*}
\sup_{x\in \mathbb{X}}\sup_{k\geq 1}\mathbb{E}_{x}^{1/p}\left\vert \rho
\left( X_{k}\right) \right\vert ^{p} &=&\sup_{\left( g,\overline{v}\right)
\in \mathbb{X}}\sup_{k\geq 1}\left( \mathbf{E}_{\mathbf{Pr}}\left\vert \log
\left\Vert g_{k}\frac{g_{k-1}...g_{1}g\overline{v}}{\left\Vert
g_{k-1}...g_{1}g\overline{v}\right\Vert }\right\Vert \right\vert ^{p}\right)
^{1/p} \\
&\leq &\left( \int_{\mathbb{G}}\left( \log N\left( g\right) \right) ^{p}%
\boldsymbol{\mu }\left( dg\right) \right) ^{1/p}=c_{p}< +\infty .
\end{eqnarray*}
\end{proof}}

\HL{Now using the moment condition \textbf{M4} we obtain:}

\vskip0.2cm \noindent \textbf{M5:}\textit{\ The stationary probability
measure }$\boldsymbol{\lambda }$\textit{\ satisfies } 
$$
\int \sup_{n\geq 0}\mathbf{P}^{n} \rho  ^{2}(x) \boldsymbol{\lambda }\left( dx\right) < +\infty.
$$
\HL{\begin{proof}
The assertion follows since
$\sup_{n\geq 0}\mathbf{P}^{n} \rho  ^{2} \leq \rho^2 + \sup_{n\geq 1}\mathbf{P}^{n} \rho  ^{2},$ 
the function $\rho^2$ is integrable with respect to the stationary measure $\boldsymbol{\lambda }$ and
by \textbf{M4} it holds
 $ \sup_{k\geq 1}\mathbf{P}^{k} \rho^{2} \leq c_p.$ 
\end{proof}}

\section{Martingale approximation\label{secMartDec}}

The goal of this section is to construct a martingale approximation for the
Markov walk $\left( S_{n}\right) _{n\geq 0}.$ All over the section it is
assumed that hypotheses \textbf{M1-M3} hold true. Following Gordin \cite%
{Gordin}, we define the function $\theta $ as the solution of the Poisson
equation $\rho =\theta -\mathbf{P}\theta $ and the approximating martingale
by $M_{n}=\sum_{k=1}^{n}\left( \theta \left( X_{k}\right) -\mathbf{P}\theta
\left( X_{k-1}\right) \right) ,$ $n\geq 1.$ Note that, the norm cocycle (\ref%
{cocycle001}) does not belong to the Banach space $\mathcal{B},$ so that the
existence of the solution of the Poisson equation does not follow directly
from condition \textbf{M3}.

\begin{lemma}
\label{Lemma bound g}The sum $\theta =\rho +\sum_{n=1}^{+\infty }\mathbf{P}%
^{n}\rho $ exists and satisfies the Poisson equation $\rho =\theta -\mathbf{P%
}\theta .$ Moreover,%
\begin{equation}
\sup_{x\in \mathbb{X}}\left\vert \theta \left( x\right) -\rho \left(
x\right) \right\vert < +\infty .  \label{bound001}
\end{equation}
\end{lemma}

\begin{proof}
Let us emphasize that the function $\rho$  may be unbounded and so may not belong to $\mathcal{B}$. 
The key point in what follows is that  $\overline{\rho }=\mathbf{P}\rho \in \mathcal{B}$, which is established in Lemma \ref{lemma rho smoo} of the Appendix; 
by condition \textbf{M2},  this readily implies  %
\begin{equation*}
\mathbf{P}^{n}\rho =\mathbf{P}^{n-1}\overline{\rho }=\boldsymbol{\lambda }%
\left( \overline{\rho }\right) +R^{n-1}\left( \overline{\rho }\right)
=R^{n-1}\left( \overline{\rho }\right) ,
\end{equation*}%
where, by the stationarity of $\boldsymbol{\lambda }$ 
we have $\boldsymbol{\lambda }\left( \overline{\rho }\right) =\boldsymbol{%
\lambda }\left( \mathbf{P}\rho \right) =\boldsymbol{\lambda }\left( \rho
\right) =\gamma_{\boldsymbol{ \mu}}=0.$ Since the spectral radius of the operator $R$ is less than $1,$ it holds%
\begin{equation*}
\left\Vert \mathbf{P}^{n}\rho \right\Vert _{\infty }\leq \left\Vert \mathbf{P%
}^{n}\rho \right\Vert _{\mathcal{B\rightarrow B}} = 
\left\Vert R^{n-1}\left( \overline{\rho }\right) \right\Vert _{\mathcal{B\rightarrow B}}.
\end{equation*}%
Consequently,  by {\bf M2} b),  there exists a real number $c_1>0$ such that, for any $x\in \mathbb{X}$ and   $n\geq 1,$%
\begin{equation}
 \left\vert \mathbf{P}^{n}\rho \left( x\right)
\right\vert \leq c_{1} r ^{n}  \label{bound000}
\end{equation}%
 with $r\in (0, 1)$. It readily follows that 
 the sum $\theta =\rho
+\sum_{n=1}^{+\infty }\mathbf{P}^{n}\rho $ exists. The Poisson equation $\rho
=\theta -\mathbf{P}\theta $ is obvious. Finally, for any $x\in \mathbb{X},$
one gets $\left\vert \theta \left( x\right) -\rho \left( x\right)
\right\vert \leq \sum_{n=1}^{+\infty }\left\vert \mathbf{P}^{n}\rho \left(
x\right) \right\vert \leq c_{1}\sum_{n=1}^{+\infty }r ^{n}< +\infty ,$
which proves (\ref{bound001}).
\end{proof}

\begin{corollary}
\label{bound moment g}For any $p>2$, it holds 
$\sup_{n \geq 1} \sup_{x \in \mathbb X}  \mathbb{E}_{x}\left\vert \theta
\left( X_{n}\right) \right\vert ^{p}< \infty.$ 
Moreover $\displaystyle \left\Vert \mathbf{P}\theta \right\Vert _{\infty }=\sup_{x\in 
\mathbb{X}}\left\vert \mathbf{P}\theta \left( x\right) \right\vert < +\infty .$
\end{corollary}

\begin{proof}
The first assertion 
follows from the moment condition \textbf{M4} and the bound (\ref{bound001}).
Since $\mathbf{P}\theta \left( x\right) =\mathbb{E}_{x}\theta \left( X_{1}\right) $ the second
assertion follows from the first one.
\end{proof}

Let $\mathcal{F}_{0}$ be the trivial $\sigma $-algebra and $\mathcal{F}%
_{n}=\sigma \left\{ X_{k}:k\leq n\right\} $ for $n\geq 1.$ Let $\left(
M_{n}\right) _{n\geq 0}$ be defined by%
\begin{equation}
M_{0}=0\;\text{and}\;M_{n}=\sum_{k=1}^{n}\left\{ \theta \left( X_{k}\right) -%
\mathbf{P}\theta \left( X_{k-1}\right) \right\} ,\;n\geq 1.  \label{mart001}
\end{equation}%
By the Markov property we have $\mathbb{E}_{x}\left( \theta \left(
X_{k}\right) |\mathcal{F}_{k-1}\right) =\mathbf{P}\theta \left(
X_{k-1}\right) ,$ which implies that the sequence $\left( M_{n},\mathcal{F}%
_{n}\right) _{n\geq 0}$ is a $0$ mean $\mathbb{P}_{x}$-martingale.
The following lemma shows that the difference $\left( S_{n}-M_{n}\right)
_{n\geq 0}$ is bounded.

\begin{lemma}
\label{Martingale decomp}
Let $a=2\left\Vert \mathbf{P}\theta \right\Vert
_{\infty }.$  It holds  $ \displaystyle \sup_{n\geq 0}\left\vert S_{n}-M_{n}\right\vert \leq
a\ \mathbb P_x$-a.s. for any $x \in \mathbb X$.
\end{lemma}

\begin{proof}
Using the Poisson equation $\rho =\theta -\mathbf{P}\theta $ we obtain, for $%
k\geq 1,$%
\begin{eqnarray*}
\rho \left( X_{k}\right) =\theta \left( X_{k}\right) -\mathbf{P}\theta
\left( X_{k}\right) =
\left\{ \theta \left( X_{k}\right) -\mathbf{P}\theta \left(
X_{k-1}\right) \right\} +\left\{ \mathbf{P}\theta \left( X_{k-1}\right) -%
\mathbf{P}\theta \left( X_{k}\right) \right\} .
\end{eqnarray*}%
Summing in $k$ from $1$ to $n$ we get%
\begin{equation*}
S_{n}=\sum_{k=1}^{n}\rho \left( X_{k}\right) =\sum_{k=1}^{n}\left\{ \theta
\left( X_{k}\right) -\mathbf{P}\theta \left( X_{k-1}\right) \right\} +%
\mathbf{P}\theta \left( X_{0}\right) -\mathbf{P}\theta \left( X_{n}\right) ,
\end{equation*}%
which implies 
$\displaystyle 
S_{n}-M_{n}=\mathbf{P}\theta \left( X_{0}\right) -\mathbf{P}\theta \left(
X_{n}\right) ,\;n\geq 0. $
The assertion of the theorem now follows from Corollary \ref{bound moment g}.
\end{proof}

The following simple consequence of Burkholder's inequality will be used
repeatedly in the paper.

\begin{lemma}
\label{Lp boud martingales} For any $p>2$, it holds 
$\displaystyle \quad 
\sup_{n \geq 1} {1\over { n^{p/2} }}\sup_{x \in \mathbb X}  \mathbb{E}_{x}\left\vert M_n \right\vert ^{p}<+\infty.
$
\end{lemma}

\begin{proof}
Denoting $\xi _{k}=\theta \left( X_{k}\right) -\mathbf{P}\theta \left(
X_{k-1}\right) $ and applying Burkholder's inequality one gets%
\begin{equation}
 \mathbb{E}_{x}\left\vert M_{n}\right\vert ^{p} \leq
c_{p} \mathbb{E}_{x}\left\vert \sum_{k=1}^{n}\xi _{k}^{2}\right\vert^{p/2}.  \label{Lpb001}
\end{equation}%
By  H\"{o}lder's inequality, one gets 
$$ \mathbb{E}_{x}\left\vert \sum_{k=1}^{n}\xi _{k}^{2}\right\vert ^{p/2}\leq
n^{p/2-1}\mathbb{E}_{x}\sum_{k=1}^{n}\left\vert \xi _{k}\right\vert ^{p}\leq
n^{p/2}\sup_{1\leq k\leq n}\mathbb{E}_{x}\left\vert \xi _{k}\right\vert ^{p}$$
with $\displaystyle \sup_{1\leq k\leq n}\left( \mathbb{E}_{x}\left\vert \xi _{k}\right\vert
^{p}\right) ^{1/p}\leq 2\sup_{k\geq 1}\left( \mathbb{E}_{x}\left\vert \theta
\left( X_{k}\right) \right\vert ^{p}\right) ^{\frac{1}{p}}$ 
since 
\HL{$\mathbb{E}_{x}\left\vert \left( \mathbf{P}\theta\right) \left( X_{k-1}\right) \right\vert ^{p} 
\leq 
\mathbb{E}_{x}\left\vert \theta \left( X_{k}\right) \right\vert ^{p}$.}  
The desired inequality follows from Corollary \ref{bound moment g}.
\end{proof}


\section{Existence of the harmonic function\label{sec invar func}}

We start with a series of auxiliary assertions. Assume hypotheses \textbf{M1-\HL{M5}}. 
Recall that by Lemma \ref{Martingale decomp} 
the differences $S_{n}-M_{n}$ are bounded $\mathbb P_x$-a.s.\ for any $x \in \mathbb X$ and $n \geq 1.$ 
For any $y>0$ denote by $T_{y}$ the first time
when the martingale $\left( y+M_{n}\right) _{n\geq 1}$ exits $\mathbb{R}%
^{\ast }_{+}=\left( 0,\infty \right) ,$ 
\begin{equation*}
T_{y}=\min \left\{ n\geq 1:y+M_{n}\leq 0\right\} .
\end{equation*}%
The following assertion shows that the stopping times $\tau _{y}$ and $T_{y}$
are $\mathbb{P}_{x}$-a.s.\ finite.

\begin{lemma}
\label{Lemma finite tau}For any $x\in \mathbb{X}$ and $y>0$ it holds 
$$
\mathbb{P}_{x}\left( \tau _{y}< +\infty \right)= \mathbb{P}_{x}\left(
T_{y}< +\infty \right) =1.
$$
\end{lemma}

\begin{proof}
\HL{The fact that the first probability equals $1$} is a  consequence of the law of iterated logarithm for 
$\left( S_{n}\right) _{n\geq 0}$ established in Theorem 5 of \cite{LePage82}; 
\HL{the fact that the second probability equals $1$} follows  from Lemma \ref{Martingale decomp}.
\end{proof}

\begin{lemma}
\label{Lemma 1} There exist $c>0$ and  $\varepsilon _{0}>0$ such that for any $%
\varepsilon \in (0,\varepsilon _{0}),\;n\geq 1, x \in \mathbb X$ and $ y\geq  n^{1/2-\varepsilon },$%
\begin{equation}
\mathbb{E}_{x}\left( \left\vert y+M_{T_{y}}\right\vert
;\ T_{y}\leq n\right) \leq c\frac{y}{n^{\varepsilon }}.  \label{L1-001}
\end{equation}
\end{lemma}

\begin{proof}
Consider the event $A_{n}=\left\{ \max_{1\leq k\leq n}\left\vert \xi
_{k}\right\vert \leq  n^{1/2-2\varepsilon } \right\} ,$ where $\xi _{k}=\theta
\left( X_{k}\right) -\mathbf{P}\theta \left( X_{k-1}\right); $   one gets%
\begin{eqnarray}
\mathbb{E}_{x}\left( \left\vert y+M_{T_{y}}\right\vert ;\ T_{y}\leq n\right)
&=&\mathbb{E}_{x}\left( \left\vert y+M_{T_{y}}\right\vert ;\ T_{y}\leq
n, A_{n}\right)  \notag \\
&&+\ \mathbb{E}_{x}\left( \left\vert y+M_{T_{y}}\right\vert ;\ T_{y}\leq n, %
A_n^c\right)  \notag \\
&=&J_{1}\left( x,y\right) +J_{2}\left( x,y\right) .  \label{eq-lemma1-000}
\end{eqnarray}

We bound first $J_{1}\left( x,y\right) .$ Since $T_{y}$ is the first time
when $y+M_{T_{y}}$ becomes negative and the size of the jump $\xi _{T_{y}}$
of $y+M_{T_{y}}$ at time $T_{y}$ does not exceed $n^{1/2-2\varepsilon }$ on
the event $A_{n},$ it follows that
$
J_{1}\left( x,y\right) \leq n^{1/2-2\varepsilon }\mathbb{P}_{x}\left(
T_{y}\leq n;A_{n}\right) \leq n^{1/2-2\varepsilon }.
$
Therefore, for any $y\geq n^{1/2-\varepsilon }$ and $x\in \mathbb X$
\begin{equation}
J_{1}\left( x,y\right) \leq {y\over n^{\varepsilon }}.
\label{eq-lemma1-R1}
\end{equation}

Now we  bound $J_{2}\left( x,y\right).$ We set  $M_{n}^{\ast }=\max_{1\leq k\leq n}\left\vert
M_{k}\right\vert$;  since  $\left\vert y+M_{T_{y}}\right\vert \leq
y+M_{n}^{\ast }$ on the event $\left\{
T_{y}\leq n\right\} $,   it is clear that, for any $x\in \mathbb{X},$%
\begin{equation}
J_{2}\left( x,y\right) \leq y\mathbb{P}_{x}\left( A_n^c\right) +%
\mathbb{E}_{x}\left( M_{n}^{\ast };A_n^c\right) .
\label{eq-lemma1-001}
\end{equation}%
The probability $\mathbb{P}_{x}\left( A_n^c\right) $ in (\ref%
{eq-lemma1-001}) can be bounded as follows:%
\begin{eqnarray}
\mathbb{P}_{x}\left( A_n^c\right) &=&\mathbb{P}_{x}\left(
\max_{1\leq k\leq n}\left\vert \xi _{k}\right\vert >  n^{1/2-2\varepsilon 
}\right)  \notag \\
&\leq &\sum_{1\leq k\leq n}\mathbb{P}_{x}\left( \left\vert \xi
_{k}\right\vert >  n^{1/2-2\varepsilon  }\right)  \notag \\
&\leq & {1\over   n^{\left( 1/2-2\varepsilon \right) p}}\sum_{1\leq k\leq n}%
\mathbb{E}_{x}\left( \left\vert \xi_{k}\right\vert
^{p}\right)  \notag \\
&\leq &{2^p \over n^{\left( 1/2-2\varepsilon \right) p}}\sum_{1\leq k\leq n}%
\mathbb{E}_{x}\left( \left\vert \theta \left( X_{k}\right) \right\vert
^{p}\right)  \notag \\
&=&{c_p\over n^{p/2-1-2\varepsilon p}},  \label{PAbar bound}
\end{eqnarray}%
where the last inequality follows from Corollary \ref{bound moment g}.
Similarly, we have%
\begin{eqnarray}
\mathbb{E}_{x}\left( M_{n}^{\ast };A_n^c\right) &\leq &\mathbb{E}%
_{x}\left( M_{n}^{\ast };M_{n}^{\ast }>n^{1/2+2\varepsilon },
A_{n}^c\right) +n^{1/2+2\varepsilon }\mathbb{P}_{x}\left( A_{n}^c\right)  \notag \\
&\leq &\int_{n^{1/2+2\varepsilon }}^{+\infty }\mathbb{P}_{x}\left(
M_{n}^{\ast }>t\right) dt+2n^{1/2+2\varepsilon }\mathbb{P}_{x}\left( 
A_n^c\right),   \label{EZnbarAn}
\end{eqnarray}%
with
\begin{equation}
2 n^{1/2+2\varepsilon }\mathbb{P}_{x}\left( A_n^c\right)
<2yn^{3\varepsilon }\mathbb{P}_{x}\left( A_n^c\right) \leq
{4c_{p}y\over n^{p/2-1 -2\varepsilon p-3\varepsilon }}.
\label{eq-lemma1-002a}
\end{equation}%
Now we analyze the integral in (\ref{EZnbarAn}). Using Doob's maximal
inequality for martingales and Lemma \ref{Lp boud martingales} we get 
$\displaystyle \mathbb{P}_{x}\left( M_{n}^{\ast }>t\right) \leq \frac{1}{t^{p}}\mathbb{E}%
_{x}\left\vert M_{n}\right\vert ^{p}\leq  c_{p}\frac{n^{p/2}}{t^{p}}$, therefore%
\begin{eqnarray}
\int_{n^{1/2+2\varepsilon }}^{+\infty }\mathbb{P}_{x}\left( M_{n}^{\ast } >t\right) dt &\leq
& c_{p} n^{p/2}\int_{n^{1/2+2\varepsilon }}^{+\infty }\frac{1}{t^{p}}dt  \notag
\\ &\leq & c_{p} \frac{n^{p/2}}{n^{\left( 1/2+2\varepsilon \right) \left( p-1\right)
}}  \notag \\
 &\leq&
c_{p} \frac{y}{n^{2\varepsilon p-3\varepsilon}}.
\label{eq-lemma1-003}
\end{eqnarray}%
Taking (\ref{EZnbarAn}), (\ref{eq-lemma1-003}) and (\ref{eq-lemma1-002a})
altogether, one gets%
\begin{equation}
\mathbb{E}_{x}\left( M_{n}^{\ast };A_n^c\right) \leq 
c_{p} \Bigl(\frac{y}{n^{2\varepsilon p-3\varepsilon}}+ {y\over n^{p/2-1 -2\varepsilon p-3\varepsilon }}\Bigr).  \label{eq-lemma1-004}
\end{equation}%
Implementing the bounds (\ref{PAbar bound}) and (\ref{eq-lemma1-004}) in (%
\ref{eq-lemma1-001}), we obtain,%
\begin{equation}
J_{2}\left( x,y\right) \leq c_p\Bigl({y\over n^{p/2-1 -2\varepsilon p}}+
 \frac{y}{n^{2\varepsilon p-3\varepsilon}}+{ y\over n^{p/2-1 -2\varepsilon p-3\varepsilon }}\Bigr)
  \label{eq-lemma1-R2}
\end{equation}
Finally, from (\ref{eq-lemma1-000}), (\ref{eq-lemma1-R1}), (\ref%
{eq-lemma1-R2}), we get, for any $y\geq  n^{1/2-\varepsilon },$%
\begin{eqnarray*}
\mathbb{E}_{x}\left( \left\vert y+M_{T_{y}}\right\vert ;\ T_{y}\leq n\right)
&\leq &{y\over n^{\varepsilon }}+c_p\Bigl({y\over n^{p/2-1 -2\varepsilon p}}+
 \frac{y}{n^{2\varepsilon p-3\varepsilon}}+{ y\over n^{p/2-1 -2\varepsilon p-3\varepsilon }}\Bigr).
\end{eqnarray*}%
Choose $p>2.$ Then there exist $c>0$ and $\varepsilon _{0}>0$ such that for any $%
\varepsilon \in \left( 0,\varepsilon _{0}\right) $ and $y\geq n^{1/2-\varepsilon },$ 
\begin{equation*}
\mathbb{E}_{x}\left( \left\vert y+M_{T_{y}}\right\vert ;\ T_{y}\leq n\right)
\leq c{y\over n^{\varepsilon }}
\end{equation*}%
which proves  the lemma.
\end{proof}

Let $\varepsilon >0,$ $y>0.$ Consider the first time $\nu _{n}$
when $\left\vert y+M_{k}\right\vert $ exceeds $2 n^{1/2-\varepsilon }:$ 
\begin{equation}
\nu _{n}=\nu _{n,y,\varepsilon}=\min \left\{ k\geq 1:\left\vert
y+M_{k}\right\vert \geq 2 n^{1/2-\varepsilon }\right\} .  \label{nu n}
\end{equation}

\begin{lemma} \label{nu001}
There exists $\varepsilon _{1}>0$ such that for any 
$\varepsilon  \in (0,\varepsilon _{1}),$ $x\in \mathbb{X}$ and $y>0$ it holds 
$$\mathbb{P}_{x}\left( \lim_{n\rightarrow +\infty }\nu _{n,y,\varepsilon}= +\infty \right)=1.$$
\end{lemma}
\HL{\begin{proof}
Let $\Omega_0 = \cap_{k\geq 1} \left\{ \left| y+M_k \right| < +\infty \right\}.$ Note that,
by Lemma \ref{Lp boud martingales}, 
$\mathbb P_x \left( \Omega_0\right) =1,$ for any $x\in \mathbb X.$
Introduce the decreasing sequence of random sets $A_n = \left\{ k \geq 1: \left| y+M_k \right| >  n^{1/2-\varepsilon} \right\}.$
Let $A=\cap_{n\geq 1} A_n. $
It is easy to verify that for any $\omega \in \Omega_0$ we have $A(\omega)=\varnothing.$
Indeed, if $k(\omega) \in A(\omega)$ it holds $\left| y+M_{k(\omega)}(\omega) \right| = +\infty$
and therefore $\omega \not\in \Omega_0.$
Thus, for any $x \in \mathbb X$ we get $\mathbb P_x\left( A = \varnothing \right) =1$
and taking into account that
$\nu_{n}= \inf A_{n}$ is increasing the assertion follows. 
\end{proof}}
\begin{lemma} 
\label{Lemma 2}For any $\varepsilon \in (0,\frac{1}{2})$, there exists $c_\epsilon >0$ 
such that for any $n\geq 1, x \in \mathbb X$ and $y>0$%
\begin{equation*}
\mathbb{P}_{x}\left( \nu _{n,y,\varepsilon}>n^{1-\varepsilon }\right) \leq \exp \left( -c_{\varepsilon }n^{\varepsilon
}\right) .
\end{equation*}
\end{lemma}

\begin{proof}
Let $m=\left[ b^{2}n^{1-2\varepsilon }\right] $ and $K=\left[ n^{\varepsilon
}/b^{2}\right] ,$ where $b$ will be chosen later on.  By Lemma \ref{Martingale
decomp}, for any $x \in \mathbb X$  we have  $\left\vert S_{n}-M_{n}\right\vert \leq a\leq n^{1/2-\varepsilon }$
$\mathbb{P}_{x}$-a.s. with $a=2\left\Vert 
\mathbf{P}\theta \right\Vert _{\infty }.$ So, for $n$ sufficiently large and for any  $y>0,$ 
\begin{eqnarray}
\mathbb{P}_{x}\left( \nu _{n,y,\varepsilon }>n^{1-\varepsilon }\right)  &=&%
\mathbb{P}_{x}\left( \max_{1\leq k\leq n^{1-\varepsilon }}\left\vert
y+M_{k}\right\vert \leq 2 n^{1/2-\varepsilon }\right)   \notag \\
&\leq &\mathbb{P}_{x}\left( \max_{1\leq k\leq K}\left\vert
y+M_{km}\right\vert \leq 2 n^{1/2-\varepsilon }\right)   \notag \\
&\leq &\mathbb{P}_{x}\left( \max_{1\leq k\leq K}\left\vert
y+S_{km}\right\vert \leq 3 n^{1/2-\varepsilon }\right) .  \label{nu000}
\end{eqnarray}%
Using the Markov property, it follows that, for any $x\in \mathbb{X},$ 
\begin{eqnarray*}
&&\mathbb{P}_{x}\left( \max_{1\leq k\leq K}\left\vert y+S_{km}\right\vert
\leq 3 n^{1/2-\varepsilon }\right)  \\
&\leq &\mathbb{P}_{x}\left( \max_{1\leq k\leq K-1}\left\vert
y+S_{km}\right\vert \leq 3 n^{1/2-\varepsilon }\right) \sup_{z\in \mathbb{R}%
,\ x\in \mathbb{X}}\mathbb{P}_{x}\left( \left\vert z+S_{m}\right\vert \leq
3 n^{1/2-\varepsilon }\right) ,
\end{eqnarray*}%
from which iterating, we get%
\begin{equation}
\mathbb{P}_{x}\left( \max_{1\leq k\leq K}\left\vert y+S_{km}\right\vert \leq
3 n^{1/2-\varepsilon }\right) \leq \left( \sup_{y\in \mathbb{R},\ x\in 
\mathbb{X}}\mathbb{P}_{x}\left( \left\vert y+S_{m}\right\vert \leq
3 n^{1/2-\varepsilon }\right) \right) ^{K}.  \label{nu001}
\end{equation}%
Denote $\mathbb{B}_{y}\left( r\right) =\left\{ z:\left\vert y+z\right\vert
\leq r\right\} .$ Then, for any $x\in \mathbb{X}$ and $y\in \mathbb{R},$ 
\begin{equation*}
\mathbb{P}_{x}\left( \left\vert y+S_{m}\right\vert \leq 3 n^{1/2-\varepsilon
}\right) =\mathbb{P}_{x}\left( \frac{S_{m}}{\sqrt{m}}\in \mathbb{B}_{y/\sqrt{%
m}}\left( r_{n}\right) \right) ,
\end{equation*}%
where $r_{n}=3 n^{1/2-\varepsilon }/\sqrt{m}.$ Using the central limit
theorem for $S_{n}$ (Theorem \ref{CLT GLDR}) we have, as $n\rightarrow
\infty ,$%
\begin{equation}
\sup_{y\in \mathbb{R},\ x\in \mathbb{X}}\left\vert \mathbb{P}_{x}\left( 
\frac{S_{m}}{\sqrt{m}}\in \mathbb{B}_{y/\sqrt{m}}\left( r_{n}\right) \right)
-\int_{\mathbb{B}_{y/\sqrt{m}}\left( r_{n}\right) }\mathbb{\phi }_{\sigma
^{2}}\left( u\right) du\right\vert \rightarrow 0.  \label{nu002}
\end{equation}%
From (\ref{nu002}) we deduce that, as $n\rightarrow +\infty ,$%
\begin{equation*}
\sup_{y\in \mathbb{R},\ x\in \mathbb{X}}\mathbb{P}_{x}\left( \left\vert
y+S_{m}\right\vert \leq 3 n^{1/2-\varepsilon }\right) \leq \sup_{y\in 
\mathbb{R}}\int_{\mathbb{B}_{y/\sqrt{m}}\left( r_{n}\right) }\mathbb{\phi }%
_{\sigma ^{2}}\left( u\right) du+o\left( 1\right) .
\end{equation*}%
Since $r_{n}\leq c_{1}b^{-1},$ we get%
\begin{equation*}
\sup_{y\in \mathbb{R}}\int_{\mathbb{B}_{y/\sqrt{m}}\left( r_{n}\right) }%
\mathbb{\phi }_{\sigma ^{2}}\left( u\right) du\leq \int_{-r_{n}}^{r_{n}}%
\mathbb{\phi }_{\sigma ^{2}}\left( u\right) du\leq  c_{2}r_{n} %
\leq \frac{c_{3}}{b}.
\end{equation*}%
Choosing $b$ large, for some $q_{\varepsilon }<1$ and $n$ large enough, we
obtain%
\begin{equation*}
\sup_{y\in \mathbb{R},\ x\in \mathbb{X}}\mathbb{P}_{x}\left( \left\vert
y+S_{m}\right\vert \leq 3 n^{1/2-\varepsilon }\right) \leq q_{\varepsilon }.
\end{equation*}%
Implementing this bound in (\ref{nu001}) and using (\ref{nu000}) it follows
that%
\begin{equation*}
\sup_{y>0,\ x\in \mathbb{X}}\mathbb{P}_{x}\left( \nu _{n,y,\varepsilon}
>n^{1-\varepsilon }\right) \leq q_{\varepsilon }^{K}\leq \exp \left(
-c_{\varepsilon }n^{\varepsilon }\right) ,
\end{equation*}%
which proves the lemma.
\end{proof}

\begin{lemma}
\label{Lemma 3}There exists $c>0$ such that for any $\varepsilon \in \left( 0,\frac{1}{2}\right), 
n \geq 1,  x\in \mathbb{X}$ and $y>0$,
\begin{equation*}
\sup_{1\leq k \leq n} 
\mathbb{E}_{x}\left( \left\vert y+M_{k}\right\vert
;\nu _{n,y,\varepsilon}>n^{1-\varepsilon }\right) \leq c\left( 1+y\right)
\exp \left( -c_{\varepsilon }n^{\varepsilon }\right) 
\end{equation*}
for some $c_\varepsilon >0$.
\end{lemma}

\begin{proof}
By   Cauchy-Schwartz inequality, for any $n \geq 1, 1\leq k\leq n, x\in \mathbb{X}$ and $y>0,$%
\begin{equation*}
\mathbb{E}_{x}\left( \left\vert y+M_{k}\right\vert ;\nu _{n,y,\varepsilon}
>n^{1-\varepsilon }\right) \leq \mathbb{E}_{x}^{1/2}\left( \left\vert
y+M_{k}\right\vert ^{2}\right) \mathbb{P}_{x}^{1/2}\left( \nu_{n,y,\varepsilon}>n^{1-\varepsilon }\right) .
\end{equation*}%
By Minkowsky's inequality  and Lemma \ref{Lp boud martingales}
$$\mathbb{E}_{x}^{1/2}\left( \left\vert
y+M_{k}\right\vert ^{2}\right) \leq y+\sqrt{\mathbb{E}_{x}M_{k}^{2}}
\leq y+\left( \mathbb{E}_{x}M_{k}^{3}\right)
^{1/3}\leq y+c n^{1/2}.$$   The claim follows by Lemma \ref%
{Lemma 2}.
\end{proof}

\begin{lemma}
\label{Lemma 4}There exists $c>0$ such that for any $n \geq 1, x\in \mathbb{X}$ and $y>0,$
\begin{equation}
\mathbb{E}_{x}\left(
y+M_{n};\ T_{y}>n\right) \leq c\left( 1+y\right) .  \label{mmm000}
\end{equation}
\end{lemma}

\begin{proof}
First we prove that there exist $c>0$ and  $\varepsilon _{0}>0$ such that for any $x\in 
\mathbb{X},$ $\varepsilon \in \left( 0,\varepsilon _{0}\right) $ and $y\geq n^{1/2-\varepsilon },$ 
\begin{equation}
\mathbb{E}_{x}\left( y+M_{n};\ T_{y}>n\right) \leq \left( 1+\frac{c}{%
n^{\varepsilon }}\right) y.  \label{bound G}
\end{equation}
Since $\left(
M_{n},\mathcal{F}_{n}\right) _{n\geq 1}$ is a zero mean $\mathbb{P}_{x}$%
-martingale, we have $\mathbb{E}_{x}M_{n}=0$ and $\mathbb{E}_{x}\left(
y+M_{n};\ T_{y}\leq n\right) =\mathbb{E}_{x}\left( y+M_{T_{y}};\ T_{y}\leq
n\right) ;$ so%
\begin{eqnarray}
\mathbb{E}_{x}\left( y+M_{n};\ T_{y}>n\right) &=&\mathbb{E}_{x}\left(
y+M_{n}\right) -\mathbb{E}_{x}\left( y+M_{n};\ T_{y}\leq n\right)  \notag \\
&=&y-\mathbb{E}_{x}\left( y+M_{n};\ T_{y}\leq n\right)  \notag \\
&=&y-\mathbb{E}_{x}\left( y+M_{T_{y}};\ T_{y}\leq n\right)  \notag \\
&=&y+\mathbb{E}_{x}\left( \left\vert y+M_{T_{y}}\right\vert ;\ T_{y}\leq
n\right) .  \label{bound G 000}
\end{eqnarray}%
By Lemma \ref{Lemma 1}, there exist $c>0$ and $\varepsilon _{0}>0$
such that $\mathbb{E}_{x}\left( \left\vert y+M_{T_{y}}\right\vert ;\ T_{y}\leq
n\right) \leq \frac{c}{n^{\varepsilon }}y,$ 
for any $y\geq n^{1/2-\varepsilon }$ and
$\varepsilon \in \left(0,\varepsilon _{0}\right) .$
Implementing this inequality in (\ref{bound G
000}), we obtain (\ref{bound G}).

Now we show (\ref{mmm000}) for any $x\in \mathbb{X}$ and $y>0.$ We use a
recursive argument based on the Markov property coupled with the bound (\ref%
{bound G}). First we note that, for any $\varepsilon >0,$%
\begin{eqnarray}
\mathbb{E}_{x}\left( y+M_{n};\ T_{y}>n\right) &=&\mathbb{E}_{x}\left(
y+M_{n};\ T_{y}>n,\nu _{n,y,\varepsilon}\leq n^{1-\varepsilon }\right)  \notag \\
&&+\mathbb{E}_{x}\left( y+M_{n};\ T_{y}>n,\nu _{n,y,\varepsilon }
> n^{1-\varepsilon }\right)  \notag \\
&=&J_{1}\left( x,y\right) +J_{2}\left( x,y\right) .  \label{bound J1+J2}
\end{eqnarray}%
By Lemma \ref{Lemma 3}, for some $\varepsilon \in \left( 0,\varepsilon_{0}\right) ,$ 
\begin{equation*}
J_{2}\left( x,y\right) \leq c\left( 1+y\right) \exp \left( -c_{\varepsilon
}n^{\varepsilon }\right) .
\end{equation*}%
To control $J_{1}\left( x,y\right) ,$ write it in the form 
\begin{equation}
J_{1}\left( x,y\right) =\sum_{k=1}^{\left[ n^{1-\varepsilon }\right] }%
\mathbb{E}_{x}\left( y+M_{n};\ T_{y}>n,\ \nu _{n,y,\varepsilon}=k\right) .
\label{bound J2 000}
\end{equation}%
By the Markov property of the chain $\left( X_{n}\right)
_{n\geq 1},$%
\begin{eqnarray}
\mathbb{E}_{x}\left( y+M_{n};\ T_{y}>n,\ \nu _{n,y,\varepsilon }=k\right) 
&=&\int \mathbb{E}_{x^{\prime }}\left( y^{\prime }+M_{n-k};\ T_{y^{\prime
}}>n-k\right)  \notag \\
&&\times \mathbb{P}_{x}\left( X_{k}\in dx^{\prime },y+M_{k}\in dy^{\prime
};\ T_{y}>k,\ \nu _{n,y,\varepsilon }=k\right)  \notag \\
&=&\mathbb{E}_{x}\left( U_{n-k}\left( X_{k},y+M_{k}\right) ;\ T_{y}>k,\ \nu_{n,y,\varepsilon }=k\right) ,  \label{bound E}
\end{eqnarray}%
where $U_{m}\left( x,y\right) =\mathbb{E}_{x}\left( y+M_{m};\ T_{y}>m\right) $
for any $m\geq 1.$ Since 
\begin{equation*}
\left\{ \nu _{n,y,\varepsilon }=k\right\} \subset \left\{ \left\vert
y+M_{k}\right\vert \geq n^{1/2-\varepsilon }\right\} ,
\end{equation*}%
using (\ref{bound G}), on the event $\left\{ T_y>k, \ \nu _{n,y,\varepsilon }=k\right\} $ we have%
\begin{equation}
U_{n-k}\left( X_{k},y+M_{k}\right) \leq \left( 1+\frac{c}{\left( n-k\right)
^{\varepsilon }}\right) \left( y+M_{k}\right) .  \label{bound Sn J3}
\end{equation}%
Inserting (\ref{bound Sn J3}) into (\ref{bound E}),  we obtain%
\begin{eqnarray}
&&\mathbb{E}_{x}\left( y+M_{n};\ T_{y}>n,\ \nu _{n,y,\varepsilon }=k\right) 
\notag \\
&\leq &\left( 1+\frac{c}{\left( n-k\right) ^{\varepsilon }}\right) \mathbb{E}%
_{x}\left( y+M_{k};\ T_{y}>k,\ \nu _{n,y,\varepsilon }=k\right) .
\label{bound E001}
\end{eqnarray}%
Combining (\ref{bound E001}) and (\ref{bound J2 000}) it follows that, for $%
n $ sufficiently large,%
\begin{eqnarray*}
J_{1}\left( x,y\right) &\leq &\sum_{k=1}^{\left[ n^{1-\varepsilon }\right]
}\left( 1+\frac{c}{\left( n-k\right) ^{\varepsilon }}\right) \mathbb{E}_{x}\left( y+M_{k};\ T_{y}>k,\nu _{n,y,\varepsilon }=k\right) \\
&\leq &\left( 1+\frac{c_{\varepsilon }^{\prime}}{n^{\varepsilon }}\right) 
\sum_{k=1}^{\left[ n^{1-\varepsilon }\right] }\mathbb{E}_{x}\left( y+M_{k};\ T_{y}>k,\nu_{n,y,\varepsilon }=k\right) .
\end{eqnarray*}%
Since $\left( \left( y+M_{n}\right) 1_{\left\{ T_{y}>n\right\} }\right)
_{n\geq 1}$ is a submartingale, for any $x\in \mathbb{X}$ and $1\leq k\leq \left[ n^{1-\varepsilon }\right], $ 
\begin{equation*}
\mathbb{E}_{x}\left( y+M_{k};\ T_{y}>k,\nu _{n,y,\varepsilon }=k\right) \leq 
\mathbb{E}_{x}\left( y+M_{\left[ n^{1-\varepsilon }\right] };\ T_{y}>\left[
n^{1-\varepsilon }\right] ,\nu _{n,y,\varepsilon }=k\right) .
\end{equation*}%
This implies%
\begin{eqnarray*}
J_{1}\left( x,y\right) &\leq &\left( 1+\frac{c_{\varepsilon }^{\prime}}{%
n^{\varepsilon }}\right) \sum_{k=1}^{\left[ n^{1-\varepsilon }\right] }%
\mathbb{E}_{x}\left( y+M_{\left[ n^{1-\varepsilon }\right] };\ T_{y}>\left[
n^{1-\varepsilon }\right] ,\nu _{n,y,\varepsilon }=k\right) \\
&\leq &\left( 1+\frac{c_{\varepsilon }^{\prime}}{n^{\varepsilon }}\right) \mathbb{E}%
_{x}\left( y+M_{\left[ n^{1-\varepsilon }\right] };\ T_{y}>\left[
n^{1-\varepsilon }\right] \right) .
\end{eqnarray*}%
Implementing the bounds for $J_{1}\left( x,y\right) $ and $J_{2}\left(
x,y\right) $ in (\ref{bound J1+J2}) we obtain 
\begin{eqnarray}
\mathbb{E}_{x}\left( y+M_{n};\ T_{y}>n\right) &\leq &\left( 1+\frac{%
c_{\varepsilon }^{\prime}}{n^{\varepsilon }}\right) \mathbb{E}_{x}\left( y+M_{\left[
n^{1-\varepsilon }\right] };\ T_{y}>\left[ n^{1-\varepsilon }\right] \right) 
\notag \\
&&+c\left( 1+y\right) \exp \left( -c_{\varepsilon }n^{\varepsilon
}\right) .  \label{bound E002}
\end{eqnarray}%
Let $k_{j}=\left[ n^{\left( 1-\varepsilon \right) ^{j}}\right] $ for $j\geq
0.$ Note that $\left[ k_{j}^{1-\varepsilon }\right] =\left[ \left[ n^{\left(
1-\varepsilon \right) ^{j}}\right] ^{1-\varepsilon }\right] \leq \left[
n^{\left( 1-\varepsilon \right) ^{j+1}}\right] =k_{j+1}.$ Then, using the
bound (\ref{bound E002}) and the fact that $\left( \left( y+M_{n}\right)
1_{\left\{ T_{y}>n\right\} }\right) _{n\geq 1}$ is a submartingale, we get%
\begin{eqnarray*}
&&\mathbb{E}_{x}\left( y+M_{k_{1}};\ T_{y}>k_{1}\right) \\
&\leq &\left( 1+\frac{c_{\varepsilon }^{\prime}}{k_{1}^{\varepsilon }}\right) \mathbb{%
E}_{x}\left( y+M_{\left[ k_{1}^{1-\varepsilon }\right] };\ T_{y}>\left[
k_{1}^{1-\varepsilon }\right] \right) +c\left( 1+y\right) \exp \left(
-c_{\varepsilon }k_{1}^{\varepsilon }\right) \\
&\leq &\left( 1+\frac{c_{\varepsilon }^{\prime}}{k_{1}^{\varepsilon }}\right) \mathbb{%
E}_{x}\left( y+M_{k_{2}};\ T_{y}>k_{2}\right) +c\left( 1+y\right) \exp \left(
-c_{\varepsilon }k_{1}^{\varepsilon }\right) .
\end{eqnarray*}%
Substituting this bound in (\ref{bound E002}) and continuing in the same
way, after $m$ iterations, we obtain 
\begin{equation}
\mathbb{E}_{x}\left( y+M_{n};\ T_{y}>n\right) \leq A_{m}\left( \mathbb{E}%
_{x}\left( y+M_{k_{m}};\ T_{y}>k_{m}\right) +c\left( 1+y\right) B_{m}\right) ,
\label{ineq iter}
\end{equation}%
where%
\begin{equation}
A_{m}=\prod_{j=1}^{m}\left( 1+\frac{c^\prime_{\varepsilon }}{k_{j-1}^{\varepsilon }}%
\right)  \label{Am000}
\end{equation}%
and%
\begin{equation}
B_{m}=\sum_{j=1}^{m}\exp \left( -c_{\varepsilon }k_{j-1}^{\varepsilon }\right) .  \label{Bm000}
\end{equation}%
Letting $\alpha _{j}=n^{-\varepsilon \left( 1-\varepsilon \right) ^{j-1}},$
we obtain%
\begin{equation}
A_{m}\leq \exp \left( 2^{\varepsilon }c^\prime_{\varepsilon }\sum_{j=1}^{m}\alpha
_{j}\right) .  \label{Am001}
\end{equation}%
We choose $m=m\left( n\right) $ such that $k_{m}=\left[ n^{\left(
1-\varepsilon \right) ^{m}}\right] \leq n_{0} \leq k_{m-1},$ where $n_{0}
\geq 1$ is a constant. Let $q_{0}=n_{0}^{-\varepsilon ^{2}}.$ Note that, for
any $j$ satisfying $0\leq j\leq m,$ it holds%
\begin{equation*}
\frac{\alpha _{j}}{\alpha _{j+1}}=n^{-\varepsilon ^{2}\left( 1-\varepsilon
\right) ^{j-1}}\leq n^{-\varepsilon ^{2}\left( 1-\varepsilon \right)
^{m-1}}\leq n_{0}^{-\varepsilon ^{2}}=q_{0}<1.
\end{equation*}%
Therefore 
\begin{equation*}
\alpha _{j}\leq q_{0}^{m-j}\alpha _{m}=q_{0}^{m-j}n^{-\varepsilon \left(
1-\varepsilon \right) ^{m-1}}\leq q_{0}^{m-j}n_{0}^{-\varepsilon }.
\end{equation*}%
This implies%
\begin{equation}
\sum_{j=1}^{m}\alpha _{j}\leq n_{0}^{-\varepsilon
}\sum_{j=1}^{m}q_{0}^{m-j}\leq \frac{n_{0}^{-\varepsilon }}{%
1-n_{0}^{-\varepsilon ^{2}}}.  \label{Am002}
\end{equation}

In the same way we can check that 
\begin{equation}
B_{m}\leq c_{1}\sum_{j=1}^{m}n^{-\varepsilon \left( 1-\varepsilon \right)
^{j-1}}\leq c_{2}\frac{n_{0}^{-\varepsilon }}{1-n_{0}^{-\varepsilon ^{2}}}.
\label{Bm002}
\end{equation}%
The assertion of the lemma follows from (\ref{ineq iter}), (\ref{Am001}), (\ref{Am002}) and (\ref{Bm002}) taking into account that 
\begin{eqnarray*}
\mathbb{E}_{x}\left( y+M_{k_{m}};\ T_{y}>k_{m}\right) &=&\mathbb{E}_{x}\left(
y+M_{n_{0}};\ T_{y}>n_{0}\right) \\
&\leq &\mathbb{E}_{x}\left( y+\left\vert M_{n_{0}}\right\vert \right) \\
&\leq &y+c.
\end{eqnarray*}
\end{proof}

\begin{corollary}
\label{corr to lemma4}
There exists $c>0$ such that for any $n\geq 1, x \in \mathbb X$ and $y>0$,
\begin{equation*}
\mathbb{E}_{x}\left( y+S_{n};\tau
_{y}>n\right) \leq c\left( 1+y\right)
\end{equation*}%
and%
\begin{equation*}
\mathbb{E}_{x}\left( y+M_{n};\tau
_{y}>n\right) \leq c\left( 1+y\right) .
\end{equation*}
\end{corollary}

\begin{proof}
Let $a=\left\Vert \mathbf{P}\theta \right\Vert _{\infty }$ and $x\in \mathbb{X}$. 
By Lemma \ref{Martingale decomp}, $\sup_{n\geq
0}\left\vert S_{n}-M_{n}\right\vert \leq a,$ $\mathbb{P}_{x}$-a.s., which
implies $\mathbb{P}_{x}\left( \tau _{y}\leq T_{y+a}\right) =1.$ Since $%
y+a+M_{n}\geq y+S_{n}>0$ on $\left\{ \tau _{y}>n\right\} ,$ $\mathbb{P}_{x}$
-a.s., for any $x\in \mathbb{X},$ using Lemma \ref{Lemma 4}, we get,%
\begin{eqnarray}
\mathbb{E}_{x}\left( y+a+M_{n};\tau _{y}>n\right) &\leq &\mathbb{E}%
_{x}\left( y+a+M_{n};\ T_{y+a}>n\right)  \notag \\
&\leq &c_{1}\left( 1+y+a\right)  \notag \\
&\leq &c_{2}\left( 1+y\right) .  \label{M plus a}
\end{eqnarray}%
Using the bound (\ref{M plus a}) it follows that%
\begin{equation*}
\mathbb{E}_{x}\left( y+S_{n};\tau _{y}>n\right) \leq \mathbb{E}_{x}\left(
y+a+M_{n};\tau _{y}>n\right) \leq c_{2}\left( 1+y\right)
\end{equation*}%
and%
\begin{equation*}
\mathbb{E}_{x}\left( y+M_{n};\tau _{y}>n\right) \leq \mathbb{E}_{x}\left(
y+a+M_{n};\tau _{y}>n\right) \leq c_{2}\left( 1+y\right) .
\end{equation*}
\end{proof}

\begin{lemma}
\label{Lemma 5}There exists $c>0$ such that for any $ x \in \mathbb X$ and $y>0$,
\begin{equation*}
\mathbb{E}_{x}\left\vert y+M_{T_{y}}\right\vert \leq
c\left( 1+y\right) < +\infty .
\end{equation*}
\end{lemma}

\begin{proof}
Let $x\in \mathbb{X}$. As $\left( M_{n},\mathcal{F}%
_{n}\right) _{n\geq 1}$ is a zero mean $\mathbb{P}_{x}$-martingale, we have $%
\mathbb{E}_{x}M_{n}=0$ and $\mathbb{E}_{x}\left( y+M_{n};\ T_{y}\leq n\right) =%
\mathbb{E}_{x}\left( y+M_{T_{y}};\ T_{y}\leq n\right) ;$ so, since $y>0,$ 
\begin{eqnarray*}
\mathbb{E}_{x}\left\vert y+M_{T_{y}\wedge n}\right\vert &=&\mathbb{E}%
_{x}\left( y+M_{n};\ T_{y}>n\right) -\mathbb{E}_{x}\left(
y+M_{T_{y}};\ T_{y}\leq n\right) \\
&=&\mathbb{E}_{x}\left( y+M_{n};\ T_{y}>n\right) -\mathbb{E}_{x}\left(
y+M_{n};\ T_{y}\leq n\right) \\
&=&2\mathbb{E}_{x}\left( y+M_{n};\ T_{y}>n\right) -y \\
&\leq &2\mathbb{E}_{x}\left( y+M_{n};\ T_{y}>n\right) .
\end{eqnarray*}%
Taking into account that by Lemma \ref{Lemma 4}, $\mathbb{E}_{x}\left(
y+M_{n};\ T_{y}>n\right) \leq c\left( 1+y\right) ,$ we get%
\begin{equation*}
\mathbb{E}_{x}\left\vert y+M_{T_{y}\wedge n}\right\vert \leq 2c\left(
1+y\right) .
\end{equation*}%
Since%
\begin{equation*}
\mathbb{E}_{x}\left( \left\vert y+M_{T_{y}}\right\vert ;\ T_{y}\leq n\right)
\leq \mathbb{E}_{x}\left\vert y+M_{T_{y}\wedge n}\right\vert \leq 2c\left(
1+y\right) ,
\end{equation*}%
by Lebesgue's monotone convergence theorem it follows that 
\begin{equation*}
\mathbb{E}_{x}\left\vert y+M_{T_{y}}\right\vert =\lim_{n\rightarrow +\infty }%
\mathbb{E}_{x}\left( \left\vert y+M_{T_{y}}\right\vert ;\ T_{y}\leq n\right)
\leq 2c \left( 1+y\right) < +\infty .
\end{equation*}
\end{proof}

\begin{corollary}
\label{Intergrab of Mtau}There exists $c>0$ such that for any $x \in \mathbb X$ and $y>0$,
\begin{equation*}
\mathbb{E}_{x}\left( \left\vert y+S_{\tau
_{y}}\right\vert \right) \leq c\left( 1+y\right) < +\infty
\end{equation*}%
and 
\begin{equation*}
\mathbb{E}_{x}\left( \left\vert y+M_{\tau
_{y}}\right\vert \right) \leq c\left( 1+y\right) < +\infty .
\end{equation*}
\end{corollary}

\begin{proof}
For any $x\in \mathbb{X}$ and $y>0,$%
\begin{equation}
\mathbb{E}_{x}\left\vert y+S_{\tau _{y}\wedge n}\right\vert =\mathbb{E}%
_{x}\left( y+S_{n};\tau _{y}>n\right) -\mathbb{E}_{x}\left( y+S_{\tau
_{y}};\tau _{y}\leq n\right) .  \label{integ001}
\end{equation}%
By Lemma \ref{Martingale decomp} we have $\sup_{n\geq 0}\left\vert
S_{n}-M_{n}\right\vert \leq 2\left\Vert \mathbf{P}\theta \right\Vert
_{\infty }=a< +\infty .$ Note also that $\mathbb{E}_{x}M_{n}=0$ and $\mathbb{E}%
_{x}\left( y+M_{n};\tau _{y}\leq n\right) =\mathbb{E}_{x}\left( y+M_{\tau
_{y}};\tau _{y}\leq n\right) $; therefore the second term is bounded as follows:%
\begin{eqnarray*}
-\mathbb{E}_{x}\left( y+S_{\tau _{y}};\tau _{y}\leq n\right) &\leq &-\mathbb{%
E}_{x}\left( y+M_{\tau _{y}};\tau _{y}\leq n\right) +a \\
&=&-\mathbb{E}_{x}\left( y+M_{n};\tau _{y}\leq n\right) +a \\
&=&\mathbb{E}_{x}\left( y+M_{n};\tau _{y}>n\right) -\mathbb{E}_{x}\left(
y+M_{n}\right) +a \\
&\leq &\mathbb{E}_{x}\left( y+M_{n};\tau _{y}>n\right) +a \\
&\leq &\mathbb{E}_{x}\left( y+S_{n};\tau _{y}>n\right) +2a.
\end{eqnarray*}%
Substituting this bound in (\ref{integ001}) we get%
\begin{equation*}
\mathbb{E}_{x}\left\vert y+S_{\tau _{y}\wedge n}\right\vert \leq 2\mathbb{E}%
_{x}\left( y+S_{n};\tau _{y}>n\right) +2a,
\end{equation*}%
from which, by Corollary \ref{corr to lemma4}, we obtain%
\begin{equation*}
\mathbb{E}_{x}\left\vert y+S_{\tau _{y}\wedge n}\right\vert \leq c\left(
1+y\right) +2a.
\end{equation*}%
Since%
\begin{equation*}
\mathbb{E}_{x}\left( \left\vert y+S_{\tau _{y}}\right\vert ;\tau
_{y}>n\right) \leq \mathbb{E}_{x}\left\vert y+S_{\tau _{y}\wedge
n}\right\vert \leq c\left( 1+y\right) +2a,
\end{equation*}
by Lebesgue's monotone convergence theorem it follows that 
\begin{equation*}
\mathbb{E}_{x}\left\vert y+S_{\tau _{y}}\right\vert =\lim_{n\rightarrow
\infty }\mathbb{E}_{x}\left( \left\vert y+S_{\tau _{y}}\right\vert ;\tau
_{y}>n\right) \leq c\left( 1+y\right) +2a< +\infty .
\end{equation*}%
The second assertion follows from the first one since $\mathbb{P}_{x}\left(
\sup_{n\geq 0}\left\vert S_{n}-M_{n}\right\vert \leq a\right) =1.$
\end{proof}

\begin{lemma}
\label{Lemma 6}Uniformly in $x\in \mathbb{X},$%
\begin{equation*}
\lim_{y\rightarrow +\infty }\frac{1}{y}\lim_{n\rightarrow +\infty }\mathbb{E}%
_{x}\left( y+M_{n};\ T_{y}>n\right) =1.
\end{equation*}
\end{lemma}

\begin{proof}
Let $x\in \mathbb{X}$ and $y>0.$ Let $n_{0}$ be a constant and
$m=m\left( n\right) $ is such that $\left[ n^{\left( 1-\varepsilon \right)^{m}}\right] =n_{0}.$ 
Consider the sequence $k_{j}=\left[ n^{\left( 1-\varepsilon \right) ^{j}}\right] ,$ $j=0,\dots,m.$  
From (\ref{ineq iter}) it
follows%
\begin{equation*}
\mathbb{E}_{x}\left( y+M_{n};\ T_{y}>n\right) \leq A_{m}\left( \mathbb{E}%
_{x}\left( y+M_{n_{0}};\ T_{y}>n_{0}\right) +c\left( 1+y\right) B_{m}\right) ,
\end{equation*}%
where $A_{m}$ and $B_{m}$ are defined by (\ref{Am000}) and (\ref{Bm000}).
Let $\delta >0.$ From (\ref{Am001}), (\ref{Am002}) and (\ref{Bm002}),
choosing $n_{0}$ sufficiently large, one gets $A_{m}\leq 1+\delta $ and $%
B_{m}\leq \delta $ uniformly in $m$ (and thus in $n$ sufficiently large).
This gives%
\begin{equation}
\mathbb{E}_{x}\left( y+M_{n};\ T_{y}>n\right) \leq \left( 1+\delta \right)
\left( \mathbb{E}_{x}\left( y+M_{n_{0}};\ T_{y}>n_{0}\right) +c\left(
1+y\right) \delta \right) .  \label{MMM002}
\end{equation}%
Since $\left( \left( y+M_{n}\right) 1_{\left\{ T_{y}>n\right\} }\right)
_{n\geq 1}$ is a submartingale, the sequence $\mathbb{E}_{x}\left(
y+M_{n};\ T_{y}>n\right) $ is increasing (and bounded by Lemma \ref{Lemma 4}): 
it thus converges as $n\rightarrow +\infty $ and one gets%
\begin{equation*}
\lim_{n\rightarrow +\infty }\mathbb{E}_{x}\left( y+M_{n};\ T_{y}>n\right) \leq
\left( 1+\delta \right) \left( \mathbb{E}_{x}\left(
y+M_{n_{0}};\ T_{y}>n_{0}\right) +c\left( 1+y\right) \delta \right) .
\end{equation*}%
From (\ref{bound G 000}) we have the lower bound $\mathbb{E}_{x}\left(
y+M_{n};\ T_{y}>n\right) \geq y;$  
we obtain%
\begin{equation*}
y\leq \lim_{n\rightarrow +\infty }\mathbb{E}_{x}\left( y+M_{n};\ T_{y}>n\right)
\leq \left( 1+\delta \right) \left( y+\mathbb{E}%
_{x}\left\vert M_{n_{0}}\right\vert+c\left( 1+y\right) \delta \right)
\end{equation*}%
and the claim follows since $\delta >0$ is arbitrary.
\end{proof}

For any $x\in \mathbb{X}$ denote%
\begin{equation*}
V\left( x,y\right) =\left\{ 
\begin{array}{cc}
-\mathbb{E}_{x}M_{\tau _{y}} & \text{if\ }y>0, \\ 
0 & \text{if\ }y\leq 0.%
\end{array}%
\right.
\end{equation*}%
The following proposition presents some properties of the function $V.$

\begin{proposition}
\label{PROP func V for S}The function $V$ satisfies:

\noindent 1. For any $y>0$ and $x\in \mathbb{X},$%
\begin{eqnarray*}
V\left( x,y\right) &=&\lim_{n\rightarrow +\infty }\mathbb{E}_{x}\left(
y+M_{n};\tau _{y}>n\right) =\lim_{n\rightarrow +\infty }\mathbb{E}_{x}\left(
y+S_{n};\tau _{y}>n\right).
\end{eqnarray*}

\noindent 2. For any $y>0$ and $x\in \mathbb{X},$%
\begin{equation*}
0\vee \left( y-a\right) \leq V\left( x,y\right) \leq c\left( 1+y\right) ,
\end{equation*}%
where $a=2\left\Vert \mathbf{P}\theta \right\Vert _{\infty }.$

\noindent 3. For any $x\in \mathbb{X}, \quad 
\lim_{y\rightarrow +\infty }\frac{V\left( x,y\right) }{y}=1.$

\noindent 4. For any $x\in \mathbb{X},$ the function $V\left( x,\cdot
\right) $ is increasing.
\end{proposition}

\begin{proof}
Let $x\in \mathbb{X}$ and $y>0$. Since $\left( M_{n},\mathcal{F}%
_{n}\right) _{n\geq 1}$ is a zero mean $\mathbb{P}_{x}$-martingale,  
\begin{eqnarray}
\mathbb{E}_{x}\left( y+M_{n};\tau _{y}>n\right) &=&\mathbb{E}_{x}\left(
y+M_{n}\right) -\mathbb{E}_{x}\left( y+M_{n};\tau _{y}\leq n\right)  \notag
\\
&=&y-\mathbb{E}_{x}\left( y+M_{\tau _{y}};\tau _{y}\leq n\right) .
\label{VVV001}
\end{eqnarray}

\emph{Proof of the claim 1.} According to Corollary \ref{Intergrab of Mtau}
one gets $\sup_{x\in \mathbb{X}}\mathbb{E}_{x}\left\vert y+M_{\tau
_{y}}\right\vert \leq c\left( 1+y\right) < +\infty $; thus, by Lebesgue's
dominated convergence theorem, for any $x\in \mathbb{X},$%
\begin{equation*}
\lim_{n\rightarrow +\infty }\mathbb{E}_{x}\left( y+M_{\tau _{y}};\tau
_{y}\leq n\right) =\mathbb{E}_{x}\left( y+M_{\tau _{y}}\right) =y-V\left(
x,y\right) .
\end{equation*}%
Therefore, from (\ref{VVV001}), it follows%
\begin{equation*}
\lim_{n\rightarrow +\infty }\mathbb{E}_{x}\left( y+M_{n};\tau _{y}>n\right)
=y-\mathbb{E}_{x}\left( y+M_{\tau _{y}}\right) =V\left( x,y\right) .
\end{equation*}%
Since $\left\vert S_{n}-M_{n}\right\vert \leq 2\left\Vert \mathbf{P}\theta
\right\Vert _{\infty }$ and $\lim_{n\rightarrow +\infty }\mathbb{P}_{x}\left(
\tau _{y}>n\right) =0$ one obtains%
\begin{equation*}
\lim_{n\rightarrow +\infty }\mathbb{E}_{x}\left( y+S_{n};\tau _{y}>n\right)
=V\left( x,y\right) .
\end{equation*}

Taking into account that $\left( y+S_{n}\right) 1_{\left\{ \tau
_{y}>n\right\} }\geq 0,$ we have $V\left( x,y\right) \geq 0,$ which proves
the first claim.

\emph{Proof of the claim 2.} Corollary \ref{corr to lemma4} implies that for
any $x\in \mathbb{X},$ $y>0$ and $n\geq 1,$ 
\begin{equation*}
\mathbb{E}_{x}\left( y+S_{n};\tau _{y}>n\right) \leq c\left( 1+y\right) .
\end{equation*}%
Taking the limit as $n\rightarrow +\infty ,$ we obtain $V\left( x,y\right)
\leq c\left( 1+y\right) ,$ which proves the upper bound. From (\ref{VVV001}%
), taking into account the bound $\left\vert S_{n}-M_{n}\right\vert \leq
2\left\Vert \mathbf{P}\theta \right\Vert _{\infty }=a,$ we get%
\begin{equation*}
\mathbb{E}_{x}\left( y+M_{n};\tau _{y}>n\right) \geq y-\mathbb{E}_{x}\left(
y+S_{\tau _{y}};\tau _{y}\leq n\right) -a\geq y-a.
\end{equation*}%
The expected lower bound follows letting $n\rightarrow +\infty .$

\emph{Proof of the claim 3. }From the claim 2 it follows that $%
\lim_{y\rightarrow +\infty }\frac{V\left( x,y\right) }{y}\geq 1.$ Let $%
a=2\left\Vert \mathbf{P}\theta \right\Vert _{\infty }$ and $y>0.$ Since $%
\tau _{y}\leq T_{y+a},$ we have%
\begin{equation*}
\mathbb{E}_{x}\left( y+S_{n};\tau _{y}>n\right) \leq \mathbb{E}_{x}\left(
y+M_{n}+a;\tau _{y}>n\right) \leq \mathbb{E}_{x}\left(
y+a+M_{n};\ T_{y+a}>n\right) .
\end{equation*}%
Taking into account the claim 1 and Lemma \ref{Lemma 6}, we obtain $%
\lim_{y\rightarrow +\infty }\frac{V\left( x,y\right) }{y}\leq 1.$

\emph{Proof of the claim 4.}\textbf{\ }It is clear that $y\leq y^{\prime }$
implies $\tau _{y}\leq \tau _{y^{\prime }}.$ Therefore%
\begin{equation*}
\mathbb{E}_{x}\left( y+S_{n};\tau _{y}>n\right) \leq \mathbb{E}_{x}\left(
y^{\prime }+S_{n};\tau _{y}>n\right) \leq \mathbb{E}_{x}\left( y^{\prime
}+S_{n};\tau _{y^{\prime }}>n\right) .
\end{equation*}%
Taking the limit as $n\rightarrow +\infty $ one gets the claim 4.
\end{proof}

In the following proposition we prove that $V$ is $\mathbf{Q}_{+}$-harmonic.

\begin{proposition}
\label{PROP harmonic func}For any $x\in \mathbb{X}$ and $y>0$ it holds%
\begin{equation*}
\mathbf{Q}_{+}V\left( x,y\right) =\mathbb{E}_{x}\left( V\left(
X_{1},y+S_{1}\right) ;\tau _{y}>1\right) =V\left( x,y\right) >0.
\end{equation*}
\end{proposition}

\begin{proof}
Let $x\in \mathbb{X}$ and $y>0$ and set $V_{n}\left( x,y\right) =\mathbb{E}%
_{x}\left( y+S_{n};\tau _{y}>n\right) ,$ for any $n\geq 1.$ By the Markov
property we have%
\begin{eqnarray*}
V_{n+1}\left( x,y\right) &=&\mathbb{E}_{x}\left( y+S_{n+1};\tau
_{y}>n+1\right) \\
&=&\mathbb{E}_{x}\left( \left( V_{n}\left( X_{1};y+S_{1}\right) \right)
;\tau _{y}>1\right) .
\end{eqnarray*}%
By Corollary \ref{corr to lemma4}, we have%
\begin{equation*}
\sup_{x\in \mathbb{X}}V_{n}\left( x,y\right) \leq \sup_{x\in \mathbb{X}%
}\sup_{n}\mathbb{E}_{x}\left( y+S_{n};\tau _{y}>n\right) \leq c\left(
1+y\right) .
\end{equation*}%
This implies that $V_{n}\left( X_{1};y+S_{1}\right) 1_{\left\{ \tau
_{y}>1\right\} }$ is dominated by $c\left( 1+y+S_{1}\right) 1_{\left\{ \tau
_{y}>1\right\} }$ which is integrable. Taking the limit as $n\rightarrow +\infty ,$ 
by Lebesgue's dominated convergence theorem, we get%
\begin{eqnarray}
V\left( x,y\right) &=&\lim_{n\rightarrow +\infty }\mathbb{E}_{x}\left(
V_{n}\left( X_{1};y+S_{1}\right) ;\tau _{y}>1\right)  \notag \\
&=&\mathbb{E}_{x}\left( \lim_{n\rightarrow +\infty }V_{n}\left(
X_{1};y+S_{1}\right) ;\tau _{y}>1\right)  \notag \\
&=&\mathbb{E}_{x}V\left( X_{1},y+S_{1}\right) 1_{\left\{ \tau _{y}>1\right\}
}  \notag \\
&=&\mathbf{Q}_{+}V\left( x,y\right) .  \label{V001a}
\end{eqnarray}

To prove that $V$ is strictly positive on $\mathbb{X\times R}^{\ast }_{+}$
we first iterate (\ref{V001a}): for any $n\geq 1,$%
\begin{equation}
V\left( x,y\right) =\mathbf{Q}_{+}^{n}V\left( x,y\right) =\mathbb{E}%
_{x}\left( V\left( X_{n},y+S_{n}\right) ;\tau _{y}>n\right) .  \label{V002}
\end{equation}%
Let $\varepsilon >0.$ From (\ref{V002}) using the claim 2 of Proposition \ref%
{PROP func V for S} we conclude that, for any $x\in \mathbb{X}$ and $n\geq
1, $ $\mathbb{P}_{x}$-a.s.%
\begin{eqnarray*}
&&\mathbb{E}_{x}\left( V\left( X_{n},y+S_{n}\right) ;\tau _{y}>n\right) \\
&\geq &\mathbb{E}_{x}\left( 0\vee \left( y+S_{n}-a\right) \right) 1\left(
\tau _{y}>n\right) \\
&\geq &\mathbb{\varepsilon P}_{x}\left\{
y+S_{1}>0,\dots,y+S_{n-1}>0,y+S_{n}>a+\varepsilon \right\} .
\end{eqnarray*}%
According to condition \textbf{P5} there exists $\delta >0$ such that 
$q_{\delta }:=\inf_{x\in \mathbb{X}}\mathbb{P}_{x}\left( \rho \left(
X_{1}\right) \geq \delta \right) >0.$ Choose $n$ sufficiently large such
that $n\delta >a+\varepsilon .$ Then%
\begin{eqnarray*}
&&\mathbb{P}_{x}\left( y+S_{1}>0,\dots,y+S_{n-1}>0,y+S_{n}>a+\varepsilon
\right) \\
&\geq &\mathbb{P}_{x}\left( y+S_{1}>\delta ,y+S_{2}>2\delta,\dots,y+S_{n}>n\delta \right) .
\end{eqnarray*}%
By the Markov property%
\begin{eqnarray}
&&\mathbb{P}_{x}\left( y+S_{1}>\delta ,y+S_{2}>2\delta ,\dots,y+S_{n}>n\delta
\right)  \notag \\
&=&\int_{\delta }^{+\infty }\int_{\mathbb{X}} \dots \int_{\left( n-1\right)
\delta }^{+\infty }\int_{\mathbb{X}}\int_{n\delta }^{+\infty }\int_{\mathbb{X}}%
\mathbf{Q}\left( x_{n-1},y_{n-1},dx_{n},dy_{n}\right)  \notag \\
&&\times \mathbf{Q}\left( x_{n-2},y_{n-2},dx_{n-1},dy_{n-1}\right) \dots%
\mathbf{Q}\left( x,y,dx_{1},dy_{1}\right) ,  \label{LB001}
\end{eqnarray}%
where%
\begin{equation*}
\mathbf{Q}\left( x,y,dx^{\prime },dy^{\prime }\right) =\mathbb{P}_{x}\left(
X_{1}\in dx^{\prime },y+\rho \left( X_{1}\right) \in dy^{\prime }\right) .
\end{equation*}%
For any $1\leq m\leq n,$ $y_{m-1}\geq \left( m-1\right) \delta $ and $x\in 
\mathbb{X}$ we have%
\begin{eqnarray*}
\int_{m\delta }^{+\infty }\int_{\mathbb{X}}\mathbf{Q}\left(
x_{m-1},y_{m-1},dx_{m},dy_{m}\right) &=&\mathbb{P}_{x_{m-1}}\left( \rho
\left( X_{1}\right) \geq m\delta -y_{m-1}\right) \\
&\geq &\inf_{x\in \mathbb{X}}\mathbb{P}_{x}\left( \rho \left( X_{1}\right)
\geq \delta \right) \\
&=&q_{\delta }>0.
\end{eqnarray*}%
Inserting consecutively these bounds in (\ref{LB001}), it readily follows
that%
\begin{equation*}
\mathbb{P}_{x}\left( y+S_{1}>\delta ,y+S_{2}>2\delta ,\dots,y+S_{n}>n\delta
\right) \geq q_{\delta }^{n}>0,
\end{equation*}%
which proves that $V$ is strictly positive on \HL{$\mathbb{X\times R}^{\ast}_{+}.$}
\end{proof}

\section{Coupling argument and proof of Theorem \protect\ref{Th asympt tau}}

Let $\left( B_{t}\right) _{t\geq 0}$ be a standard Brownian motion on the
probability space $\left( \Omega ,\mathcal{F},\mathbf{Pr}\right) .$ For any $%
y>0$ define the exit time%
\begin{equation*}
\tau _{y}^{bm}=\inf \left\{ t\geq 0:y+\sigma B_{t}<0\right\} ,
\end{equation*}%
where $\sigma >0$ is given by (\ref{sigma001}). The following
well known formulas are due to Levy \cite{Levy37} (Theorem 42.I, pp.
194-195).

\begin{lemma}
\label{lemma tauBM} Let $y>0$. The stopping time $\tau _{y}^{bm}$ has the following
properties:

\noindent 1. For any $n\geq 1$,%
\begin{equation}
\mathbf{Pr}\left( \tau _{y}^{bm}>n\right) =\mathbf{Pr}\left( \sigma
\inf_{0\leq u\leq n}B_{u}>-y\right) =\frac{2}{\sqrt{2\pi n}\sigma }%
\int_{0}^{y}e^{-\frac{s^{2}}{2n\sigma ^{2}}}ds.  \label{levy001a}
\end{equation}

\noindent 2. For any  $a,b$ satisfying $0\leq a<b<+\infty $ and $n\geq 1$,%
\begin{equation}
\mathbf{Pr}\left( \tau _{y}^{bm}>n,y+\sigma B_{n}\in \left[ a,b\right]
\right) =\frac{1}{\sqrt{2\pi n}\sigma }\int_{a}^{b}\left( e^{-\frac{\left(
s-y\right) ^{2}}{2n\sigma ^{2}}}-e^{-\frac{\left( s+y\right) ^{2}}{2n\sigma
^{2}}}\right) 1_{\left\{ s\geq 0\right\} }ds.  \label{levy001b}
\end{equation}
\end{lemma}

From Lemma \ref{lemma tauBM} we easily deduce:

\begin{lemma}
\label{lemma expantau} The stopping time $\tau _{y}^{bm}$ has the following
properties: for any $y>0$ and $n \geq 1,$
\begin{equation}
\mathbf{Pr}\left( \tau _{y}^{bm}>n\right) \leq c\frac{y}{\sqrt{n}\sigma }
\label{levy002a}
\end{equation}
and, for any sequence of real numbers $(\alpha _{n})_n$ such that $\alpha_n\rightarrow 0$, 
as $n\rightarrow +\infty ,$%
\begin{equation}
\sup_{y\in \left[ 0,\alpha _{n}\sqrt{n}\right] }\left( \frac{\mathbf{Pr}%
\left( \tau _{y}^{bm}>n\right) }{\frac{2y}{\sqrt{2\pi n}\sigma }}-1\right)
=O\left( \alpha _{n}\right) .  \label{levy002b}
\end{equation}
\end{lemma}

We transfer the properties of the exit time $\tau _{y}^{bm}$ to the exit
time $\tau _{y}$ for large $y$ using the following coupling result proved in 
\cite{GLP2014}, Theorem 2.1. 
Let $\widetilde{\Omega }=\mathbb{R}^{\infty}\times \mathbb{R}^{\infty }$ 
and for any $\omega =\left( \omega _{1},\omega_{2}\right) \in \widetilde{\Omega }$ 
define the coordinate processes 
$\widetilde{Y}_{i}=\omega _{1,i}$ and $\widetilde{W}_{i}=\omega _{2,i}$ for $i\geq 1.$
\HL{Recall that according to condition \textbf{M4} for any $p>2$
the moments 
$\mathbb{E}_{x}^{1/p}\left\vert \rho \left( X_{n}\right) \right\vert ^{p} $
are uniformly bounded in $x\in \mathbb{X}$ and $n\geq 1.$}
\begin{proposition}
\label{Proposition KMT}Assume that the Markov chain $\left( X_{i}\right)
_{i\geq 0}$ and the function $\rho $ satisfy hypotheses \textbf{M1-M5}. Let $%
p>2$ and $0<\alpha <\frac{p-2}{2}.$ Then, there exists a Markov transition
kernel $x\rightarrow \widetilde{\mathbb{P}}_{x}\left( \cdot \right) $ from $%
\left( \mathbb{X},\mathcal{B}\left( \mathbb{X}\right) \right) $ to $\left( 
\widetilde{\Omega },\mathcal{B}\left( \widetilde{\Omega }\right) \right) $
such that:

\noindent 1. For any $x\in \mathbb{X}$ the distribution of $\left( 
\widetilde{Y}_{i}\right) _{i\geq 1}$ under $\widetilde{\mathbb{P}}_{x}$
coincides with the distribution of $\left( \rho \left( X_{i}\right) \right)
_{i\geq 1}$ under $\mathbb{P}_{x};$

\noindent 2. For any $x\in \mathbb{X}$ the 
$\widetilde{W}_{i}, i\geq 1,$ are i.i.d. standard normal random variables under $%
\widetilde{\mathbb{P}}_{x};$

\noindent 3. For any  
$ \varepsilon \in (0, \frac{1}{2}\frac{\alpha }{1+2\alpha })$    
there exist a constant $C$ depending on $\varepsilon ,\alpha $ and $p$ and an absolute constant $c$
such that for any
$x\in \mathbb{X}$ and $n\geq 1,$%
\begin{equation*}
\widetilde{\mathbb{P}}_{x}\left( n^{-1/2}\sup_{1\leq k\leq n}\left\vert
\sum_{i=1}^{k}\left( \widetilde{Y}_{i}-\sigma \widetilde{W}_{i}\right)
\right\vert > c n^{-\varepsilon }\right) \leq Cn^{-\alpha \frac{1+\alpha }{%
1+2\alpha }+\varepsilon \left( 2+2\alpha \right) }.
\end{equation*}%
\end{proposition}

\begin{remark}
In Theorem 2.1 of \cite{GLP2014}, the constant $C$ depends on $\Vert \delta_x \Vert_{\mathcal B'}$
and on 
$\mu_{p}\left( x\right) =\sup_{k\geq 1}\mathbb{E}_{x}^{1/p}\left\vert \rho
\left( X_{k}\right) \right\vert ^{p}.$ 
By conditions \textbf{M1} and \textbf{M5}
we have $\sup_{x\in \mathbb{X}}\Vert \delta_x \Vert_{\mathcal B'} \leq 1$ 
and $\sup_{x\in \mathbb{X}}\mu _{p}\left( x\right) < +\infty$ which implies that
we can choose the constant $C$ to be independent of $x.$ 
Note that the constant $C$ depends also on other constants introduced so far, in
particular on the variance $\sigma $ and the constants in conditions \textbf{M1-M5}. 
\end{remark}

Without loss of generality we shall consider in the sequel that 
$\HL{\rho\left(X_{i}\right)} =\widetilde{Y}_{i},$  $B_{i}=\sum_{j=0}^{i}\widetilde{W}_{j}$ and 
$\widetilde{\mathbb{P}}_{x}=\mathbb{P}_{x}.$ 
Choosing $\alpha <\frac{p-2}{2}$
and $p$ sufficiently large in Proposition \ref{Proposition KMT}, it follows
that there exists $ \varepsilon_0 >0$ such that, 
for any $\varepsilon \in (0, \varepsilon_0),$ $x\in \mathbb{X}$ and $n\geq 1,$ 
\begin{equation}
\mathbb{P}_{x}\left( \sup_{0\leq t\leq 1}\left\vert S_{\left[ nt\right]
}-\sigma B_{nt}\right\vert >  n^{1/2-2\varepsilon }\right) \leq c_{\varepsilon}  n^{-2\varepsilon},
\label{KMTbound001}
\end{equation}%
where $c_{\varepsilon}$ depends on $\varepsilon.$
To pass from the Brownian motion in discrete
time to those in continuous time one can use standard bounds for the
oscillation of $\left( B_{t}\right) _{t\geq 0}$ from Revuz and Yor \cite{RY}.

In the proof of Theorem \ref{Th asympt tau} we use the following auxiliary
result:

\begin{lemma}
\label{lemma tau ylarge}Let $\varepsilon \in (0, \varepsilon_0)$ and 
$\left(\theta _{n}\right) _{n\geq 1}$ be a sequence of positive numbers such that $%
\theta _{n}\rightarrow 0$ and $\theta _{n}n^{\varepsilon/4 }\rightarrow +\infty$ as $n\rightarrow +\infty .$ Then:

\noindent 1. There exists a constant $c>0$ such that, for $n$ sufficiently large,%
\begin{equation*}
\sup_{x\in \mathbb{X},\ y\in \left[ n^{1/2-\varepsilon },\theta _{n}n^{1/2}\right] }
\left\vert \frac{\mathbb{P}_{x}\left( \tau _{y}>n\right) }
{\frac{2y}{\sqrt{2\pi n}\sigma }}-1\right\vert \leq c\theta _{n}.
\end{equation*}

\noindent 2. There exists a constant $c_\varepsilon>0$ such that for any 
$n\geq 1$ and $y\geq n^{1/2-\varepsilon },$ 
\begin{equation*}
\sup_{x\in \mathbb{X}}\mathbb{P}_{x}\left( \tau _{y}>n\right) \leq
c_{\varepsilon }\frac{y}{\sqrt{n}}.
\end{equation*}
\end{lemma}

\begin{proof}
We start with the claim 1. Let $y\geq n^{1/2-\varepsilon }.$ Denote $%
y^{+}=y+n^{1/2-2\varepsilon }$ and $y^{-}=y-n^{1/2-2\varepsilon }.$ 
Let
\begin{equation*}
A_{n}=\left\{ \sup_{0\leq t\leq 1}\left\vert S_{\left[ nt\right] }-\sigma B_{nt}\right\vert 
\leq n^{1/2-2\varepsilon }\right\} .
\end{equation*}%
Using (\ref{KMTbound001}), we have $\mathbb{P}_{x}\left( {A}_{n}^c\right) \leq c_{\varepsilon }n^{-2 \varepsilon},$ 
for any $x\in \mathbb{X}.$ Since%
\begin{equation*}
\left\{ \tau _{y}>n\right\} \cap A_{n}\subset \left\{ \tau
_{y^{+}}^{bm}>n\right\} \cap A_{n}
\end{equation*}%
we obtain, for any $x\in \mathbb{X}$ and $y\geq n^{1/2-\varepsilon },$ 
\begin{eqnarray}
\mathbb{P}_{x}\left( \tau _{y}>n\right) &\leq &\mathbb{P}_{x}\left( \tau
_{y}>n,\ A_{n}\right) +\mathbb{P}_{x}\left( A_n^c\right)  \notag
\\
&\leq &\mathbb{P}_{x}\left( \tau _{y^{+}}^{bm}>n\right) + c_{\varepsilon}n^{-2 \varepsilon}.  \label{TAU001}
\end{eqnarray}%
In the same way we get, for any $x\in \mathbb{X}$ and $y\geq
n^{1/2-2\varepsilon },$%
\begin{equation}
\mathbb{P}_{x}\left( \tau _{y}>n\right) \geq \mathbb{P}_{x}\left( \tau
_{y^{-}}^{bm}>n\right) - c_{\varepsilon }n^{-2 \varepsilon}.  \label{TAU002}
\end{equation}%
Combining (\ref{TAU001}) and (\ref{TAU002}), for any $x\in \mathbb{X}$ and $%
y\geq n^{1/2-2\varepsilon },$%
\begin{equation}
\left\vert \mathbb{P}_{x}\left( \tau _{y}>n\right) -\mathbb{P}_{x}\left(
\tau _{y^{\pm }}^{bm}>n\right) \right\vert \leq c_{\varepsilon }n^{-2 \varepsilon}.
\label{TAU003a}
\end{equation}%
For any $y\in \left[  n^{1/2-\varepsilon  },\theta _{n}n^{1/2}\right] $ and $%
n$ large enough,%
\begin{equation}
y^{\pm } = y \pm n^{1/2-2\varepsilon }  = y(1 + \beta_n n^{-\varepsilon } )  \leq 2\theta_{n}n^{1/2},  \label{TAU005}
\end{equation}%
with some $\beta_n$ satisfying $\vert \beta_n \vert \leq 1.$
Using (\ref{levy002b}) and (\ref{TAU005}), for any $x\in \mathbb{X}$ and $%
y\in \left[  n^{1/2-\varepsilon  },\theta _{n}\sqrt{n}\right] ,$ we obtain%
\begin{equation}
\mathbb{P}_{x}\left( \tau _{y^{\pm }}^{bm}>n\right) =\frac{2y}{\sqrt{2\pi  n }%
\sigma }\left( 1+O\left( \theta _{n}\right) \right),  \label{TAU003}
\end{equation}%
where the constant in $O$ is absolute. 
 Since $\theta _{n}n^{\varepsilon/4 }\rightarrow +\infty, $ we have 
\begin{equation}
\theta _{n} \frac{y}{\sqrt{n}} 
\geq 
\frac{ n^{1/2-\varepsilon  }}{n^{\varepsilon }\sqrt{n}}= n^{-2\varepsilon },  \label{TAU006}
\end{equation}
for $n$ sufficiently large.
From (\ref{TAU003a}) and (\ref{TAU003}), taking into account (\ref{TAU006}), it follows that,
for any $x\in \mathbb{X}$ and $y\in \left[ n^{1/2-\varepsilon  },\theta _{n}%
\sqrt{n}\right] ,$%
\begin{equation*}
\mathbb{P}_{x}\left( \tau _{y}>n\right) =\frac{2y}{\sqrt{2\pi  n }%
\sigma }\left( 1+O\left( \theta _{n}\right) \right) ,
\end{equation*}%
as $n\rightarrow +\infty ,$ where the constant in $O$ does not depend on $x$
and $y.$ This proves the claim 1.

\emph{Proof of the claim 2.} Note first that for any $y\geq n^{1/2-\varepsilon  }$ and $n$ large enough%
\begin{equation}
y\leq y^{+}\leq y+n^{1/2-2\varepsilon }\leq y+y n^{-\varepsilon }\leq 2y;
\label{TAU011}
\end{equation}%
consequently, from (\ref{TAU001}) and (\ref{levy002a}), it follows that
\begin{equation*}
\mathbb{P}_{x}\left( \tau _{y}>n\right) \leq \mathbb{P}_{x}\left( \tau
_{y^{+}}^{bm}>n\right) +c_{\varepsilon }n^{-2 \varepsilon}\leq c%
\frac{y}{\sqrt{n}\sigma }+c_{\varepsilon }n^{-2 \varepsilon}
\end{equation*}%
 with $n^{-2 \varepsilon}\leq {y\over \sqrt{n}}$, for $n$ large enough. 
 This yields $\mathbb{P}_{x}\left( \tau
_{y}>n\right) \leq c_{\varepsilon }\frac{y}{\sqrt{n}}$ for any $x \in \mathbb X$ and $y\geq n^{{1\over 2}-\epsilon}$.
\end{proof}

Now we proceed to prove Theorem \ref{Th asympt tau}. 
Let $\varepsilon \in (0, \varepsilon_0)$ and $\left( \theta _{n}\right) _{n\geq 1}$ be a sequence of
positive numbers such that 
$\theta _{n}\rightarrow 0$ and $\theta_{n}n^{\varepsilon/4 }\rightarrow +\infty $
as $n\rightarrow +\infty .$ 
Let $x\in \mathbb{X}$ and $y>0.$ 
Recall that $\nu _{n}=\min \left\{ k\geq 1:\left\vert y+M_{k}\right\vert \geq 2 n^{1/2-\varepsilon }\right\}. $ A
simple decomposition gives%
\begin{equation}
P_{n}\left( x,y\right) :=\mathbb{P}_{x}\left( \tau _{y}>n\right)  = \mathbb{P%
}_{x}\left( \tau _{y}>n,\nu _{n}>n^{1-\varepsilon }\right)   
 +\mathbb{P}_{x}\left( \tau _{y}>n,\nu _{n}\leq n^{1-\varepsilon }\right) .
\label{proof tau003}
\end{equation}%
The first probability in the right hand side of (\ref{proof tau003}) is
estimated using Lemma \ref{Lemma 2},%
\begin{equation}
\sup_{x\in \mathbb{X},\ y>0}\mathbb{P}_{x}\left( \tau _{y}>n,\nu
_{n}>n^{1-\varepsilon }\right) \leq \sup_{x\in \mathbb{X},\ y>0}\mathbb{P}%
_{x}\left( \nu _{n}>n^{1-\varepsilon }\right) =O\left( e^{-cn^{\varepsilon
}}\right) .  \label{proof tau004}
\end{equation}%
For the second probability in (\ref{proof tau003}) we have by the Markov
property,%
\begin{eqnarray}
&&\mathbb{P}_{x}\left( \tau _{y}>n,\nu _{n}\leq n^{1-\varepsilon }\right) 
\notag \\
&=&\mathbb{E}_{x}\left( P_{n-\nu _{n}}\left( X_{\nu _{n}},y+S_{\nu
_{n}}\right) ;\tau _{y}>\nu _{n},\nu _{n}\leq n^{1-\varepsilon }\right) 
\notag \\
&=&\mathbb{E}_{x}\left( P_{n-\nu _{n}}\left( X_{\nu _{n}},y+S_{\nu
_{n}}\right) ;y+S_{\nu _{n}}\leq \theta _{n}n^{1/2},\tau _{y}>\nu _{n},\nu
_{n}\leq n^{1-\varepsilon }\right)  \notag \\
&&+\mathbb{E}_{x}\left( P_{n-\nu _{n}}\left( X_{\nu _{n}},y+S_{\nu
_{n}}\right) ;y+S_{\nu _{n}}>\theta _{n}n^{1/2},\tau _{y}>\nu _{n},\nu
_{n}\leq n^{1-\varepsilon }\right)  \notag \\
&=&J_{1}\left( x,y\right) +J_{2}\left( x,y\right) .  \label{proof tau006}
\end{eqnarray}%
By Lemma \ref{Martingale decomp} one gets $y+S_{\nu _{n}}\geq y+M_{\nu
_{n}}-a, \mathbb{P}_{x}$-a.e., where $a=2\left\Vert \mathbf{P}\theta
\right\Vert _{\infty }.$ 
By the definition of $\nu _{n},$ we have, for $n$ sufficiently large, $ \mathbb{P}_{x}$-a.e.
\begin{equation}
y+S_{\nu _{n}}\geq y+M_{\nu _{n}}-a\geq 2 n^{1/2-\varepsilon }-a\geq  n^{1/2-\varepsilon  }.\label{proof tau006bis}
\end{equation}
 On the other hand, obviously, for $1\leq k\leq n^{1-\varepsilon },$%
\begin{equation}
\mathbb{P}_{x}\left( \tau _{y}>n\right) \leq P_{n-k}\left( x,y\right) \leq 
\mathbb{P}_{x}\left( \tau _{y}>n-n^{1-\varepsilon }\right) .
\label{proof tau001}
\end{equation}%
Using (\ref{proof tau006bis}),  the two sided bounds of (\ref{proof tau001}) and the claim 1 of 
Lemma \ref{lemma tau ylarge} 
with $ \theta_n$ replaced by $ \theta_n \left( \frac{n}{n-n^{1- \varepsilon}} \right)^{1/2},$
we deduce that on the event 
$F=\left\{ y+S_{\nu _{n}}\leq \theta _{n}n^{1/2}, \tau_y>\nu_n, \nu _{n}\leq
n^{1-\varepsilon }\right\} ,$ for $n$ sufficiently large, $\mathbb{P}_{x}$-a.e.
\begin{equation}
P_{n-\nu _{n}}\left( X_{\nu _{n}},y+S_{\nu _{n}}\right) = \frac{2\left( y+S_{\nu _{n}}\right) }{\sqrt{2\pi n}\sigma }
\left( 1+o(1) \right);  
\label{proof tau007}
\end{equation}%
in particular, (\ref{proof tau007}) implies that, on this event $F,$ we have, 
$\mathbb{P}_{x}$-a.e.%
\begin{equation}
P_{n-\nu _{n}}\left( X_{\nu _{n}},y+S_{\nu _{n}}\right) \leq c\frac{y+S_{\nu_{n}}}{\sqrt{n}}.  \label{proof tau008}
\end{equation}%
Implementing (\ref{proof tau007}) in the expression for $J_{1}$ we obtain%
\begin{eqnarray}
J_{1}\left( x,y\right) &=&\frac{2\left( 1+o\left( 1\right) \right) }{\sqrt{%
2\pi n}\sigma }\mathbb{E}_{x}\left( y+S_{\nu _{n}};y+S_{\nu _{n}}\leq \theta
_{n}n^{1/2},\tau _{y}>\nu _{n},\nu _{n}\leq n^{1-\varepsilon }\right)  \notag
\\
&=&\frac{2\left( 1+o\left( 1\right) \right) }{\sqrt{2\pi n}\sigma }\mathbb{E}%
_{x}\left( y+S_{\nu _{n}};\tau _{y}>\nu _{n},\nu _{n}\leq n^{1-\varepsilon
}\right)  \notag \\
&&+\frac{2\left( 1+o\left( 1\right) \right) }{\sqrt{2\pi n}\sigma }%
J_{3}\left( x,y\right) ,  \label{proof tau009}
\end{eqnarray}%
where%
\begin{equation*}
J_{3}\left( x,y\right) =\mathbb{E}_{x}\left( y+S_{\nu _{n}};y+S_{\nu
_{n}}>\theta _{n}n^{1/2},\tau _{y}>\nu _{n},\nu _{n}\leq n^{1-\varepsilon
}\right) .
\end{equation*}%
Similarly, implementing (\ref{proof tau008}) in the expression for $J_{2},$
we have 
\begin{equation}
J_{2}\left( x,y\right) \leq \frac{c}{\sqrt{n}}J_{3}\left( x,y\right) .
\label{proof tau011}
\end{equation}%
From (\ref{proof tau003}), (\ref{proof tau004}), (\ref{proof tau006}), (\ref%
{proof tau009}) and (\ref{proof tau011}) we get 
\begin{eqnarray*}
\mathbb{P}_{x}\left( \tau _{y}>n\right) &=&\frac{2\left( 1+o\left( 1\right)
\right) }{\sqrt{2\pi n}\sigma }\mathbb{E}_{x}\left( y+S_{\nu _{n}};\tau
_{y}>\nu _{n},\nu _{n}\leq n^{1-\varepsilon }\right) \\
&&+O\left( n^{-1/2}J_{3}\left( x,y\right) \right) +O\left(
e^{-cn^{\varepsilon }}\right) .
\end{eqnarray*}%
The first assertion of Theorem \ref{Th asympt tau} follows if we show that $%
\mathbb{E}_{x}\left( y+S_{\nu _{n}};\tau _{y}>\nu _{n},\nu _{n}\leq
n^{1-\varepsilon }\right) $ converges to $V\left( x,y\right) $ and that 
$J_{3}\left( x,y\right) =o(1)$ as $n\rightarrow +\infty .$ This is proved in Lemmas \ref{Lemma BB-1} and \ref{Lemma BB2}
below. Note that the convergence established in these lemmas is not uniform
in $x\in \mathbb{X}$ which explains why the convergence in Theorem \ref{Th
asympt tau} is not uniform.

\begin{lemma}
\label{Lemma BB-1}
Let $\varepsilon \in (0, \varepsilon_0).$ 
For any $x\in \mathbb{X}$ and $y>0,$
\begin{equation}
\lim_{n\rightarrow +\infty }\mathbb{E}_{x}\left( y+S_{\nu _{n}};\tau _{y}>\nu
_{n},\nu _{n}\leq n^{1-\varepsilon }\right) =V\left( x,y\right) .
\label{lemma BB1 0}
\end{equation}
\end{lemma}

\begin{proof}
First we prove (\ref{lemma BB1 0}) for the martingale $M_{n}.$ For any $x\in 
\mathbb{X}$ and $y>0,$%
\begin{eqnarray}
&&\mathbb{E}_{x}\left( y+M_{\nu _{n}};\tau
_{y}>\nu _{n},\nu _{n}\leq n^{1-\varepsilon }\right)  \notag \\
&=&\mathbb{E}_{x}\left( y+M_{\nu _{n}\wedge \left[ n^{1-\varepsilon }\right]
};\tau _{y}>\nu _{n}\wedge n^{1-\varepsilon }, \nu _{n}\leq n^{1-\varepsilon
}\right)  \notag \\
&=&\mathbb{E}_{x}\left( y+M_{\nu _{n}\wedge \left[ n^{1-\varepsilon }\right]
};\tau _{y}>\nu _{n}\wedge n^{1-\varepsilon }\right)  \label{proof BB000} \\
&&-\mathbb{E}_{x}\left( y+M_{\nu _{n}\wedge \left[ n^{1-\varepsilon }\right]
};\tau _{y}>\nu _{n}\wedge n^{1-\varepsilon },\nu _{n}>n^{1-\varepsilon
}\right) .  \label{proof BB001}
\end{eqnarray}%
Using Lemma \ref{Lemma 3}, the expectation in (\ref{proof BB001}) is bounded
as follows:%
\begin{eqnarray}
&&\mathbb{E}_{x}\left( y+M_{\nu _{n}\wedge \left[ n^{1-\varepsilon }\right]%
};\tau _{y}>\nu _{n}\wedge n^{1-\varepsilon },\nu _{n}>n^{1-\varepsilon
}\right)  \notag \\
&=&\mathbb{E}_{x}\left( y+M_{\left[ n^{1-\varepsilon }\right]};\tau
_{y}>n^{1-\varepsilon },\nu _{n}>n^{1-\varepsilon }\right)  \notag \\
&\leq &c\left( 1+y\right) \exp \left( -c_{\varepsilon }n^{\varepsilon
}\right) .  \label{proof BB002}
\end{eqnarray}%
The expectation in (\ref{proof BB000}) is decomposed into two terms:%
\begin{eqnarray}
&&\mathbb{E}_{x}\left( y+M_{\nu _{n}\wedge \left[ n^{1-\varepsilon }\right]
};\tau _{y}>\nu _{n}\wedge n^{1-\varepsilon }\right)  \notag \\
&=&\mathbb{E}_{x}\left( y+M_{\nu _{n}\wedge \left[ n^{1-\varepsilon }\right]
}\right) -\mathbb{E}_{x}\left( y+M_{\nu _{n}\wedge \left[ n^{1-\varepsilon }%
\right]};\tau _{y}\leq \nu _{n}\wedge n^{1-\varepsilon }\right) .
\label{proof BB003}
\end{eqnarray}%
Since $\left( M_{n}\right) _{n\geq 1}$ is a martingale, $\mathbb{E}%
_{x}\left( y+M_{\nu _{n}\wedge \left[ n^{1-\varepsilon }\right] }\right) =y$
and%
\begin{equation}
\mathbb{E}_{x}\left( y+M_{\nu _{n}\wedge \left[ n^{1-\varepsilon }\right]
};\tau _{y}\leq \nu _{n}\wedge n^{1-\varepsilon }\right) =\mathbb{E}%
_{x}\left( y+M_{\tau _{y}};\tau _{y}\leq \nu _{n}\wedge n^{1-\varepsilon
}\right) .  \label{proof BB003a}
\end{equation}%
By Corollary \ref{Intergrab of Mtau}, $M_{\tau _{y}}$ is integrable, consequently %
\begin{equation*}
\lim_{n\rightarrow +\infty }\mathbb{E}%
_{x}\left( y+M_{\tau _{y}};\tau _{y}\leq \nu _{n}\wedge n^{1-\varepsilon
}\right) =\mathbb{E}_{x}\left( y+M_{\tau _{y}}\right) ,
\end{equation*}%
which, together with (\ref{proof BB003a}) and (\ref{proof BB003}), implies%
\begin{equation}
\lim_{n\rightarrow +\infty }\mathbb{E}_{x}\left( y+M_{\nu _{n}\wedge \left[
n^{1-\varepsilon }\right] };\tau _{y}>\nu _{n}\wedge n^{1-\varepsilon
}\right) =y-\mathbb{E}_{x}\left( y+M_{\tau _{y}}\right) =V\left( x,y\right) .
\label{proof BB004}
\end{equation}%
From (\ref{proof BB002}) and (\ref{proof BB004}) it follows that%
\begin{equation}
\lim_{n\rightarrow +\infty }\mathbb{E}_{x}\left( y+M_{\nu _{n}};\tau _{y}>\nu
_{n},\nu _{n}\leq n^{1-\varepsilon }\right) =V\left( x,y\right) .
\label{proof BB005}
\end{equation}

Now we extend (\ref{proof BB005}) to $S_{n}$ using the fact that the
difference $R_{n}=S_{n}-M_{n}$ is  $\mathbb P_x$-a.s.  bounded. For this we write%
\begin{eqnarray}
&&\mathbb{E}_{x}\left( y+S_{\nu _{n}\wedge \left[ n^{1-\varepsilon }\right]
};\tau _{y}>\nu _{n}\wedge n^{1-\varepsilon },\nu _{n}\leq n^{1-\varepsilon
}\right)  \notag \\
&=&\mathbb{E}_{x}\left( y+M_{\nu _{n} };\tau _{y}>\nu _{n},\nu _{n}\leq
n^{1-\varepsilon }\right) +\mathbb{E}_{x}\left( R_{\nu _{n}};\tau _{y}>\nu
_{n},\nu _{n}\leq n^{1-\varepsilon }\right) .  \label{proof BB006}
\end{eqnarray}%
Note that $\mathbb{P}_{x}$-a.s.\ we have $\tau _{y}< +\infty ,$ $%
\sup_{n\geq 0}\left\vert S_{n}-M_{n}\right\vert \leq a=2\Vert P\theta\Vert_\infty$ 
and $\nu_{n}\rightarrow +\infty$ \HL{(by Lemma \ref{nu001})}, 
which implies that, as $n\rightarrow +\infty ,$%
\begin{equation}
\mathbb{E}_{x}\left( \left\vert R_{\nu _{n} }\right\vert ;\tau _{y}>\nu
_{n},\nu _{n}\leq n^{1-\varepsilon }\right) \leq a\mathbb{P}_{x}\left( \tau
_{y}>\nu _{n}\right) \rightarrow 0.  \label{proof BB007}
\end{equation}%
The assertion of the lemma follows from (\ref{proof BB005}), (\ref{proof
BB006}) and (\ref{proof BB007}).
\end{proof}

\begin{lemma}
\label{Lemma BB2}
Let $\varepsilon \in (0, \varepsilon_0)$ and $\left( \theta _{n}\right) _{n\geq 1}$ be a sequence of
positive numbers such that 
$\theta _{n}\rightarrow 0$  and $\theta_{n}n^{\varepsilon/4 }\rightarrow +\infty $
as $n\rightarrow +\infty .$ 
For any $x\in \mathbb{X}$ and $y>0,$%
\begin{equation*}
\lim_{n\rightarrow +\infty } n^{2 \varepsilon} \mathbb{E}_{x}\left(
y+S_{\nu_{n}}; y+ S_{\nu _{n}} >\theta _{n}n^{1/2},\tau _{y}>\nu _{n},\nu
_{n}\leq n^{1-\varepsilon }\right) =0.
\end{equation*}
\end{lemma}

\begin{proof}
Let $y^{\prime }=y+a,$ where $a=2\left\Vert \mathbf{P}\theta \right\Vert_{\infty }.$ 
With the notation  
$M_{n}^*:=\max_{1\leq k\leq n}\left\vert M_{k}\right\vert$,
we have%
\begin{eqnarray*}
&&\mathbb{E}_{x}\left( y+S_{\nu_{n}}; y+ S_{\nu _{n}} >\theta_{n}n^{1/2},\tau _{y}>\nu _{n},\nu _{n}\leq n^{1-\varepsilon }\right) \\
&\leq &\mathbb{E}_{x}\left( y^{\prime }+M_{\nu _{n}}; y^{\prime }+M_{\nu_{n}} >\theta _{n}n^{1/2}, 
\nu _{n} \leq n^{1-\varepsilon }\right) \\
&\leq &\mathbb{E}_{x}\left( y^{\prime } + M^*_{[n^{1-\varepsilon }]}; 
y^{\prime } + M^*_{[n^{1-\varepsilon }]}  >\theta_{n}n^{1/2} \right) .
\end{eqnarray*}%
Since $\theta _{n}n^{\varepsilon/4 }\rightarrow +\infty $ 
as $n\rightarrow +\infty $, 
to finish the proof it is enough to show that, for any $\delta >0$ and 
  $x\in \mathbb{X},$%
\begin{equation}
\lim_{n\rightarrow +\infty }  n^{2 \varepsilon} \mathbb{E}_{x}\left( y^{\prime }+M_{n}^*;
M_{n}^* >n^{1/2+\delta } \right) =0  \label{proofCCC001}
\end{equation}%
(let us remark that the condition $\theta _{n}n^{\varepsilon}\rightarrow +\infty $ at this step would not be sufficient).
Obviously
\begin{eqnarray}
\mathbb{E}_{x}\left( y^{\prime }+M_n^*;M_n^*>n^{1/2+\delta }\right) 
\leq y^{\prime }\mathbb{P}_{x}\left( M_n^*>n^{1/2+\delta }\right) +%
\mathbb{E}_{x}\left( M_n^*;M_n^*>n^{1/2+\delta }\right) ,
\label{proofCCC002}
\end{eqnarray}%
with 
\begin{eqnarray}
\mathbb{E}_{x}\left( M_n^*;M_n^*>n^{1/2+\delta }\right) &=&n^{1/2+\delta }%
\mathbb{P}_{x}\left( M_n^*>n^{1/2+\delta }\right) \notag \\
& + & \int_{n^{1/2+\delta }}^{+\infty } \mathbb{P}_{x} \left( M_n^*>t\right)dt .   \label{proofCCC003}
\end{eqnarray}%
By Doob's maximal inequality for martingales%
\begin{equation}
\mathbb{P}_{x}\left( M_n^*>t\right) \leq \frac{1}{t^{p}}\mathbb{E}%
_{x}\left\vert M_{n}\right\vert ^{p}\leq c\frac{n^{p/2}}{t^{p}}.
\label{proofCCC004}
\end{equation}%
Implementing (\ref{proofCCC004}) in (\ref{proofCCC002}) and (\ref{proofCCC003}),%
\begin{eqnarray*}
&&\mathbb{E}_{x}\left( \HL{y^{\prime }+M_{n}^*};\ M_n^* >n^{1/2+\delta } \right) \\
&\leq &c \left( y^{\prime }+n^{1/2+\delta }\right) \frac{n^{p/2}}{%
n^{p/2+p\delta }}+cn^{p/2}\HL{\int_{n^{1/2+\delta }}^{+\infty }t^{-p}dt}
\\
&\leq &c\left( y^{\prime }+n^{1/2+\delta }\right) n^{-p\delta }+%
cn^{p/2}n^{-\left( 1/2+\delta \right) \left( p-1\right) }
\\
&\leq &c\left( y^{\prime }+n^{1/2+\delta }\right) n^{-p\delta }+%
cn^{-p\delta +1/2+\delta }.
\end{eqnarray*}%
Since $p$ can be taken arbitrarily large we get (\ref{proofCCC001}).
\end{proof}

The small rate of convergence of order $n^{-2 \varepsilon}$ obtained in the previous lemma will be used in the proof of Theorem \protect\ref{Th weak conv cond positive} (see next section).
 
We proceed to prove the second assertion in Theorem \ref{Th asympt tau}. 
It is enough to give a proof only for $n$ large enough, otherwise the assertion is trivial.
Set $\mathbb P_n(x,y) := \mathbb P_x (\tau_y >n).$
A simple decomposition gives
\begin{equation}
\label{sa001}
\mathbb P_n(x,y)  
=  \mathbb P_x \left(\tau_y >n, \nu_n >n^{1-\varepsilon}\right) + \mathbb P_x \left(\tau_y >n, \nu_n \leq n^{1-\varepsilon}\right).  
\end{equation}
Using Lemma \ref{Lemma 2}, the first probability converges uniformly to $0$:
\begin{equation}
\label{sa002}
\sup_{x\in\mathbb X,\, y>0}  \mathbb P_x \left(\tau_y >n, \nu_n > n^{1-\varepsilon}\right) 
\leq  \sup_{x\in\mathbb X,\, y>0}  \mathbb P_x \left( \nu_n > n^{1-\varepsilon}\right) 
\leq O\left(e^{-c_{\varepsilon}n^{\varepsilon}}\right).
\end{equation}
For the second probability, by the Markov property, we have
\begin{equation}
\label{sa003}
\mathbb P_x \left(\tau_y >n, \nu_n \leq n^{1-\varepsilon}\right)
= \mathbb E_{x} \left( \mathbb P_{n-\nu_n } \left(X_{\nu_n}, y+ S_{\nu_n}\right); \tau_y >\nu_n, \nu_n \leq n^{1-\varepsilon}\right).
\end{equation}
We shall bound the probability $\mathbb P_{n-\nu_n } \left(X_{\nu_n}, y+ S_{\nu_n}\right).$
To this end we note that, for any $x' \in \mathbb X$, $y'>0$ and $0\leq k \leq n^{1-\varepsilon}$
\begin{equation}
\label{sa004}
\mathbb P_{n-k } \left(x', y' \right) \leq  \mathbb P_{n-[n^{1-\varepsilon}] } \left(x', y' \right) 
\end{equation}
and, by the point 2 of Lemma \ref{lemma tau ylarge}, for any $x' \in \mathbb X$ and $y'\geq n^{1/2-\varepsilon} $
\begin{equation}
\label{sa005}
\mathbb P_{n-[n^{1-\varepsilon}] } \left(x', y' \right) \leq c_{\varepsilon} \frac{y'}{\sqrt{n}}.
\end{equation}
By the definition of $\nu_n$ and Lemma \ref{Martingale decomp}, for $n$ sufficiently large and any $y>0,$ we have $\mathbb P_x$-a.s.
\begin{equation}
\label{sa006}
y+S_{\nu_n} \geq y+M_{\nu_n} - a \geq 2n^{1/2-\varepsilon} -a \geq n^{1/2-\varepsilon},
\end{equation}
where $a= 2\left\Vert  \mathbb{P} \theta \right \Vert _{\infty}.$
The bounds (\ref{sa004}), (\ref{sa005}) and (\ref{sa006}) imply that, $\mathbb P_x$-a.s.
\begin{equation}
\label{sa007}
\mathbb P_{n-[n^{1-\varepsilon}] } \left(X_{\nu_n}, y+ S_{\nu_n} \right) \leq c_{\varepsilon} \frac{ y+ S_{\nu_n} }{\sqrt{n}}.
\end{equation}
From (\ref{sa003}) and (\ref{sa007}) it follows that  
\begin{equation}
\label{sa008}
\mathbb P_x \left(\tau_y >n, \nu_n \leq n^{1-\varepsilon}\right)
\leq   \frac{c_{\varepsilon} }{\sqrt{n}} \mathbb E_{x} \left(   y+ S_{\nu_n}; \tau_y >\nu_n, \nu_n \leq n^{1-\varepsilon}\right).
\end{equation}
Let $y'=y+a.$ Since $\tau_{y} \leq T_{y'},$ we have 
\begin{equation*}
\label{sa009}
\mathbb E_{x} \left(   y+ S_{\nu_n}; \tau_y >\nu_n, \nu_n \leq n^{1-\varepsilon}\right)
\leq \mathbb E_{x} \left(   y'+ M_{\nu_n}; T_y' >\nu_n, \nu_n \leq n^{1-\varepsilon}\right).
\end{equation*}
Using the fact that $\left(\left(y'+M_n\right) 1 _{T_{y'} > n}\right)_{n\leq 1}$ is a submartingale and Lemma \ref{Lemma 4}, we bound the last expectation by
\begin{equation*}
\label{sa010}
\mathbb E_{x} \left(   y'+ M_{[n^{1-\varepsilon}]}; T_y' >[n^{1-\varepsilon}], \nu_n \leq n^{1-\varepsilon}\right) \leq c\left(1+y'\right) = c\left(1+y +a\right). 
\end{equation*}
Inserting this bound in (\ref{sa008}) we get
\begin{equation}
\label{sa011}
\mathbb P_x \left(\tau_y >n, \nu_n \leq n^{1-\varepsilon}\right)
\leq   \frac{c_{\varepsilon} }{\sqrt{n}} \left(1+y +a\right).
\end{equation}
From (\ref{sa001}), (\ref{sa002}) and (\ref{sa011}) it follows that
\begin{equation}
\label{sa011}
\mathbb P_x \left(\tau_y >n\right)
\leq O\left(e^{c_{\varepsilon}n^{-\varepsilon}}\right) + \frac{c_{\varepsilon} }{\sqrt{n}} \left(1+y +a\right),
\end{equation}
which proves the second assertion for $n$ large enough. 

\section{Proof of Theorem \protect\ref{Th weak conv cond positive}}

We first state  the following lemma.

\begin{lemma}
\label{Lemma Weak Conv} Let $\varepsilon \in (0, \varepsilon_0),$ $t>0$ and $\left( \theta _{n}\right) _{n\geq 1}$
be a sequence such that $\theta _{n}\rightarrow 0$ and 
$\theta_{n}n^{\varepsilon/4 }\rightarrow +\infty $ as $n\rightarrow +\infty .$ Then%
\begin{equation}
\lim_{n\rightarrow +\infty }\sup \left\vert \frac{\mathbb{P}_{x}\left( \tau
_{y}>n-k,\frac{y+S_{n-k}}{\sqrt{n}}\leq t\right) }{\frac{2y}{\sqrt{2\pi n} }%
\frac{1}{\sigma ^{3}}\int_{0}^{t}u\exp \left( -\frac{u^{2}}{2\sigma ^{2}}%
\right) du}-1\right\vert =0,  \label{WKth001}
\end{equation}%
where $\sup $ is taken over $x\in \mathbb{X},\ k\leq n^{1-\varepsilon }$ and 
$ n^{1/2-\varepsilon  } \leq y\leq \theta _{n}n^{1/2}$.
\end{lemma}

\begin{proof}
Denote $y^{+}=y+n^{1/2-2\varepsilon },$ $y^{-}=y-n^{1/2-2\varepsilon },$ $%
t^{+}=t+n^{-2\varepsilon }$ and $t^{-}=t-n^{-2\varepsilon }.$ Let $m_{n}=n-k,
$ where $1\leq k\leq n^{1-\varepsilon }.$ As in the previous section, set 
\begin{equation*}
A_{n}=\left\{ \sup_{0\leq t\leq 1}\left\vert S_{\left[ nt\right] }-\sigma
B_{nt}\right\vert \leq  n^{1/2-2\varepsilon }\right\} .
\end{equation*}%
Since on this set it holds%
\begin{equation*}
\left\{ \tau _{y}>m_{n}\right\} \subset \left\{ \tau
_{y^{+}}^{bm}>m_{n}\right\}
\end{equation*}%
and 
\begin{equation*}
\left\{ \frac{y+S_{m_{n}}}{\sqrt{n}}\leq t\right\} \subset \left\{ \frac{%
y+\sigma B_{m_{n}}}{\sqrt{n}}\leq t^{+}\right\} ,
\end{equation*}%
we obtain 
\begin{eqnarray}
&&\mathbb{P}_{x}\left( \tau _{y}>m_{n},y+S_{m_{n}}\leq \sqrt{n}t\right) 
\notag \\
&\leq &\mathbb{P}_{x}\left( \tau _{y}>m_{n},\frac{y+S_{m_{n}}}{\sqrt{n}}\leq
t,A_{n}\right) +\mathbb{P}_{x}\left( A_n^c\right)  \notag \\
&\leq &\mathbb{P}_{x}\left( \tau _{y^{+}}^{bm}>m_{n},\frac{y^{+}+\sigma
B_{m_{n}}}{\sqrt{n}}\leq t^{+},A_{n}\right) +\mathbb{P}_{x}\left( A^c_n
\right)  \notag \\
&\leq &\mathbb{P}_{x}\left( \tau _{y^{+}}^{bm}>m_{n},\frac{y^{+}+\sigma
B_{m_{n}}}{\sqrt{n}}\leq t^{+}\right) +\mathbb{P}_{x}\left( A_n^c\right) .  \label{WK001}
\end{eqnarray}%
In the same way we get%
\begin{eqnarray}
&&\mathbb{P}_{x}\left( \tau _{y}>m_{n},y+S_{m_{n}}\leq
\sqrt{n} t\right)  \notag \\
&\geq &\mathbb{P}_{x}\left( \tau _{y^{-}}^{bm}>m_{n},\frac{y^{-}+\sigma
B_{m_{n}}}{\sqrt{n}}\leq t^{-}\right) -\mathbb{P}_{x}\left(A_n^c\right) .  \label{WK002}
\end{eqnarray}

Now we deal with the first probability in (\ref{WK001}). By Lemma \ref{lemma
tauBM},%
\begin{eqnarray}
&&\mathbb{P}_{x}\left( \tau _{y^{+}}^{bm}>m_{n},\frac{y^{+}+\sigma B_{m_{n}}%
}{\sqrt{n}}\leq t^{+}\right)  \notag \\
&=&\frac{1}{\sqrt{2\pi m_{n}}\sigma }\int_{0}^{\sqrt{n}t^{+}}\left[ e^{-%
\frac{\left( s-y^{+}\right) ^{2}}{2m_{n}\sigma ^{2}}}-e^{-\frac{\left(
s+y^{+}\right) ^{2}}{2m_{n}\sigma ^{2}}}\right] ds\;\;\;  \notag \\
&&\left( \text{substituting\ }s=u\sqrt{n}\right)  \notag \\
&=&\frac{e^{-\frac{\left( y^{+}/\sqrt{n}\right) ^{2}}{2\sigma ^{2}m_{n}/n}}}{%
\sqrt{2\pi m_{n}/n}\sigma }\int_{0}^{t^{+}}e^{-\frac{u^{2}}{2\sigma
^{2}m_{n}/n}}\left[ e^{\frac{uy^{+}/\sqrt{n}}{\sigma ^{2}m_{n}/n}}-e^{\frac{%
-uy^{+}/\sqrt{n}}{\sigma ^{2}m_{n}/n}}\right] du.  \label{WK004}
\end{eqnarray}%
Note that, uniformly in $k\leq n^{1-\varepsilon }$ and $ n^{1/2-\varepsilon 
}\leq y\leq \theta _{n}n^{1/2},$ we have, as $n\rightarrow +\infty ,$ 
\begin{equation}
m_{n}/n=(n-k)/n=1-O\left( n^{-\varepsilon }\right)  \label{WK005}
\end{equation}%
and 
\begin{equation}
\frac{y^{+}}{y}=\frac{y+n^{1/2-2\varepsilon }}{y}=1+O\left( n^{-\varepsilon}\right) .  
\label{WK006}
\end{equation}%
Therefore, uniformly in $k\leq n^{1-\varepsilon },$ $n^{1/2-\varepsilon 
}\leq y\leq \theta _{n}n^{1/2}$ and $0\leq u\leq t^{+},$%
\begin{equation*}
v_{n}=\frac{uy^{+}/\sqrt{n}}{\sigma ^{2}m_{n}/n}=O\left( \left( t+n^{-2 \varepsilon}\right) \theta _{n}\right) =o\left( 1\right)
\end{equation*}%
as $n\rightarrow +\infty .$ Since, by Taylor's expansion, $%
e^{v_{n}}-e^{-v_{n}}=2v_{n}\left( 1+o\left( 1\right) \right) $ as $%
v_{n}\rightarrow 0,$ using again (\ref{WK005}), (\ref{WK006}) we get, uniformly in 
$k\leq n^{1-\varepsilon },$ $ n^{1/2-\varepsilon  }\leq y\leq \theta
_{n}n^{1/2}$ and $0\leq u\leq t^{+},$%
\begin{eqnarray}
e^{\frac{uy^{+}/\sqrt{n}}{\sigma ^{2}m_{n}/n}}-e^{\frac{-uy^{+}/\sqrt{n}}{\sigma ^{2}m_{n}/n}} 
=\frac{2uy^{+}/\sqrt{n}}{\sigma ^{2}m_{n}/n}\left(
1+o\left( 1\right) \right)  \notag 
=\frac{2uy}{\sigma ^{2}\sqrt{n}}\left( 1+o\left( 1\right) \right) ,
\label{WK007}
\end{eqnarray}%
as $n\rightarrow +\infty .$ 
Similarly, uniformly in $k\leq n^{1-\varepsilon },$ $ n^{1/2-\varepsilon  }\leq y\leq \theta_{n}n^{1/2}$ and $0\leq u\leq t^{+},$
\begin{equation}
e^{-\frac{\left( y^{+}/\sqrt{n}\right) ^{2}}{2\sigma ^{2}m_{n}/n}}=1+o\left(1\right) 
\quad {\rm and} \quad
e^{-\frac{u^{2}}{2\sigma ^{2}m_{n}/n}}   = e^{-\frac{u^{2}}{2\sigma ^{2}}}  \left( 1+o\left(1\right) \right). 
\label{WK008}
\end{equation}%
From (\ref{WK004}), (\ref{WK005}), (\ref{WK007}) and (\ref{WK008}), we obtain%
\begin{eqnarray*}
\mathbb{P}_{x}\left( \tau _{y^{+}}^{bm}>m_{n},y^{+}+\sigma B_{m_{n}}\leq 
\sqrt{n}t^{+}\right) 
=\frac{2y}{\sqrt{2\pi n}\sigma ^{3}}\int_{0}^{t^{+}}e^{-\frac{u^{2}}{%
2\sigma ^{2}}}udu\left( 1+o\left( 1\right) \right),
\end{eqnarray*}%
as $n\rightarrow +\infty .$ Since $\int_{0}^{t^{+}}e^{-\frac{u^{2}}{2\sigma
^{2}}}udu=\int_{0}^{t}e^{-\frac{u^{2}}{2\sigma ^{2}}}udu+o\left( n^{-2\varepsilon}\right) $ 
and $ n^{1/2-\varepsilon  }\leq y\leq \theta _{n}n^{1/2}$ we get%
\begin{eqnarray}
&&\mathbb{P}_{x}\left( \tau _{y^{+}}^{bm}>m_{n},\frac{y^{+}+\sigma B_{m_{n}}}{\sqrt{n}}\leq t^{+}\right)  \notag \\
&=&\frac{2y}{\sqrt{2\pi n}\sigma ^{3}}\int_{0}^{t}e^{-\frac{u^{2}}{2\sigma
^{2}}}udu\left( 1+\frac{1}{y}o\left( n^{1/2-2 \varepsilon}\right) \right) \left(
1+o\left( 1\right) \right)  \notag \\
&&\left( \text{use\ }y\geq  n^{1/2-\varepsilon  }\right)  \notag \\
&=&\frac{2y}{\sqrt{2\pi n} \sigma ^{3} }\int_{0}^{t}e^{-\frac{u^{2}}{2\sigma
^{2}}}udu\left( 1+o\left( 1\right) \right) .  \label{WK009}
\end{eqnarray}%
Similarly, we obtain%
\begin{equation}
\mathbb{P}_{x}\left( \tau _{y^{-}}^{bm}>m_{n},\frac{y^{-}+\sigma B_{m_{n}}}{%
\sqrt{n}}\leq t^{-}\right) =\frac{2y}{\sqrt{2\pi n}\sigma^3 }\int_{0}^{t}e^{-%
\frac{u^{2}}{2\sigma ^{2}}}udu\left( 1+o\left( 1\right) \right) .
\label{WK010}
\end{equation}%
Taking into account the bound (\ref{KMTbound001}),
we have 
\begin{equation}
\mathbb{P}_{x}\left( A_n^c\right) =O\left( n^{-2 \varepsilon}\right) .
\label{WK003}
\end{equation}%
From (\ref{WK001}), (\ref{WK002}), (\ref{WK009}), (\ref{WK010}) and (\ref%
{WK003}) it follows that%
\begin{eqnarray*}
\mathbb{P}_{x}\left( \tau _{y}>m_{n},\frac{y+S_{m_{n}}}{\sqrt{n}}\leq t\right) 
&=&\frac{2y}{\sqrt{2\pi n}\sigma ^{3}}\int_{0}^{t}e^{-\frac{u^{2}}{2\sigma ^{2}}}u du
\left( 1+o\left( 1\right) \right) +O\left( n^{-2 \varepsilon}\right) \\
&=&\frac{2y}{\sqrt{2\pi n} \sigma ^{3}}\int_{0}^{t}e^{-\frac{u^{2}}{2\sigma ^{2}}}udu\left( 1+o\left( 1\right) \right),  
\label{WK003a}
\end{eqnarray*}%
where for the last line we used the fact that, by (\ref{TAU006}), $n^{-2 \varepsilon}=o\left(y/n^{1/2}  \right).$   
This proves (\ref{WKth001}).   
\end{proof}

We now proceed to prove Theorem \ref{Th weak conv cond positive}.
Let $\varepsilon \in (0, \varepsilon_0).$
Note that%
\begin{eqnarray}
&&\frac{\mathbb{P}_{x}\left( \frac{y+S_{n}}{\sqrt{n}}\leq  t,\tau
_{y}>n\right) }{\mathbb{P}_{x}\left( \tau _{y}>n\right) }  \notag \\
&=&\frac{\mathbb{P}_{x}\left( y+S_{n}\leq \sqrt{n} t,\tau
_{y}>n, y+S_{\nu _{n}} \leq \theta _{n}\sqrt{n},\nu
_{n}\leq n^{1-\varepsilon }\right) }{\mathbb{P}_{x}\left( \tau _{y}>n\right) 
}  \notag \\
&+&\frac{\mathbb{P}_{x}\left( y+S_{n}\leq \sqrt{n} t,\tau
_{y}>n,y+S_{\nu _{n}} >\theta _{n}\sqrt{n},\nu _{n}\leq
n^{1-\varepsilon }\right) }{\mathbb{P}_{x}\left( \tau _{y}>n\right) }  \notag
\\
&+&\frac{\mathbb{P}_{x}\left( y+S_{n}\leq \sqrt{n} t,\tau _{y}>n,\nu
_{n}>n^{1-\varepsilon }\right) }{\mathbb{P}_{x}\left( \tau _{y}>n\right) } 
\notag \\
&=&T_{n,1}+T_{n,2}+T_{n,3}.  \label{WK201}
\end{eqnarray}%
The terms $T_{n, 2}$ and $T_{n, 3}$ in this decomposition are negligible. Let us control first the term $T_{n, 2}$; 
we have $\theta_n>  n^{-\varepsilon/4}$, so  
\begin{eqnarray*}
T_{n,2}
&\leq &\frac{\mathbb{P}_{x}\left( \tau _{y}>n, y+S_{\nu _{n}} >\theta _{n}\sqrt{n}, \nu _{n}\leq n^{1-\varepsilon }\right) 
}{\mathbb{P}_{x}\left( \tau _{y}>n\right) } \\
&\leq &    n^{-1/2+\varepsilon/4  } \frac{\mathbb{E}_{x}\left( \left(  y+S_{\nu _{n}}  \right)  ;\tau _{y}>\nu_{n},y+S_{\nu _{n}}>\theta _{n}\sqrt{n},\nu _{n}\leq n^{1-\varepsilon }\right) }
{ \mathbb{P}_{x}\left( \tau_{y}>n\right) }.
\end{eqnarray*}%
Taking into account the convergence rate of order $n^{-2 \varepsilon}$ in Lemma \ref{Lemma BB2} 
and Theorem \ref{Th asympt tau}, we get
\begin{equation}
\lim_{n\rightarrow +\infty }T_{n,2}=0.  \label{WK202}
\end{equation}%
By Lemma \ref{Lemma 2} and Theorem \ref{Th asympt tau}, as $n\rightarrow +\infty ,$ 
\begin{equation}
T_{n,3}\leq \frac{\mathbb{P}_{x}\left( \nu _{n}>n^{1-\varepsilon }\right) }{%
\mathbb{P}_{x}\left( \tau _{y}>n\right) }=\frac{O\left( \exp \left(
-cn^{\varepsilon }\right) \right) }{\mathbb{P}_{x}\left( \tau _{y}>n\right) }%
\rightarrow 0.  \label{WK203}
\end{equation}%
Let us now control the term $T_{n, 1}$ which gives the main contribution.
For the sake of brevity set $H_{m}\left( x,y\right) =\mathbb{P}_{x}\left( \tau _{y}>m,y+S_{m}\leq 
\sqrt{n}t\right);$  by the Markov property, one gets
\begin{eqnarray*}
R_{n,1}&:=&\mathbb{P}_{x}\left(y+S_{n}\leq \sqrt{n}t,  \tau _{y}>n, y+S_{\nu _{n}} \leq \theta _{n}\sqrt{n},\nu _{n}\leq n^{1-\varepsilon
}\right) \\
&=&\mathbb{E}_{x}\left( H_{n-\nu _{n}}\left( X_{\nu _{n}},y+S_{\nu_{n}}\right) ;\tau_y>\nu_n, y+S_{\nu _{n}} \leq \theta _{n}\sqrt{n}%
,\nu _{n}\leq n^{1-\varepsilon }\right) \\
&=&\sum_{k=1}^{[n^{1-\varepsilon }]}\mathbb{E}_{x}
\left( H_{n-k}\left(X_{k},y+S_{k}\right) ; \tau_y>k, y+S_{k} \leq \theta _{n}\sqrt{n}, \nu _{n}=k\right).
\end{eqnarray*}
In order to obtain a first order asymptotic we are going to expand $H_{n-k}\left(X_{k},y+S_{k}\right)$
using Lemma \ref{Lemma Weak Conv} with $n, x, y$ replaced by $m_n=n-k, X_k$ and $y+S_{k}$ respectively.
We verify that on the event $A_k=\left\{ \tau_y>k, y+S_{k} \leq \theta _{n}\sqrt{n}, \nu _{n}=k \right\}$ the conditions of  Lemma \ref{Lemma Weak Conv} are satisfied. 
For this, note that on the event $A_k$ one has the lower bound 
$y+S_k=\vert y+S_k\vert \geq \vert y+M_k\vert -a \geq n^{1/2-\epsilon} \geq m_n^{1/2-\epsilon} $ 
and the upper bound
$y+S_{k} \leq \theta'_{m_n} (m_n)^{1/2},$ where
$\theta'_{m_n} = \theta_n \left( n/m_n \right)^{1/2}$ is such that  
$\theta'_{m_n} m_n^{ \varepsilon/4} =  
\theta_n n^{ \varepsilon/4} \left( n/m_n \right)^{1/2-\varepsilon/4} \rightarrow +\infty,$ uniformly in $1 \leq k \leq n^{1- \varepsilon}.$ 
By Lemma \ref{Lemma Weak Conv}, on the event $A_n$
one gets $\mathbb P_x$-a.s.,  as $n\rightarrow +\infty ,$ 
\begin{equation*}
H_{n-k}\left( X_{k},y+S_{k}\right) = \frac{2(y+S_k)}{\sqrt{2\pi (n-k)}\sigma ^{3}}%
\int_{0}^{t}u\exp \left( -u^{2}/2\sigma ^{2}\right) du
\left( 1+o\left( 1 \right) \right),
\end{equation*}  %
which yields
\begin{eqnarray*}
R_{n,1} 
&=&\sum_{k=1}^{[n^{1-\varepsilon }]}\mathbb{E}_{x} 
\left(   2 \ \frac{y+S_{k}}{\sqrt{2\pi n}\sigma^3}\int_{0}^{t}u\exp \left( -u^{2}/2\sigma ^{2}\right) du;  \right. \\
&&  \quad\quad  
\left. \phantom{\int_{0}^{t}}  \tau_y>k, y+S_{k} \leq \theta _{n}\sqrt{n},\nu _{n}=k \right) 
\left( 1+o\left(1\right) \right)   \\
&=&\frac{2 R_{n,2} }{\sqrt{2\pi n}\sigma^3}\int_{0}^{t}u\exp \left( -u^{2}/2\sigma^2\right) du   \left( 1+o\left(1\right) \right),
\end{eqnarray*}%
where
$$
R_{n,2} = 
\mathbb{E}_{x}\left( y+S_{\nu _{n}} ;  \tau_y>\nu_n,  y+S_{\nu _{n}} \leq \theta _{n}\sqrt{n},\nu _{n}\leq n^{1-\varepsilon }\right).
$$
By  Lemmas \ref{Lemma BB-1} and \ref{Lemma BB2}, it follows, as $n\rightarrow +\infty ,$%
\begin{eqnarray*}
R_{n,2}
&=& \mathbb{E}_{x}\left( y+S_{\nu _{n}} ;  \tau_y>\nu_n,  \nu _{n}\leq n^{1-\varepsilon }\right)\\
&& + \quad \mathbb{E}_{x}\left( y+S_{\nu _{n}} ;  \tau_y>\nu_n,  y+S_{\nu _{n}} > \theta _{n}\sqrt{n},
\nu _{n}\leq n^{1-\varepsilon }\right)    \\
&=& V\left( x,y \right) \left( 1+o\left( 1 \right) \right).
\end{eqnarray*}
Implementing the obtained expansion in the expression for $R_{n,1}$, it follows, as $n\rightarrow +\infty ,$%
\begin{eqnarray*}
R_{n,1}
&=&\frac{2V\left( x,y\right) }{\sqrt{2\pi n}\sigma ^{3}}\int_{0}^{t}u\exp \left( -u^{2}/2\sigma^2 \right) du\left( 1+o\left( 1\right) \right) .
\end{eqnarray*}%
Using Theorem \ref{Th asympt tau}, this yields%
\begin{eqnarray}
T_{n,1} = \frac{R_{n,1}}{\mathbb{P}_{x}\left( \tau _{y}>n\right)}
&=&\frac{\frac{2V\left( x,y\right) }{\sqrt{2\pi n}\sigma ^{3}}%
\int_{0}^{t}u\exp \left( -\frac{u^{2}}{2\sigma ^{2}}\right) du}
{\mathbb{P}_{x}\left( \tau _{y}>n\right) }\left( 1+o\left( 1\right) \right)  \notag \\
&=&\frac{1}{\sigma ^{2}}\int_{0}^{t}u\exp \left( -\frac{u^{2}}{2\sigma ^{2}}%
\right) du\left( 1+o\left( 1\right) \right) ,  \label{WK204}
\end{eqnarray}%
as $n\rightarrow +\infty .$ Combining (\ref{WK201}), (\ref{WK202}), (\ref%
{WK203}) and  (\ref{WK204}), we get%
\begin{eqnarray*}
\lim_{n\rightarrow +\infty }\frac{\mathbb{P}_{x}\left( \frac{y+S_{n}}{\sqrt{n}%
}\leq t,\tau _{y}>n\right) }{\mathbb{P}_{x}\left( \tau _{y}>n\right) } &=&%
\frac{1}{\sigma ^{2}}\int_{0}^{t}u\exp \left( -\frac{u^{2}}{2\sigma ^{2}}%
\right) du \\
&=&1-\exp \left( -\frac{u^{2}}{2\sigma ^{2}}\right) ,
\end{eqnarray*}%
which ends the proof of Theorem \ref{Th weak conv cond positive}.

\section{Appendix} 
\subsection{Auxiliary results}

For any $g\in \mathbb{G}$ and $\overline{v}\in \mathbb{P}\left( \mathbb{V}%
\right) $ denote for brevity%
\begin{equation*}
\rho _{1}\left( g,\overline{v}\right) =e^{\rho \left( g,\overline{v}\right)
}=\frac{\left\Vert gv\right\Vert }{\left\Vert v\right\Vert }.
\end{equation*}
Also note that 
\begin{equation}
\frac{1}{\sqrt{2}}\left\Vert u - v \right\Vert \leq
d\left( \overline{u},\overline{v}\right) \leq \left\Vert u-v\right\Vert .  \label{bounds dist d}
\end{equation}%

\begin{lemma}
\label{LemmaBBB1}Let $\eta _{0}>0.$ There exists a constant $c_{\eta _{0}}$
such that for any $\left\vert t\right\vert \leq \eta _{0},$ $g\in \mathbb{G}$
it holds%
\begin{equation*}
\sup_{\overline{u},\overline{v}\in \mathbb{P}\left( \mathbb{V}\right) ,\  \overline{u}\neq \overline{v}}
\frac{\left\vert \rho _{1}\left( g,\overline{u}\right) ^{it}-\rho _{1}\left( g,\overline{v}\right) ^{it}\right\vert }
{d\left( \overline{u},\overline{v}\right) N\left( g\right) ^{4}}\leq c_{\eta_{0}}.
\end{equation*}
\end{lemma}

\begin{proof}
Let $\overline{u}$ and $\overline{v}$ be two elements of $\mathbb{P}\left( 
\mathbb{V}\right) $ with corresponding vectors $u$ and $v$ in $\mathbb{S}%
^{d-1}$ such that $\left\Vert u-v\right\Vert \leq \sqrt{2}.$ Denote $h=u-v.$
Then%
\begin{eqnarray}
&&\left\vert \rho _{1}^{it}\left( g,\overline{u}\right) -\rho
_{1}^{it}\left( g,\overline{v}\right) \right\vert  \notag \\
&=&\left\vert \left( \left\Vert gu\right\Vert ^{2}\right) ^{it/2}-\left(
\left\Vert gv\right\Vert ^{2}\right) ^{it/2}\right\vert  \notag \\
&=&\left\vert 1-\left( \frac{\left\Vert gv\right\Vert ^{2}}{\left\Vert
gu\right\Vert ^{2}}\right) ^{it/2}\right\vert  \notag \\
&=&\left\vert 1-\left( 1+\xi _{1}\right) ^{it/2}\right\vert ,  \label{BBB01}
\end{eqnarray}%
where $\xi _{1}=\frac{\left\Vert gh\right\Vert ^{2}+2\left\langle
gu,gh\right\rangle }{\left\Vert gu\right\Vert ^{2}}$ and%
\begin{eqnarray}
\left\vert \xi _{1}\right\vert &\leq &N\left( g\right) ^{2}\left\Vert
gh\right\Vert \left\{ \left\Vert gh\right\Vert +2\left\Vert gu\right\Vert
\right\}  \notag \\
&\leq &N\left( g\right) ^{3}\left\Vert h\right\Vert \left\{ \left\Vert
gh\right\Vert +2\left\Vert gu\right\Vert \right\}  \notag \\
&\leq &N\left( g\right) ^{4}\left\Vert h\right\Vert \left\{ \sqrt{2}%
+2\right\} :=\xi _{2}.  \label{BBB02}
\end{eqnarray}

We consider two cases.

1) Assume that $\xi _{2}\leq 1/2$ and $\left\vert t\right\vert \leq \eta
_{0}.$ Then%
\begin{eqnarray}
\left\vert 1-\left( 1+\xi _{1}\right) ^{it/2}\right\vert &=&2\left\vert \sin
\left( \frac{t}{2}\ln \left( 1+\xi _{1}\right) \right) \right\vert  \notag \\
&\leq &\eta _{0}\left\vert \ln \left( 1+\xi _{1}\right) \right\vert  \notag
\\
&\leq &c_{\eta _{0}}\xi _{2},  \label{BBB03}
\end{eqnarray}%
with constant $c_{\eta _{0}}$ depending only on $\eta _{0}.$ Taking into
account (\ref{bounds dist d}), from (\ref{BBB01}), (\ref{BBB02}) and (\ref%
{BBB03}), it follows%
\begin{eqnarray*}
\left\vert \rho _{1}\left( g,\overline{u}\right) ^{it}-\rho _{1}\left( g,%
\overline{v}\right) ^{it}\right\vert &\leq &c_{\eta _{0}}N\left( g\right)
^{4}\left\{ \sqrt{2}+2\right\} \left\Vert h\right\Vert \\
&\leq &c_{\eta _{0}}N\left( g\right) ^{4}\left\{ \sqrt{2}+2\right\} \sqrt{2}%
d\left( \overline{u},\overline{v}\right) .
\end{eqnarray*}%
The last bound implies the assertion of the lemma in the case $\xi _{2}\leq
1/2.$

2) Assume that $\left\vert \xi \right\vert >1/2$ and $\left\vert
t\right\vert \leq \eta _{0}.$ Then%
\begin{equation*}
\left\Vert h\right\Vert \geq \frac{1}{2\left( \sqrt{2}+2\right) N\left(
g\right) ^{4}},
\end{equation*}%
and, by (\ref{bounds dist d}),%
\begin{equation}
d\left( \overline{u},\overline{v}\right) \geq \frac{\left\Vert h\right\Vert 
}{\sqrt{2}}\geq \frac{1}{2\sqrt{2}\left( \sqrt{2}+2\right) N\left( g\right)
^{4}}.  \label{BBB05}
\end{equation}%
From (\ref{BBB05}) we conclude that%
\begin{equation*}
\frac{\left\vert \rho _{1}\left( g,\overline{u}\right) ^{it}-\rho _{1}\left(
g,\overline{v}\right) ^{it}\right\vert }{d\left( \overline{u},\overline{v}%
\right) }\leq 4\sqrt{2}\left( \sqrt{2}+2\right) N\left( g\right) ^{4},
\end{equation*}%
which again proves the assertion of the lemma.
\end{proof}

\begin{corollary}
\label{CorollaryBBB2}Let $\epsilon >0$ and $\eta _{0}>0.$ There exists a
constant $c_{\eta _{0}}$ such that for any $\left\vert t\right\vert \leq
\eta _{0},$ $g\in \mathbb{G}$ and any $\overline{u},\overline{v}\in \mathbb{P%
}\left( \mathbb{V}\right) $ it holds%
\begin{equation*}
\left\vert \rho _{1}\left( g,\overline{u}\right) ^{it}-\rho _{1}\left( g,%
\overline{v}\right) ^{it}\right\vert \leq 2^{1-\epsilon }c_{\eta
_{0}}d\left( \overline{u},\overline{v}\right) ^{\varepsilon }N\left(
g\right) ^{4\varepsilon }.
\end{equation*}
\end{corollary}

\begin{proof}
Since $\left\vert \rho _{1}^{it}\left( g,\overline{u}\right) -\rho
_{1}^{it}\left( g,\overline{v}\right) \right\vert \leq 2$ the assertion
follows from Lemma \ref{LemmaBBB1}:%
\begin{eqnarray*}
\left\vert \rho _{1}\left( g,\overline{u}\right) ^{it}-\rho _{1}\left( g,%
\overline{v}\right) ^{it}\right\vert &\leq &2^{1-\epsilon }\left\vert \rho
_{1}\left( g,\overline{u}\right) ^{it}-\rho _{1}\left( g,\overline{v}\right)
^{it}\right\vert ^{\epsilon } \\
&&2^{1-\epsilon }c_{\eta _{0}}^{\epsilon }d\left( \overline{u},\overline{v}%
\right) ^{\epsilon }N\left( g\right) ^{4\epsilon }.
\end{eqnarray*}
\end{proof}

\begin{lemma}
\label{LemmaBBB3}Let $\eta _{0}>0.$ There exists a constant $c_{\eta _{0}}$
such that for any $\left\vert t\right\vert \leq \eta _{0}$ and  $\overline{v}\in 
\mathbb{P}\left( \mathbb{V}\right) $ it
holds%
\begin{equation*}
\sup_{g, g^{\prime }\in \mathbb{G}, \  g\neq g^{\prime } }
\frac{\left\vert \rho _{1}\left( g,\overline{v}\right) ^{it}-\rho _{1}\left(g^{\prime },\overline{v}\right) ^{it}\right\vert }{\left\Vert g-g^{\prime
}\right\Vert \left( 2+\left\Vert g-g^{\prime }\right\Vert \right) N\left(
g\right) ^{3}N\left( g^{\prime }\right) ^{3}}\leq c_{\eta _{0}}.
\end{equation*}
\end{lemma}

\begin{proof}
For any $\overline{v}\in \mathbb{P}\left( \mathbb{V}\right) $ and $%
g,g^{\prime }\in \mathbb{G},$%
\begin{equation*}
\left\vert \rho _{1}^{it}\left( g,\overline{v}\right) -\rho _{1}^{it}\left(
g^{\prime },\overline{v}\right) \right\vert =\left\vert \left\vert \rho
_{1}\left( g,\overline{v}\right) \right\vert ^{it}-\left\vert \rho
_{1}\left( g^{\prime },\overline{v}\right) \right\vert ^{it}\right\vert .
\end{equation*}%
It is clear that $g^{\prime }v=gv+gh,$ where $h=\left( g^{-1}g^{\prime
}-I\right) v.$ As in Lemma \ref{LemmaBBB1} (\ref{BBB02}) we obtain%
\begin{equation}
\left\vert \rho _{1}\left( g,\overline{v}\right) ^{it}-\rho _{1}\left(
g^{\prime },\overline{v}\right) ^{it}\right\vert =\left\vert 1-\left( 1+\xi
_{1}\right) ^{it/2}\right\vert ,  \notag
\end{equation}%
where $\xi _{1}=\frac{\left\Vert gh\right\Vert ^{2}+2\left\langle
gv,gh\right\rangle }{\left\Vert gv\right\Vert ^{2}}$ and%
\begin{equation*}
\left\vert \xi _{1}\right\vert \leq N\left( g\right) ^{3}\left\Vert
h\right\Vert \left\{ \left\Vert gh\right\Vert +2\left\Vert gv\right\Vert
\right\} :=\xi _{2}.
\end{equation*}%
Taking into account that%
\begin{eqnarray*}
\left\Vert h\right\Vert &=&\left\Vert \left( g^{-1}g^{\prime }-I\right)
v\right\Vert \\
&=&\left\Vert \left( g^{-1}\left( g^{\prime }-g\right) \right) v\right\Vert
\\
&\leq &N\left( g\right) \left\Vert g-g^{\prime }\right\Vert ,
\end{eqnarray*}%
we get%
\begin{eqnarray*}
\left\vert \xi _{1}\right\vert &\leq &N\left( g\right) ^{4}\left\Vert
g-g^{\prime }\right\Vert \left\{ \left\Vert gh\right\Vert +2\left\Vert
gv\right\Vert \right\} \\
&\leq &N\left( g\right) ^{4}\left\Vert g-g^{\prime }\right\Vert \left\{
N\left( g\right) ^{2}\left\Vert g-g^{\prime }\right\Vert +2N\left( g\right)
\right\} \\
&=&N\left( g\right) ^{5}\left\Vert g-g^{\prime }\right\Vert \left\{ N\left(
g\right) \left\Vert g-g^{\prime }\right\Vert +2\right\} .
\end{eqnarray*}%
Since $N\left( g\right) \geq 1,$ 
\begin{equation*}
\left\vert \rho _{1}\left( g,\overline{v}\right) ^{it}-\rho _{1}\left(
g^{\prime },\overline{v}\right) ^{it}\right\vert \leq N\left( g\right)
^{6}\left\Vert g-g^{\prime }\right\Vert \left\{ \left\Vert g-g^{\prime
}\right\Vert +2\right\} .
\end{equation*}%
Symmetrizing this result with respect to $g$ and $g^{\prime }$ we get the
assertion of the Lemma.
\end{proof}

\begin{corollary}
\label{CorollaryBBB4}Let $0<\epsilon <1$ and $\eta _{0}>0.$ There exists a
constant $c_{\eta _{0}}$ such that for any $\left\vert t\right\vert \leq
\eta _{0},$ $\overline{u}\in \mathbb{P}\left( \mathbb{V}\right) $ and $%
g,g^{\prime }\in \mathbb{G}$ it holds:

\noindent 1. 
\begin{equation*}
\left\vert \rho _{1}\left( g,\overline{v}\right) ^{it}-\rho _{1}\left(
g^{\prime },\overline{v}\right) ^{it}\right\vert \leq c_{\eta
_{0}}\left\Vert g-g^{\prime }\right\Vert N\left( g\right) ^{3}N\left(
g^{\prime }\right) ^{3}.
\end{equation*}

\noindent 2. 
\begin{equation*}
\left\vert \rho _{1}\left( g,\overline{v}\right) ^{it}-\rho _{1}\left(
g^{\prime },\overline{v}\right) ^{it}\right\vert \leq 2^{1-\varepsilon
}c_{\eta _{0}}^{\epsilon }\left\Vert g-g^{\prime }\right\Vert ^{\epsilon
}N\left( g\right) ^{3\epsilon }N\left( g^{\prime }\right) ^{3\epsilon }.
\end{equation*}
\end{corollary}

\begin{proof}
The proof is similar to the proof of Lemma \ref{LemmaBBB1} and Corollary \ref%
{CorollaryBBB2}. Indication: consider two cases $\left\Vert g-g^{\prime
}\right\Vert \leq 1$ and $\left\Vert g-g^{\prime }\right\Vert >1.$
\end{proof}

\begin{lemma}
\label{lemma rho smoo}Let $\overline{\rho }=\mathbf{P}\rho .$ Then $%
\overline{\rho }\in \mathcal{B}.$
\end{lemma}

\begin{proof}
Let $g,h\in G$ and $v_{g}$ ($u_{h}$) be any vector in $\mathbb{V}$ of
direction $g\cdot \overline{v}$ ($h\cdot \overline{u}$). Since%
\begin{equation*}
\left\vert \rho \left( g_{1},g\cdot \overline{v}\right) \right\vert \leq
\left\vert \ln \frac{\left\Vert g_{1}v_{g}\right\Vert }{\left\Vert
v_{g}\right\Vert }\right\vert \leq \ln N\left( g_{1}\right) ,
\end{equation*}%
we have%
\begin{equation}
\left\Vert \overline{\rho }\right\Vert _{\infty }\leq \int_{\mathbb{G}%
}\left\vert \rho \left( g_{1},g\cdot \overline{v}\right) \right\vert 
\boldsymbol{\mu }\left( dg_{1}\right) \leq \int_{\mathbb{G}}\ln N\left(
g_{1}\right) \boldsymbol{\mu }\left( dg_{1}\right) < +\infty .  \label{rho bb}
\end{equation}%
Moreover%
\begin{eqnarray*}
\left\vert \rho \left( g_{1},g\cdot \overline{v}\right) -\rho \left(
g_{1},h\cdot \overline{u}\right) \right\vert &\leq &2^{1-\varepsilon }\ln
^{1-\varepsilon }N\left( g_{1}\right) \left\vert \ln \frac{\left\Vert
g_{1}v_{g}\right\Vert }{\left\Vert v_{g}\right\Vert }-\ln \frac{\left\Vert
g_{1}u_{h}\right\Vert }{\left\Vert u_{h}\right\Vert }\right\vert
^{\varepsilon } \\
&\leq &2^{1-\varepsilon }\ln ^{1-\varepsilon }N\left( g_{1}\right) \frac{%
\left\vert \left\Vert g_{1}v_{g}\right\Vert -\left\Vert
g_{1}u_{h}\right\Vert \right\vert ^{\varepsilon }}{\left\Vert
g_{1}u_{h}\right\Vert ^{\varepsilon }} \\
&\leq &2^{1-\varepsilon }\ln ^{1-\varepsilon }N\left( g_{1}\right) \frac{%
\left\Vert g_{1}\left( v_{g}-u_{h}\right) \right\Vert ^{\varepsilon }}{%
\left\Vert g_{1}u_{h}\right\Vert ^{\varepsilon }} \\
&\leq &2^{1-\varepsilon }\ln ^{1-\varepsilon }N\left( g_{1}\right)
\left\Vert g_{1}\right\Vert ^{\varepsilon }\left\Vert g_{1}^{-1}\right\Vert
^{\varepsilon }\left\Vert v_{g}-u_{h}\right\Vert ^{\varepsilon }.
\end{eqnarray*}%
Therefore 
\begin{equation*}
\left\vert \overline{\rho }\left( g,\overline{v}\right) -\overline{\rho }%
\left( h,\overline{u}\right) \right\vert \leq 2^{1-\varepsilon }\left\Vert
v_{g}-u_{h}\right\Vert ^{\varepsilon }\int \left\Vert g_{1}\right\Vert
^{\varepsilon }\left\Vert g_{1}^{-1}\right\Vert ^{\varepsilon }\ln
^{1-\varepsilon }N\left( g_{1}\right) \mu \left( dg_{1}\right) .
\end{equation*}%
In particular, with $h=g,$ we get 
\begin{equation*}
\left\vert \overline{\rho }\left( g,\overline{v}\right) -\overline{\rho }%
\left( g,\overline{u}\right) \right\vert \leq c_{\varepsilon }\left\Vert
g\right\Vert ^{\varepsilon }\left\Vert v-u\right\Vert ^{\varepsilon }\leq
c_{\varepsilon }N\left( g\right) ^{\varepsilon }\left\Vert v-u\right\Vert
^{\varepsilon }
\end{equation*}%
and with $u=v,$ 
\begin{equation*}
\left\vert \overline{\rho }\left( g,\overline{v}\right) -\overline{\rho }%
\left( h,\overline{v}\right) \right\vert \leq c_{\varepsilon }\left\Vert
g-h\right\Vert ^{\varepsilon }.
\end{equation*}%
The last two bounds imply that $k_{\varepsilon }\left( \overline{\rho }%
\right) < +\infty .$ Together with (\ref{rho bb}) this proves that $\left\Vert 
\overline{\rho }\right\Vert _{\mathcal{B}}< +\infty .$
\end{proof}

The following assertions are proved in Le Page \cite{LePage82} (see
respectively Theorem 1, p.262 and Corollary 1, p.269).

\begin{proposition}
\label{key prop}Assume conditions \textbf{P1-P3}. There exist $\varepsilon
>0 $ and $\rho _{\varepsilon }\in \left( 0,1\right) $ such that 
\begin{equation*}
\lim_{n\rightarrow +\infty }\sup_{\overline{v}_{1}\neq \overline{v}_{2}}%
\mathbf{Pr}\left( \frac{d\left( G_{n}  \cdot \overline{v}_{1},G_{n}  \cdot \overline{v}%
_{2}\right) ^{\varepsilon }}{d\left( \overline{v}_{1},\overline{v}%
_{2}\right) ^{\varepsilon }}\right) ^{1/n}=\rho _{\varepsilon }.
\end{equation*}
\end{proposition}

\begin{proposition}
Assume conditions \textbf{P1-P3}. For any continuous function $\varphi :%
\mathbb{P}\left( \mathbb{V}\right) \rightarrow \mathbb{R}$ it holds%
\begin{equation}
\lim_{n\rightarrow +\infty }\sup_{\overline{v}\in \mathbb{P}\left( \mathbb{V}\right) }
\left\vert \mathbf{Pr}\varphi \left( G_{n} \cdot \overline{v}%
\right) -\boldsymbol{\nu }\left( \varphi \right) \right\vert =0,
\label{ST001}
\end{equation}%
where $\boldsymbol{\nu }$ is the invariant measure defined by (\ref{st mes
proj sp}).
\end{proposition}

\subsection{Existence of the stationary probability\label{Sec Exist Inv Prob}}

The fact that the measure $\boldsymbol{\nu }$ is $\boldsymbol{\mu }$%
-invariant implies that the probability measure $\boldsymbol{\lambda }$
defined by (\ref{invar mes X}) is stationary for the Markov chain $\left(
X_{n}\right) _{n\geq 1}.$ Indeed, from (\ref{st mes proj sp}) and (\ref%
{trans prob X}) it follows that for any bounded measurable function $\varphi 
$ on $\mathbb{G}\times \mathbb{P}\left( \mathbb{V}\right) ,$%
\begin{eqnarray*}
\boldsymbol{\lambda }\left( \mathbf{P}\varphi \right) &=&\int_{\mathbb{G}%
}\int_{\mathbb{P}\left( \mathbb{V}\right) }\left( \int_{\mathbb{G}}\varphi
\left( g_{1},g\cdot \overline{v}\right) \boldsymbol{\mu }\left(
dg_{1}\right) \right) \boldsymbol{\mu }\left( dg\right) \boldsymbol{\nu }%
\left( d\overline{v}\right) \\
&=&\int_{\mathbb{G}}\left( \int_{\mathbb{G}}\int_{\mathbb{P}\left( \mathbb{V}%
\right) }\varphi \left( g_{1},g\cdot \overline{v}\right) \boldsymbol{\nu }%
\left( d\overline{v}\right) \boldsymbol{\mu }\left( dg\right) \right) 
\boldsymbol{\mu }\left( dg_{1}\right) \\
&=&\int_{\mathbb{G}}\left( \int_{\mathbb{P}\left( \mathbb{V}\right) }\varphi
\left( g_{1},\overline{v}\right) \boldsymbol{\nu }\left( d\overline{v}%
\right) \right) \boldsymbol{\mu }\left( dg_{1}\right) \\
&=&\boldsymbol{\lambda }\left( \varphi \right) .
\end{eqnarray*}%
It remains to prove its unicity.

For any $f\in \mathcal{B}$ define $\overline{f}:\mathbb{P}\left( \mathbb{V}%
\right) \rightarrow \mathbb{R}$ by $\overline{f}\left( \overline{v}\right)
=\int_{\mathbb{G}}f\left( g,\overline{v}\right) \boldsymbol{\mu }\left(
dg\right) .$ Since $f\in \mathcal{B}$ it follows that $\overline{f}$ is $%
\varepsilon $-H\"{o}lder on $\mathbb{P}\left( \mathbb{V}\right) .$ Indeed,
we have $k_{\varepsilon }\left( f\right) < +\infty ,$ so that 
\begin{eqnarray*}
\left\vert \overline{f}\left( \overline{u}\right) -\overline{f}\left( 
\overline{v}\right) \right\vert &\leq &\int \left\vert f\left( g,\overline{u}%
\right) -f\left( g,\overline{v}\right) \right\vert \boldsymbol{\mu }\left(
dg\right) \\
&\leq &k_{\varepsilon }\left( f\right) d\left( \overline{u},\overline{v}%
\right) ^{\varepsilon }\int N\left( g\right) ^{4\varepsilon }\boldsymbol{\mu 
}\left( dg\right) .
\end{eqnarray*}%
Using the independence of $g_{1},\dots,g_{n},$ 
\begin{equation}
\mathbf{P}^{n}f\left( g,\overline{v}\right) =\mathbf{Pr}f\left(
g_{n},G_{n-1}g\cdot \overline{v}\right) =\mathbf{Pr}\overline{f}\left(
G_{n-1}g\cdot \overline{v}\right) .  \label{ST002}
\end{equation}%
From (\ref{ST002}) and (\ref{ST001}) with $\varphi =\overline{f}$ it follows
that%
\begin{equation}
\lim_{n\rightarrow +\infty }\mathbf{P}^{n}f\left( g,\overline{v}\right) =%
\boldsymbol{\nu }\left( \overline{f}\right) =\boldsymbol{\lambda }\left(
f\right) .  \label{ST003}
\end{equation}%
This proves that $\boldsymbol{\lambda }$ is the unique stationary measure.

\subsection{The theorem of Ionescu-Tulcea and Marinescu\label{Sec I T M}}

Recall that $\mathcal{B}\subset \mathcal{C}_{b}$ and that $\mathcal{B}$ and $%
\mathcal{C}_{b}$ are endowed with the norms $\left\Vert \cdot \right\Vert _{%
\mathcal{B}}$ and $\left\Vert \cdot \right\Vert _{\infty }$ respectively.
Define $\left\Vert \mathbf{P}_{t}\right\Vert _{\infty}=\sup_{f\in 
\mathcal{B}}\frac{\left\Vert \mathbf{P}_{t}f\right\Vert _{\infty }}{%
\left\Vert f\right\Vert _{\infty }}.$ 
Consider the following conditions (see Norman \cite{Norman72}, Section 3.2):

(a) If $f_{n}\in \mathcal{B},$ 
$\left\Vert f_{n}\right\Vert _{\mathcal{B}}\leq c,$ $n\geq 1,$ 
$f\in \mathcal{C}_{b},$ 
and
lim$_{n\rightarrow +\infty }\left\Vert f_{n}-f\right\Vert _{\infty}=0,$ 
then $f\in \mathcal{B}$ and $%
\left\Vert f\right\Vert_{\mathcal{B}} \leq c.$

(b) There exist $\eta _{0}>0$ such that $\sup_{\left\vert t\right\vert \leq
\eta _{0}}\left\Vert \mathbf{P}_{t}\right\Vert _{\infty}\leq c.$

(c) There exist $\eta _{0}>0,$ $n_{0}\geq 1$ and a constant $0<r_{0}<1$ such
that for any function $f\in \mathcal{B}$ it holds 
\begin{equation*}
\sup_{\left\vert t\right\vert <\eta _{0}}\left\Vert \mathbf{P}%
_{t}^{n_{0}}f\right\Vert _{\mathcal{B}}<r_{0}\left\Vert f\right\Vert _{%
\mathcal{B}}+\left\Vert f\right\Vert _{\infty }.
\end{equation*}

(d) For any $t\in \lbrack -\eta _{0},\eta _{0}]$ the set $\mathbf{P}_{t}B$
has compact closure in $\left( \mathcal{C}_{b},\left\Vert \cdot \right\Vert
_{\infty }\right) $ for any bounded subset $B$ of $\mathcal{B}.$

The assertion (a) is obvious.

The following lemma shows that the family of operators $\mathbf{P}_{t}^{n},$ 
$t\in \lbrack -\eta _{0},\eta _{0}]$ is a contraction on $\mathcal{B}$
uniformly in $t,$ which implies assertion (b).

\begin{lemma}
\label{Lemma Contract}For any $\left\vert t\right\vert \leq \eta _{0}$ and $%
f\in \mathcal{B},$ 
\begin{equation*}
\left\Vert \mathbf{P}_{t}f\right\Vert _{\infty }\leq \left\Vert f\right\Vert
_{\infty }.
\end{equation*}
\end{lemma}

\begin{proof}
Note that%
\begin{equation*}
\mathbf{P}_{t}f\left( g,\overline{v}\right) =\mathbf{Pr}\left( e^{it\rho
\left( g_{1},g\cdot \overline{v}\right) }f\left( g_{1},g\cdot \overline{v}%
\right) \right) .
\end{equation*}%
The assertion of the lemma follows from the last identity.
\end{proof}

We prove next that the family of operators $\left( \mathbf{P}_{t}\right)
_{\left\vert t\right\vert <\eta _{0}}$ satisfies the uniform Doeblin-Fortet
property (c). This follows from the following lemma by choosing $n$
sufficiently large.

\begin{lemma}
\label{PropDoeblFortet}There exist constants $\eta _{0}>0$ and $\rho
_{\varepsilon }\in (0,1)$ such that for any $n\geq 1$ 
\begin{equation*}
\sup_{\left\vert t\right\vert <\eta _{0}}\left\Vert \mathbf{P}%
_{t}^{n}f\right\Vert _{\mathcal{B}}\leq \left( 1+c_{\varepsilon }\rho
_{\varepsilon }^{n}\right) \left\Vert f\right\Vert _{\infty }\leq \left\Vert
f\right\Vert _{\infty }+c_{\varepsilon }\rho _{\varepsilon }^{n}\left\Vert
f\right\Vert _{\mathcal{B}}.
\end{equation*}
\end{lemma}

\begin{proof}
Using Corollary \ref{CorollaryBBB2} and (\ref{module cont}),%
\begin{eqnarray}
&&\left\Vert \mathbf{P}_{t}^{n}f\left( g,\overline{u}\right) -\mathbf{P}%
_{t}^{n}f\left( g,\overline{v}\right) \right\Vert  \notag \\
&\leq &\left\Vert f\right\Vert _{\infty }\mathbf{Pr}\left( \left\vert \rho_1^{it}\left( g_{n},G_{n-1}g\cdot \overline{u}\right) -\rho ^{it}\left(
g_{n},G_{n-1}g  \cdot \overline{v}\right) \right\vert \right)  \notag \\
&&+\mathbf{Pr}\left( \left\vert f\left( g_{n},G_{n-1}g\cdot\overline{u}\right)
-f\left( g_{n},G_{n-1}g  \cdot \overline{v}\right) \right\vert \right)  \notag \\
&\leq &\left\Vert f\right\Vert _{\infty }2^{1-\varepsilon }c_{\eta_{0}}^{\varepsilon }\mathbf{Pr}\left( N\left( g\right) ^{4\varepsilon
}d\left( G_{n-1}g  \cdot \overline{u},G_{n-1}g  \cdot \overline{v}\right) ^{\varepsilon
}\right)  \notag \\
&&+k_{\varepsilon }\left( f\right) \mathbf{Pr}\left( N\left( g\right)
^{4\varepsilon }d\left( G_{n-1}g  \cdot \overline{u},G_{n-1}g  \cdot \overline{v}\right)
^{\varepsilon }\right)  \notag \\
&\leq &c\left\Vert f\right\Vert _{\infty }\mathbf{Pr}\left( d\left( G_{n-1}g  \cdot \overline{u},G_{n-1}g  \cdot \overline{v}\right) ^{\varepsilon }\right) .
\label{DF00a}
\end{eqnarray}%
Note that%
\begin{equation}
\frac{d\left( G_{n-1}g  \cdot \overline{u},G_{n-1}g  \cdot \overline{v}\right) }{d\left( 
\overline{u},\overline{v}\right) }\leq \frac{d\left( G_{n-1}g  \cdot \overline{u},G_{n-1}g  \cdot \overline{v}\right) }{d\left( g  \cdot \overline{u},g  \cdot \overline{v}\right) }%
\frac{d\left( g  \cdot \overline{u},g  \cdot \overline{v}\right) }{d\left( \overline{u},%
\overline{v}\right) }  \label{DF00b}
\end{equation}%
and%
\begin{equation}
\frac{d\left( g  \cdot \overline{u},g  \cdot \overline{v}\right) }{d\left( \overline{u},\overline{v}\right) }
=\frac{\left\Vert g  \cdot \overline{u}\wedge g  \cdot \overline{v} \right\Vert}
{\left\Vert \overline{u}\wedge \overline{v} \right\Vert}\leq N\left( g\right) ^{4}.  \label{DF00c}
\end{equation}%
From (\ref{DF00a}), (\ref{DF00b}), (\ref{DF00c}) and Proposition \ref{key
prop} it follows that for sufficiently small $\varepsilon >0$ and some $\rho
_{\varepsilon }\in \left( 0,1\right) ,$%
\begin{eqnarray}
\frac{\left\Vert \mathbf{P}_{t}^{n}f\left( g,\overline{u}\right) -\mathbf{P}%
_{t}^{n}f\left( g,\overline{v}\right) \right\Vert }{d\left( \overline{u},%
\overline{v}\right) ^{\varepsilon }N\left( g\right) ^{4\varepsilon }} &\leq
&c\left\Vert f\right\Vert _{\infty }\sup_{\overline{v}_{1}\neq \overline{v}%
_{2}}\mathbf{Pr}\left( \frac{d\left( G_{n-1} \cdot \overline{v}_{1},G_{n-1}  \cdot \overline{v}_{2}\right) ^{\varepsilon }}{d\left( \overline{v}_{1},\overline{v%
}_{2}\right) ^{\varepsilon }}\right)  \notag \\
&\leq &c_{\varepsilon }\left\Vert f\right\Vert _{\infty }\rho _{\varepsilon
}^{n}.  \label{DF001}
\end{eqnarray}

On the other hand, in the same way as above, we get%
\begin{eqnarray*}
&&\left\Vert \mathbf{P}_{t}^{n}f\left( g,\overline{v}\right) -\mathbf{P}%
_{t}^{n}f\left( g^{\prime },\overline{v}\right) \right\Vert \\
&\leq &\left\Vert f\right\Vert _{\infty }\mathbf{Pr}\left( \left\vert \rho
^{it}\left( g_{n},G_{n-1}g  \cdot \overline{v}\right) -\rho ^{it}\left(
g_{n},G_{n-1}g^{\prime }  \cdot \overline{v}\right) \right\vert \right) \\
&&+\mathbf{Pr}\left( \left\vert f\left( g_{n},G_{n-1}g  \cdot \overline{v}\right)
-f\left( g_{n},G_{n-1}g^{\prime }  \cdot \overline{v}\right) \right\vert \right) \\
&\leq &\left\Vert f\right\Vert _{\infty }2^{1-\varepsilon }
c_{\eta_{0}}^{\varepsilon }\mathbf{Pr}
\left( N\left( g\right) ^{4\varepsilon}d\left( G_{n-1}g  \cdot \overline{v},G_{n-1}g^{\prime }  \cdot \overline{v}\right)^{\varepsilon }\right) \\
&&+k_{\varepsilon }\left( f\right) \mathbf{Pr}\left( N\left( g\right)
^{4\varepsilon }d\left( G_{n-1}g  \cdot \overline{v},G_{n-1}g^{\prime }  \cdot \overline{v}%
\right) ^{\varepsilon }\right) \\
&\leq &c\left\Vert f\right\Vert _{\infty }\mathbf{Pr}\left( d\left( G_{n-1}g  \cdot \overline{v},
G_{n-1}g^{\prime }  \cdot \overline{v}\right) ^{\varepsilon }\right) \\
&\leq &c\left\Vert f\right\Vert _{\infty }\sup_{\overline{v}_{1}\neq 
\overline{v}_{2}}\mathbf{Pr}\left( \frac{d\left( G_{n-1}\overline{v}%
_{1},G_{n-1}\overline{v}_{2}\right) ^{\varepsilon }}{d\left( \overline{v}%
_{1},\overline{v}_{2}\right) ^{\varepsilon }}\right) d\left( g  \cdot \overline{v}%
,g^{\prime }  \cdot \overline{v}\right) ^{\varepsilon }.
\end{eqnarray*}%
Since $g  \cdot \overline{u}\wedge g  \cdot \overline{v}
=\left( g  \cdot \overline{u}-g  \cdot \overline{v}\right) \wedge g  \cdot \overline{v},$ we have%
\begin{eqnarray*}
d\left( g  \cdot \overline{v},g^{\prime }  \cdot \overline{v}\right) &=&\frac{\left\Vert g  \cdot \overline{v}\wedge g^{\prime }  \cdot \overline{v}\right\Vert }
{\left\Vert g  \cdot \overline{v}\right\Vert \left\Vert g^{\prime }  \cdot \overline{v}\right\Vert } \\
&=&\frac{\left\Vert g-g^{\prime }\right\Vert \inf \left\{ \left\Vert
gv\right\Vert ,\left\Vert g^{\prime }v\right\Vert \right\} }{\left\Vert g  \cdot \overline{v}\right\Vert \left\Vert g^{\prime }  \cdot \overline{v}\right\Vert } \\
&\leq &\left\Vert g-g^{\prime }\right\Vert \frac{1}{\sqrt{\left\Vert g  \cdot \overline{v}\right\Vert \left\Vert g^{\prime }  \cdot \overline{v}\right\Vert }} \\
&\leq &\left\Vert g-g^{\prime }\right\Vert \left( N\left( g\right) N\left(
g^{\prime }\right) \right) ^{1/2}.
\end{eqnarray*}%
Then, for some sufficiently small $\varepsilon >0$ and some $\rho
_{\varepsilon }\in \left( 0,1\right) ,$%
\begin{equation}
\frac{\left\Vert \mathbf{P}_{t}^{n}f\left( g,\overline{v}\right) -\mathbf{P}%
_{t}^{n}f\left( g^{\prime },\overline{v}\right) \right\Vert }{\left\Vert
g-g^{\prime }\right\Vert ^{\varepsilon }\left( N\left( g\right) N\left(
g^{\prime }\right) \right) ^{3\varepsilon }}\leq c_{\varepsilon }\left\Vert
f\right\Vert _{\infty }\rho _{\varepsilon }^{n}.  \label{DF002}
\end{equation}

From (\ref{DF001}) and (\ref{DF002}) we obtain, uniformly in $t\in \left[
-\eta _{0},\eta _{0}\right] ,$%
\begin{equation*}
k_{\varepsilon }\left( \mathbf{P}_{t}^{n}f\right) \leq 2c_{\varepsilon }\rho
_{\varepsilon }^{n}\left\Vert f\right\Vert _{\infty },
\end{equation*}%
which in turn implies%
\begin{equation*}
\left\Vert \mathbf{P}_{t}^{n}f\right\Vert _{\mathcal{B}}\leq \left\Vert
f\right\Vert _{\infty }+2c_{\varepsilon }\rho _{\varepsilon }^{n}\left\Vert
f\right\Vert _{\infty }.
\end{equation*}
\end{proof}

Finally, condition (d) follows from the next lemma.

\begin{lemma}
\label{LemmaRelatComp}For any $\left\vert t\right\vert \leq \eta _{0}$ the
set $\mathbf{P}_{t}B$ is relatively compact in $\left( \mathcal{C}%
_{b},\left\Vert \cdot \right\Vert _{\infty }\right) $ for any bounded subset 
$B$ of $\mathcal{B}.$
\end{lemma}

\begin{proof}
Let $B$ be a bounded subset of $\mathcal{B}$ and $\left( f_{n}\right)
_{n\geq 1}$ be a sequence in $B.$ This sequence is also bounded in the
Banach space $\mathcal{B}$ and therefore $\left( f_{n}\right) _{n\geq 1}$ is
uniformly bounded and uniformly equicontinuous on the compacts of $\mathbb{G}%
\times \mathbb{P}\left( \mathbb{V}\right) .$ By the theorem of Arzela-Ascoli
(see e.g. Dunford and Schwartz \cite{DunSchw58}) we can extract a
subsequence $\left( f_{n_{k}}\right) _{k\geq 1}$ converging uniformly on the
compacts of $\mathbb{G}\times \mathbb{P}\left( \mathbb{V}\right) $ to a
function $f\in \mathcal{B}.$ Then, for any $\left\vert t\right\vert \leq
\eta _{0},$%
\begin{eqnarray*}
\left\Vert \mathbf{P}_{t}f_{n_{k}}-\mathbf{P}_{t}f\right\Vert _{\infty }
&\leq &\sup_{\left\Vert g_{1}\right\Vert \leq A,\ \overline{v}\in \mathbb{P}%
\left( \mathbb{V}\right) }\left\vert f_{n_{k}}\left( g_{1},\overline{v}%
\right) -f\left( g_{1},\overline{v}\right) \right\vert \\
&&+\left( \sup_{n}\left\Vert f_{n}\right\Vert _{\infty }+\left\Vert
f\right\Vert _{\infty }\right) \int_{\left\Vert g_{1}\right\Vert >A}\mu
\left( dg_{1}\right) .
\end{eqnarray*}%
Taking the limit as $k\rightarrow +\infty $ and then as $A\rightarrow +\infty $
we get, for any $\left\vert t\right\vert \leq \eta _{0},$ 
\begin{equation*}
\lim_{k\rightarrow +\infty }\left\Vert \mathbf{P}_{t}f_{n_{k}}-\mathbf{P}%
_{t}f\right\Vert _{\infty }=0,
\end{equation*}%
which shows that the set $\mathbf{P}_{t}B$ is relatively compact.
\end{proof}

\subsection{Existence of the measure satisfying condition \textbf{P5} } \label{sec cond P5}
We prove the existence  of a measure $\boldsymbol{\mu }$ satisfying
conditions \textbf{P1-P5.}
Let $\boldsymbol{\mu }_0$ be a probability measure on $\mathbb{G}$ satisfying conditions 
\textbf{P1-P4} which admits $\boldsymbol{\nu }$ as invariant measure and
whose upper Lyapunov exponent is $0.$ Let $\lambda >1.$ Define the measure 
\begin{equation*}
\boldsymbol{\mu }_{\lambda }\left( dg^{\prime }\right) =\alpha \boldsymbol{%
\delta }_{\lambda I}\left( dg^{\prime }\right) +\left( 1-\alpha \right) 
\boldsymbol{\mu }_0\left( \frac{1}{\lambda }dg^{\prime }\right) ,
\end{equation*}%
where $\alpha \in \left( 0,1\right) .$ Then $\boldsymbol{\mu }_{\lambda }$
satisfies conditions \textbf{P1-P4} and $\boldsymbol{\mu }_{\lambda }\ast 
\boldsymbol{\nu }=\boldsymbol{\nu }$ , i.e. $\boldsymbol{\nu }$ is $%
\boldsymbol{\mu }_{\lambda }$-invariant measure. Moreover, the upper
Lyapunov exponent of $\boldsymbol{\mu }_{\lambda }$ is $0,$ i.e.%
\begin{equation*}
\int_{\mathbb{G}}\int_{\mathbb{P}\left( \mathbb{V}\right) }\rho \left( g,%
\overline{v}\right) \boldsymbol{\mu }_{\lambda }\left( dg\right) \boldsymbol{%
\nu }\left( dv\right) =0
\end{equation*}%
and 
\begin{equation*}
\inf_{v\in \mathbb{S}^{d-1}}\boldsymbol{\mu }_{\lambda }\left( g:\log
\left\Vert gv\right\Vert >\ln \lambda \right) \geq \alpha >0
\end{equation*}%
which means that condition \textbf{P5} is satisfied.

\end{document}